\documentclass[reqno,12pt,letterpaper]{amsart}
\usepackage{amsmath,amssymb,amsthm,graphicx,mathrsfs,url} 
\usepackage[usenames,dvipsnames]{color}
\usepackage[colorlinks=true,linkcolor=Red,citecolor=Green]{hyperref}
\usepackage{amsxtra}
\usepackage{epstopdf}

\def\?[#1]{\textbf{[#1]}\marginpar{\Large{\textbf{??}}}}

\def\smallsection#1{\smallskip\noindent\textbf{#1}.}
\let\epsilon=\varepsilon 

\newcommand{\tbe}{\tilde{\beta}}
\newcommand{\Z}{\mathbb{Z}}
\newcommand{\T}{\mathbb{T}}
\newcommand{\Q}{\mathbb{Q}}
\newcommand{\R}{\mathbb{R}}
\newcommand{\C}{\mathbb{C}}
\newcommand{\N}{\mathbb{N}}
\newcommand{\ap}{\alpha^\prime}
\newtheorem{theo}{Theorem}
\newtheorem{prop}{Proposition}[section]	
\newtheorem{defi}[prop]{Definition}
\newtheorem{Assumption}{Assumption}

\newtheorem{lemm}[prop]{Lemma}
\newtheorem{corr}[prop]{Corollary}
\newtheorem{rem}{Remark}
\newtheorem{ex}{Example}
\numberwithin{equation}{section}

\newtheorem{obs}{Observation}

\DeclareMathOperator{\supp}{supp}

\DeclareMathOperator{\tr}{tr}

\title[Cantor spectrum of graphene in magnetic fields]{Cantor spectrum of graphene in magnetic fields}

\author{Simon Becker}
\email{simon.becker@damtp.cam.ac.uk}

\address{DAMTP, University of Cambridge \\ Wilbeforce Rd, Cambridge \\ CB3 0WA (UK).}

\author{Rui Han}
\email{rhan@ias.edu}
\address{School of Math, Institute for Advanced Study (USA)}

\author{Svetlana Jitomirskaya}
\email{szhitomi@math.uci.edu}
\address{Department of Mathematics, University of California, Irvine (USA).}

\begin{document}

\begin{abstract}
We consider a quantum graph as a model of graphene in magnetic fields and give a complete analysis of the spectrum, for all constant fluxes. In particular, we show that if the reduced magnetic flux $\Phi / 2\pi$ through a honeycomb is irrational, the continuous spectrum is an unbounded Cantor set of Lebesgue measure zero.
\end{abstract}

\maketitle
\section{Introduction}\label{s:intr}
Graphene is a two-dimensional material that consists of carbon atoms
at the vertices of a hexagonal lattice. 
 Its experimental
discovery, unusual properties, and 
 applications led to  a
lot of attention in physics, see e.g. \cite{16FW}.  Electronic
properties of graphene have been extensively studied rigorously in the absence of magnetic fields \cite{FW,KP}.

Magnetic properties of graphene have also become a major research direction
in physics that has been kindled recently by the observation of the
quantum Hall effect \cite{nature} and strain-induced pseudo-magnetic
fields \cite{Gu} in graphene.  The purpose of this paper is to provide
for the first time an analysis of the spectrum of graphene in
magnetic fields with constant flux.


The fact that magnetic electron spectra have fractal structures was
first predicted by Azbel \cite{Az} and then numerically observed by
Hofstadter \cite{Ho} for the Harper's model. The scattering plot of the electron spectrum as a function of the magnetic flux is nowadays known as Hofstadter's butterfly. Verifying such results experimentally has been restricted for a long time due to the extraordinarily strong magnetic fields required. Only recently, self-similar structures in the electron spectrum in graphene have been observed \cite{Ch}, \cite{De}, \cite{Ga}, and \cite{Gor}. 

With this work, we provide a rigorous foundation for self-similarity
by showing that for irrational flux quanta, the electron spectrum of
graphene is a Cantor set. We say $A$ is a Cantor set if it is closed,
nowhere dense and has no isolated points (so compactness not
required). Let $\sigma_\Phi,\sigma_{cont}^\Phi,\sigma_{ess}^\Phi$  be
the (continuous, essential) spectra of $H^B$, the Hamiltonian
of the quantum graph
graphene in a magnetic field with constant flux $\Phi$, as defined in
\eqref{magop} \eqref{magdom} \eqref{defAe} \eqref{defPhi}, with some Kato-Rellich potential $V_{\vec{e}}\in L^2(\vec{e}).$ Let $H^D$ be the Dirichlet operator (no
magnetic field)  defined in
\eqref{Dirichlet} \eqref{maximaloperator}, and $\sigma(H^D)$ its spectrum.  Let
$\sigma_p^\Phi$ be the collection of eigenvalues of $H^B.$ Then we
have the following description of the topological structure and
point/continuous decomposition of the spectrum

\begin{theo}\label{T1}For any symmetric Kato-Rellich potential
  $V_{\vec{e}}\in L^2(\vec{e})$ 
we have
\begin{enumerate}
\item $\sigma^\Phi = \sigma_{ess}^\Phi,$
\item $\sigma_p^\Phi = \sigma(H^D),$
\item $\sigma_{cont}^\Phi$ is 
\begin{itemize}
\item a Cantor set of measure zero for $\frac{\Phi}{2\pi}\in \R\setminus \Q,$
\item a countable union of disjoint intervals for $\frac{\Phi}{2\pi} \in\Q,$
\end{itemize}
\item $\sigma_p^\Phi \cap \sigma_{cont}^\Phi =\emptyset$ for $\frac{\Phi}{2\pi} \notin \Z$,
\item the Hausdorff dimension $\dim_H (\sigma^{\Phi})\leq 1/2$ for generic $\Phi$.
\end{enumerate}
\end{theo}
Thus for irrational flux, the spectrum is a zero measure Cantor set 
plus a countable collection of flux-independent isolated eigenvalues, each of infinite
multiplicity, while for the rational flux the Cantor set is replaced by
a countable union of intervals.



Furthermore, we can also describe the spectral
decomposition of $H^B$.

\begin{theo}\label{T2}For any symmetric Kato-Rellich potential
  $V_{\vec{e}}\in L^2(\vec{e})$ 
we have
\begin{enumerate}
\item For $\frac{\Phi}{2\pi}\in \R\setminus \Q$ the spectrum on $\sigma_{cont}^\Phi$ is purely
  singular continuous.
\item For $\frac{\Phi}{2\pi}\in\Q,$ the spectrum on $\sigma_{cont}^\Phi$ is 
absolutely continuous.
\end{enumerate}
\end{theo}

Since molecular bonds in graphene are equivalent and delocalized, we
use an effective one-particle electron model 
 \cite{KP} on
a hexagonal graph with magnetic field \cite{KS}. While closely
related to the commonly used tight-binding model \cite{AEG}, we note
that unlike the latter, our model starts from actual differential
operator and is exact in every step, so does not involve any approximation. 
Moreover, while the density of states of the tight-binding model is symmetric around the Dirac point \cite[Fig. $5$]{C} and \cite[Fig. $3$]{HKL}, it is not what is found in experiments \cite{Go}, while the density of states for the quantum graph model \cite[Fig. $7$(A)]{BZ} has a significant similarity with the one observed experimentally
\cite{ourprl}. We note though that isolated eigenvalues are likely an
artifact of the graph model which does not allow something
similar to actual Coulomb potentials close to the carbon atoms or dissolving of
eigenstates supported on edges in the bulk. Thus while the isolated
eigenvalues are probably unphysical, there are reasons to expect that {\it continuous}
spectrum of the quantum graph operator described in this paper does
adequately capture the experimental properties
of graphene in the magnetic field. Finally, our analysis provides
full description of the spectrum of the tight-binding Hamiltonian as
well. Moreover, the applicability of our model is certainly not limited to graphene. Many atoms and even particles confined to the same lattice structure show similar physical properties \cite{Go} that are described well by this model (compare \cite[Fig. 7(A)]{BZ} with \cite[Fig. 2c]{Go}).

Earlier work showing Cantor spectrum on
quantum graphs with magnetic fields, e.g. for the square lattice
\cite{BGP} and magnetic chains studied in \cite{EV}, has been mostly
limited to applications of the Cantor spectrum of the almost Mathieu operator \cite{AJ,puig}.
In the case of graphene, we can no longer resort to this operator. The
discrete operator is then matrix-valued and can
be further reduced to a one-dimensional discrete quasiperiodic operator using
supersymmetry.  The resulting discrete operator is a singular Jacobi
matrix. Cantor spectrum (in fact, a stronger, dry ten martini type
statement) for Jacobi matrices of this type has been
studied in the framework of the extended Harper's model
\cite{han}. However, the method of \cite{han} that goes back to that
of \cite{aj2} relies on (almost) reducibility, and thus in particular  is not
applicable in absence of (dual) absolutely continuous spectrum which is
prevented by singularity. Similarly, the method of
\cite{AJ}  breaks down in presence of singularity in the Jacobi matrix
as well. Instead, we
present a new way that {\it exploits} singularity rather than
circumvent it by showing that the singularity leads to vanishing of the measure of
the spectrum, thus Cantor structure and singular continuity, once 4 of
Theorem \ref{T1} is established. \footnote{We note that singular continuity of
the spectrum of critical extended Harper's model (including for parameters leading
to singularity) has been proved recently in \cite{ajm,han2}
without establishing the Cantor nature.} Our method applies to also prove
zero measure Cantor spectrum of the extended Harper's model  whenever
the corresponding Jacobi matrix 
is singular.

As mentioned, our first step is a reduction to a matrix-valued
tight-binding hexagonal model. This leads to an operator $Q_\Lambda$
defined in \eqref{Q-op}. This operator has been studied
before for the case of rational magnetic flux (see \cite{HKL} and
references therein). Our analysis gives complete spectral description
for this operator as well. 

\begin{theo}\label{T3}
The spectrum of $Q_\Lambda(\Phi)$ is 
\begin{itemize} 
\item a finite union of intervals and purely
absolutely continuous for $\frac{\Phi}{2\pi}=\frac{p}{q}\in\Q,$ with the following measure estimate
\begin{align*}
|\sigma(Q_{\Lambda}(\Phi))|\leq \frac{C}{\sqrt{q}},
\end{align*}
\item singular continuous and a zero measure Cantor set for
$\frac{\Phi}{2\pi}\in \R\setminus \Q$,
\item a set of Hausdorff dimension $\dim_H(\sigma(Q_{\Lambda}(\Phi))) \leq 1/2$ for generic $\Phi$.
\end{itemize}
\end{theo} 

\begin{rem} We will show that the constant $C$ in the first item can be 
bounded by  $\frac{8\sqrt{6 \pi}}{9}$.\end{rem}
The theory of magnetic Schr\"odinger operators on graphs can be found
in \cite{KS}. The effective one-particle graph model for graphene
without magnetic fields was introduced in \cite{KP}. After
incorporating a magnetic field according to \cite{KS} in the model of
\cite{KP}, the reduction of differential operators on the graph to a
discrete tight-binding operator can be done using Krein's extension
theory for general self-adjoint operators on Hilbert spaces. This
technique has been introduced in \cite{P} for magnetic quantum graphs
on the square lattice. The quantum graph nature of the differential
operators causes, besides the contribution of the tight-binding operator to the continuous spectrum, a contribution to the point spectrum that consists of \emph{Dirichlet eigenfunctions} vanishing at every vertex. 

In this paper we develop the corresponding reduction for the hexagonal
structure and derive spectral conclusions in a way that allows easy
generalization to other planar graphs spanned by two basis vectors as
well.  In particular, our techniques should be applicable to study
quantum graphs on the
triangular lattice, which will be pursued elsewhere.


One of the striking properties of graphene is the presence
 of a linear dispersion relation which leads to the formation of
  conical structures in the Brillouin zone. The points where the cones
 match are called \emph{Dirac points} to account for the special
  dispersion relation.
Using magnetic translations introduced by \cite{Zak} we establish a
one-to-one correspondence between bands of the magnetic Schr\"odinger
operators on the graph and of the tight-binding operators for rational
flux quanta that only relies on Krein's theory. In particular, the
bands of the graph model always touch at the Dirac points and are
shown to have open gaps at the band edges of the associated Hill
operator if the magnetic flux is non-trivial. This way, the conical
Dirac points are preserved in rational magnetic fields. We obtain the
preceding results by first proving a bound on the operator norm of the
tight-binding operator and analytic perturbation theory. 

In \cite{KP} it was shown that the Dirichlet contribution to the
spectrum in the non-magnetic case is generated by compactly supported
eigenfunctions and that this is the only contribution to the point
spectrum of the Schr\"odinger operator on the graph. We extend this
result to magnetic Schr\"odinger operators on hexagonal graphs. Let
$H_{pp}$ be the pure point subspace accociated with $H^B.$ Then

\begin{theo}\label{T4}
For any $\Phi$, $H_{pp}$ is spanned by compactly supported 
eigenfunctions (in fact, by double hexagonal states).
\end{theo}

While
for rational $\Phi$ the proof is based on ideas similar to those of \cite{KP}, for
irrational $\Phi$ we no longer have an underlying periodicity thus
cannot use the arguments of \cite{K2}. After showing that there are double
hexagonal state eigenfunctions for each Dirichlet eigenvalue, it
remains to show their completeness. While there are various ways to show
that all $\ell^1$ (in a suitable sense) eigenfunctions are in the closure of the span of double
hexagonal states,  the $\ell^2$ condition is more elusive. Bridging the gap between $\ell^1$  and $\ell^2$  has been a
known difficult problem in several other scenarios
\cite{avilapreprint,ajm,aizenmanwarzel,jaksiklast}. Here we achieve
this by constructing, for each $\Phi$, an operator that would have all slowly decaying
$\ell^2$ eigenfunctions in its kernel and showing its
invertibility. This is done using constructive arguments and properties of holomorphic families of operators.

This paper is organized as follows. 
Section \ref{prelimsec} serves as background.
In Section \ref{secwithmagnetic}, we introduce the magnetic Schr\"odinger operator $H^B$ and (modified) Peierls' substitution, which enables us to reduce $H^B$ to a non-magnetic Schr\"odinger operator $\Lambda^B$ with magnetic contributions moved into the boundary conditions.
In Section \ref{mainlemmassec}, we present several key ingredients of the proofs of the main theorems: Lemmas \ref{HBQLam} and \ref{measure0cor} - \ref{Hausdorff}.
Lemma \ref{HBQLam} involves a further reduction from $\Lambda^B$ to a two-dimensional tight-binding Hamiltonian $Q_{\Lambda}(\Phi)$, and 
Lemmas \ref{measure0cor} - \ref{Hausdorff} reveal the topological structure of
$\sigma(Q_{\Lambda}(\Phi))$ (thus proving the topological part of Theorem \ref{T3}).
The proofs of Lemmas \ref{HBQLam}, \ref{measure0cor}, \ref{measureestlem} and \ref{Hausdorff} will be given is Sections \ref{KreinRed}, \ref{secproofofLemmakey2} and \ref{secHausdorff} respectively.
Section \ref{speanasec} is devoted to a complete spectral analysis of $H^B$, thus proving Theorem \ref{T1}, 
with the analysis of Dirichlet spectrum in Section
\ref{Dirichsec}, where in particular we prove Theorem \ref{T4};
absolutely continuous spectrum for rational flux in Section
\ref{acsec}, singular continuous spectrum for irrational flux in
Section \ref{scsec} (thus proving Theorem \ref{T2}).

\section{Preliminaries}\label{prelimsec}
Given a graph $G$, we denote the set of edges of $G$ by $\mathcal{E}(G),$ the set of vertices by $\mathcal{V}(G)$, and the set of edges adjacent to a vertex $v \in \mathcal{V}(G)$ by $\mathcal{E}_v(G).$  

For an operator $H$, let $\sigma(H)$ be its spectrum and $\rho(H)$ be the resolvent set.

The space $c_{00}$ is the space of all infinite sequences with only finitely many non-zero terms (finitely supported sequences). We denote by $\Omega^i(\R^2)$ the vector space of all $i$-covectors or differential forms of degree $i$ on $\R^2.$

For a set $U\subseteq \R$, let $|U|$ be its Lebesgue measure.

\subsection{Hexagonal quantum graphs}
This subsection is devoted to reviewing hexagonal quantum graphs without magnetic fields. 
The readers could refer to \cite{KP} for details. 
We include some material here that serves as a preparation for the study of quantum graphs with magnetic fields in Section \ref{secwithmagnetic}.

The effective one electron behavior in graphene can be described by a hexagonal graph with Schr\"odinger operators defined on each edge \cite{KP}.
The hexagonal graph $\Lambda$ is obtained by translating its fundamental cell $W_{\Lambda}$ shown in Figure \ref{Fig:Fcell}, consisting of vertices
\begin{equation}
r_0:=(0,0) \text{ and } r_1:=\left(\frac{1}{2}, \frac{\sqrt{3}}{2}\right)
\end{equation}
and edges
\begin{equation}
\begin{split}
\vec{f}&:=\operatorname{conv}\left(\{r_0,r_1 \}\right) \ \backslash \ \{r_0,r_1 \}, \\
\vec{g}&:=\operatorname{conv}\left(\{r_0,\left(-1,0\right) \}\right) \ \backslash \{r_0,\left(-1,0\right) \} , \text{ and } \\
\vec{h}&:=\operatorname{conv}\left( \left\{r_0, \left( \frac{1}{2}, -\frac{\sqrt{3}}{2} \right) \right\} \right) \ \backslash
 \ \left\{r_0, \left( \frac{1}{2}, -\frac{\sqrt{3}}{2} \right)\right\},
\end{split}
\end{equation}
along the basis vectors of the lattice.
The basis vectors are
\begin{equation}
\vec{b}_1:= \left(\frac{3}{2}, \frac{\sqrt{3}}{2} \right) \text{ and } 
\vec{b}_2:= \left(0,\sqrt{3}\right)
\end{equation}
and so the hexagonal graph $\Lambda \subset \mathbb{R}^2$ is given by the range of a $\mathbb{Z}^2$-action on the fundamental domain $W_{\Lambda}$
\begin{equation}
\label{lattice}
\Lambda:= \left\{ x \in \mathbb{R}^2:\ x = \gamma_1 \vec{b}_1+\gamma_2 \vec{b}_2+y \text{ for } \gamma \in \mathbb{Z}^2 \text{ and } y \in W_{\Lambda} \right\}.
\end{equation}

The fundamental domain of the dual lattice can be identified with the dual $2$-torus where the dual tori are defined as
\begin{equation}
\mathbb{T}_k^*:= \mathbb{R}^k / (2 \pi \mathbb{Z})^k.
\end{equation}

For any vertex $v\in \mathcal{V}(\Lambda)$, we denote by $[v]\in \mathcal{V}(W_{\Lambda})$ the unique vertex, $r_0$ or $r_1$, for which there is $\gamma\in \mathbb{Z}^2$ such that 
\begin{align}
v=\gamma_1 \vec{b}_1+\gamma_2 \vec{b}_2+[v].
\end{align}
We will sometimes denote $v$ by $(\gamma_1, \gamma_2, [v])$ to emphasize the location of $v$.
We also introduce a similar notation for edges. 
For an edge $\vec{e} \in \mathcal{E}(\Lambda)$, we will sometimes denote it by $\gamma_1, \gamma_2, \vec{[e]}$.
Finally, for any $x\in \Lambda$, we will also denote its unique preimage in $W_{\Lambda}$ by $[x]$ \footnote{so that $y$ in \eqref{lattice}=$[x]$}.

\begin{center}
\begin{figure}
\centerline{\includegraphics[height=7cm]{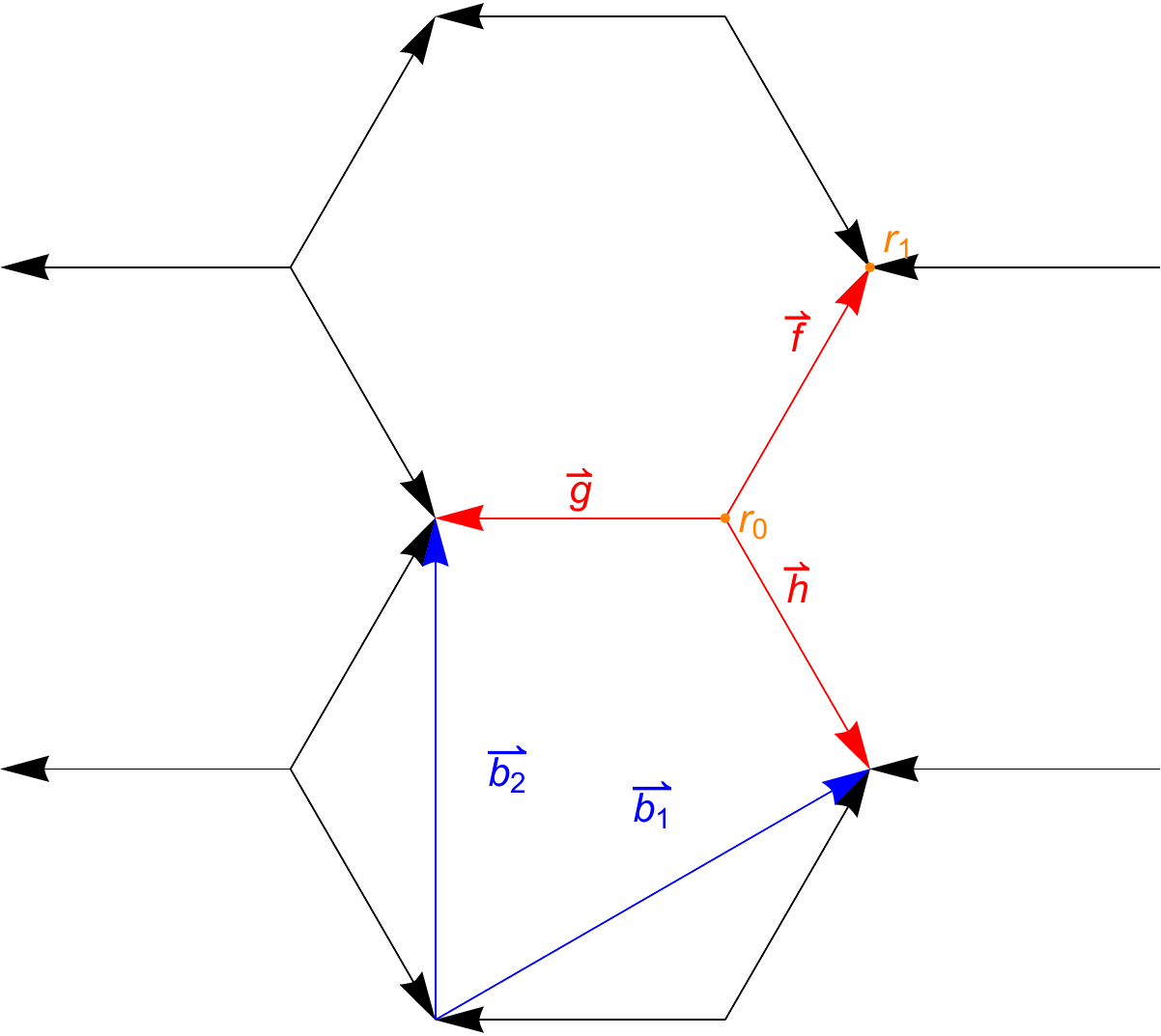}} 
\caption{The fundamental cell and lattice basis vectors of $\Lambda$ \label{Fig:Fcell}.}
\label{Figure1}
\end{figure}
\end{center}

We can orient the edges in terms of initial and terminal maps
\begin{align}
i&:\mathcal{E}(\Lambda) \rightarrow \mathcal{V}(\Lambda) \text{ and }
t:\mathcal{E}(\Lambda) \rightarrow \mathcal{V}(\Lambda)
\end{align}
where $i$ and $t$ map edges to their initial and terminal ends respectively.
It suffices to specify the orientation on the edges of the fundamental domain $W_{\Lambda}$ to obtain an oriented graph $\Lambda$
\begin{equation}
\begin{split}
i(\vec{f})&=i(\vec{g})=i(\vec{h})=r_0, \\
t(\vec{f})&=r_1, \ t(\vec{g})=r_1-\vec{b}_1,\text{ and } t(\vec{h})=r_1-\vec{b}_2.
\end{split}
\end{equation}
For arbitrary $\vec{e} \in \mathcal{E}(\Lambda)$, we then just extend those maps by
\begin{align}
i(\vec{e}):=\gamma_1 \vec{b}_1+\gamma_2 \vec{b}_2+i(\vec{[e]}) \text{ and }
t(\vec{e}):=\gamma_1 \vec{b}_1+\gamma_2 \vec{b}_2+t(\vec{[e]}).
\end{align}

Let $i(\Lambda)=\{v\in \mathcal{V}(\Lambda): v=i(\vec{e}) \text{ for some }\vec{e}\in \mathcal{E}(\Lambda)\}$ be the collection of initial vertices, and 
$t(\Lambda)=\{v\in \mathcal{V}(\Lambda): v=t(\vec{e}) \text{ for some }\vec{e}\in \mathcal{E}(\Lambda)\}$ be the collection of terminal ones.
It should be noted that based on our orientation, $\mathcal{V}(\Lambda)$ is a disjoint union of $i(\Lambda)$ and $t(\Lambda)$.

Every edge $\vec{e} \in \mathcal{E}(\Lambda)$ is of length one and thus has a canonical chart
\begin{equation}
\label{chart}
\kappa_{\vec{e}}: \vec{e} \rightarrow (0,1), (i(\vec{e})x+t(\vec{e})(1-x))\mapsto x
\end{equation}    
that allows us to define function spaces and operators on $\vec{e}$ and finally on the entire graph.
For $n \in \mathbb{N}_0$, the Sobolev space $\mathcal{H}^n(\mathcal{E}\left(\Lambda\right))$ on $\Lambda$ is the Hilbert space direct sum
\begin{equation}
\mathcal{H}^n(\mathcal{E}\left(\Lambda\right)):= \bigoplus_{\vec{e} \in \mathcal{E}(\Lambda)} \mathcal{H}^n(\vec{e}).
\end{equation}

On every edge $\vec{e} \in \mathcal{E}(\Lambda)$ we define the maximal Schr\"odinger operator 
\begin{equation}
\begin{split}
\label{maximaloperator}
H_{\vec{e}}&: \mathcal{H}^2(\vec{e}) \subset L^2(\vec{e}) \rightarrow L^2(\vec{e})  \\
H_{\vec{e}}&\psi_{\vec{e}} := -\psi''_{\vec{e}}+V_{\vec{e}} \psi_{\vec{e}}
\end{split}
\end{equation}
with Kato-Rellich potential $V_{\vec{e}} \in L^2(\vec{e})$ that is the same on every edge and {\it even} with respect to the center of the edge. 
Let 
\begin{align}\label{V01}
V(t)=V_{\vec{e}}((\kappa_{\vec{e}})^{-1}(t)).
\end{align}
Then 
\begin{align}\label{Veven}
V(t)=V(1-t).
\end{align}

One self-adjoint restriction of \eqref{maximaloperator} is the Dirichlet operator 
\begin{equation}
\begin{split}
\label{Dirichlet} 
&H^D:=\bigoplus_{\vec{e} \in \mathcal{E}(\Lambda)}\left(\mathcal{H}_0^1(\vec{e})\cap \mathcal{H}^2(\vec{e}) \right) \subset L^2(\mathcal{E}\left(\Lambda\right)) \rightarrow L^2(\mathcal{E}\left(\Lambda\right))  \\
&(H^D\psi)_{\vec{e}}:=H_{\vec{e}}\psi_{\vec{e}},
\end{split}
\end{equation}
where $\mathcal{H}_0^1(\vec{e})$ is the closure of compactly supported smooth functions in $\mathcal{H}^1(\vec{e})$.
The Hamiltonian we will use to model the graphene without magnetic fields is the self-adjoint \cite{K2} operator $H$ on $\Lambda$ with Neumann type boundary conditions
\begin{equation}
\begin{split}
\label{domnon-magnHam}
D(H):= \biggl\{ \psi=(\psi_{\vec{e}}) \in \mathcal{H}^2(\mathcal{E}(\Lambda)): 
&\text{ for all }v \in \mathcal{V}(\Lambda), \psi_{\vec{e}_1}(v)=\psi_{\vec{e}_2}(v) \text{ if }\vec{e}_1,\vec{e}_2 \in \mathcal{E}_v(\Lambda)  \\
&\text{ and } \sum_{\vec{e} \in \mathcal{E}_v(\Lambda)} \psi_{\vec{e}}^\prime(v)=0\biggr\}
\end{split}
\end{equation}
and defined by
\begin{equation}
\begin{split}
\label{non-magnHam}
&\ H: D(H) \subset L^2(\mathcal{E}(\Lambda)) \rightarrow L^2(\mathcal{E}(\Lambda)) \\
&(H\psi)_{\vec{e}}:=(H_{\vec{e}}\psi_{\vec{e}}).
\end{split}
\end{equation}
\begin{rem}
The self-adjointness of $H$ will also follow from the self-adjointness of the more general family of magnetic Schr\"odinger operators that is obtained in Sec. \ref{KreinRed}. 
\end{rem}
\begin{rem}
The orientation is chosen so that all edges at any vertex are either all incoming or outgoing. Thus, there is no need to distinguish those situations in terms of a directional derivative in the boundary conditions \eqref{domnon-magnHam}.
\end{rem}

\subsubsection{Floquet-Bloch decomposition}
Operator $H$ commutes with the standard lattice translations
\begin{equation}
\begin{split}
\label{standardtran}
T^{\text{st}}_{\gamma}
: &L^2(\mathcal{E}(\Lambda)) \rightarrow L^2(\mathcal{E}(\Lambda))\\
  &f \mapsto f(\cdot-\gamma_1 \vec{b}_1-\gamma_2 \vec{b}_2)
\end{split}
\end{equation}
for any $\gamma \in \mathbb{Z}^2.$
In terms of those, we define the Floquet-Bloch transform for $x \in \mathcal{E}(W_{\Lambda})$ and $k \in \mathbb{T}_2^*$ first on function $f \in C_c(\mathcal{E}(\Lambda))$ 
\begin{equation}
\label{Gelfand}
(Uf)(k,x):= \sum_{\gamma \in \mathbb{Z}^2} (T^{\text{st}}_{\gamma}f)(x)e^{i\langle k,\gamma \rangle}
\end{equation}
and then extend it to a unitary map $U \in \mathcal{L}(L^2(\mathcal{E}(\Lambda)),L^2(\mathbb{T}_2^* \times \mathcal{E}(W_{\Lambda}))$
with inverse
\begin{equation}
(U^{-1}\varphi)(x)= \int_{\mathbb{T}_2^*} \varphi(k,[x]) e^{-i \langle \gamma,k \rangle} \frac{dk}{(2 \pi)^2},
\end{equation}
where $[x] \in \mathcal{E}\left(W_{\Lambda}\right)$ is the unique pre-image of $x$ in $W_{\Lambda}$, and $\gamma\in \Z^2$ is defined by $x=\gamma_1 \vec{b}_1+\gamma_2 \vec{b}_2+[x]$.

Then standard Floquet-Bloch theory implies that there is a direct integral representation of $H$
\begin{equation}
\label{directint}
UHU^{-1}= \int_{\mathbb{T}_2^{*}}^{\oplus} H^k \frac{dk}{(2\pi)^2}
\end{equation}
in terms of self-adjoint operators $H^k$ 
\begin{equation}
\begin{split}
&\ H^k:D(H^k) \subset L^2(\mathcal{E}\left(W_{\Lambda}\right)) \rightarrow L^2(\mathcal{E}\left(W_{\Lambda}\right)) \\
&(H^k\psi)_{\vec{e}}:=(H_{\vec{e}}\psi_{\vec{e}})
\end{split}
\end{equation}
on the fundamental domain $W_{\Lambda}$ with Floquet boundary conditions
\begin{equation}
\begin{split}
\label{DomainHk}
D(H^k):=\biggl\{&\psi \in \mathcal{H}^2(\mathcal{E}\left(W_{\Lambda}\right)): \psi_{\vec{f}}(r_0)=\psi_{\vec{g}}(r_0)=\psi_{\vec{h}}(r_0) \text{ and } \sum_{\vec{e} \in \mathcal{E}_{r_0}(\Lambda)} \psi_{\vec{e}}^\prime (r_0)=0, \\
&\text{ as well as }\psi_{\vec{f}}(r_1)=e^{ik_1}\psi_{\vec{g}}(r_1-\vec{b}_1)=e^{ik_2}\psi_{\vec{h}}(r_1-\vec{b}_2)\\
&\ \text{and }\psi_{\vec{f}}^\prime (r_1)+e^{ik_1}\psi_{\vec{g}}^\prime (r_1-\vec{b}_1)+e^{ik_2}\psi_{\vec{h}}^\prime (r_1-\vec{b}_2)=0
\biggr\}.
\end{split}
\end{equation}

Fix an edge $\vec{e} \in \mathcal{E}(\Lambda)$ and $\lambda \notin \sigma(H^D)$. There are linearly independent $\mathcal{H}^2(\vec{e})$-solutions $\psi_{\lambda,1,\vec{e}}$ and $\psi_{\lambda,2,\vec{e}}$ to the equation 
$H_{\vec{e}}\psi_{\vec{e}}=\lambda \psi_{\vec{e}}$ with the following boundary condition
\begin{equation}
\psi_{\lambda,1,\vec{e}}(i(\vec{e}))=1, \quad \psi_{\lambda,1,\vec{e}}(t(\vec{e}))=0, \quad \psi_{\lambda,2,\vec{e}}(i(\vec{e}))=0, \text{ and }  \psi_{\lambda,2,\vec{e}}(t(\vec{e}))=1.
\end{equation}
Any eigenfunction to operators $H^k$, with eigenvalues away from $\sigma(H^D)$, can therefore be written in terms of those functions for constants $a,b \in \mathbb{C}$
\begin{equation}
\psi:= \begin{cases} a \ \psi_{\lambda,1,\vec{f}}+ b \ \psi_{\lambda,2,\vec{f}} &\mbox{along edge } \vec{f}  \\
a \ \psi_{\lambda,1,\vec{g}} + e^{-ik_1}b \ \psi_{\lambda,2,\vec{g}} & \mbox{along edge } \vec{g} \\
a \ \psi_{\lambda,1,\vec{h}} + e^{-ik_2}b \ \psi_{\lambda,2,\vec{h}} & \mbox{along edge } \vec{h} \end{cases} 
\end{equation}
with the continuity conditions of \eqref{DomainHk} being already incorporated in the representation of $\psi.$
Imposing the conditions stated on the derivatives in \eqref{DomainHk} shows that $\psi$ is non-trivial ($a,b$ not both equal to zero) and therefore an eigenfunction with eigenvalue $\lambda \in \mathbb{R}$ to $H^k$  iff
\begin{equation}
\label{condition}
\eta(\lambda)^2 = \frac{\left\vert 1+ e^{ik_1}+ e^{ik_2} \right \vert^2}{9}
\end{equation}
with $\eta(\lambda):=\frac{\psi_{\lambda,2,\vec{e}}'(t(\vec{e}))}{\psi_{\lambda,2,\vec{e}}'(i(\vec{e}))}$ well-defined away from the Dirichlet spectrum.

By noticing that the range of the function on the right-hand side of \eqref{condition} is $[0,1],$ the following spectral characterization is obtained \cite[Theorem $3.6$]{KP}.
\begin{theo}
As a set, the spectrum of $H$ away from the Dirichlet spectrum is given by
\begin{equation}
\sigma(H)\backslash \sigma(H^D) = \left\{ \lambda \in \mathbb{R}:\ \left\lvert \eta(\lambda) \right\rvert \le 1  \right\}\backslash \sigma(H^D).
\end{equation}
\end{theo}

\subsubsection{Dirichlet-to-Neumann map}
Fix an edge $\vec{e}\in \mathcal{E}(\Lambda)$. Let $c_{\lambda, \vec{e}}, s_{\lambda, \vec{e}}$ be solutions to $H_{\vec{e}}\psi_{\vec{e}}=\lambda \psi_{\vec{e}}$ with the following boundary condition
\begin{equation}\label{defcs}
\left(
\begin{matrix}
c_{\lambda, \vec{e}}(i(\vec{e}))\ &s_{\lambda, \vec{e}}(i(\vec{e}))\\
c^{\prime}_{\lambda, \vec{e}}(i(\vec{e}))\ &s^{\prime}_{\lambda, \vec{e}}(i(\vec{e}))
\end{matrix}
\right)=
\left(
\begin{matrix}
1\ \ &0\\
0\ \ &1
\end{matrix}
\right).
\end{equation}

We point out that $c_{\lambda}(t):=c_{\lambda, \vec{e}}(\kappa_{\vec{e}}^{-1}(t))$ and 
$s_{\lambda}(t):=s_{\lambda, \vec{e}}(\kappa_{\vec{e}}^{-1}(t))$ are independent of $\vec{e}$. 
They are clearly solutions to $-\psi^{\prime\prime}+V\psi=\lambda \psi$ on $(0,1)$, with 
$c_{\lambda}(0)=1, c^{\prime}_{\lambda}(0)=0, s_{\lambda}(0)=0, s^{\prime}_{\lambda}(0)=1$, where $V$ is defined in (\ref{V01}).

Then for $\lambda\notin \sigma(H^D)$, namely when $s_{\lambda}(1)\neq 0$, any $\mathcal{H}^2(\vec{e})$-solution $\psi_{\lambda, \vec{e}}$ can be written as a linear combination of $c_{\lambda, \vec{e}}, s_{\lambda, \vec{e}}$
\begin{equation}
\label{representsol}
\psi_{\lambda, \vec{e}} (x) = \frac{\psi_{\lambda, \vec{e}}(t(\vec{e}))-\psi_{\lambda, \vec{e}}(i(\vec{e}))c_{\lambda}(1)}{s_{\lambda}(1)}s_{\lambda, \vec{e}}(x)+\psi_{\lambda, \vec{e}}(i(\vec{e})) c_{\lambda, \vec{e}}(x).
\end{equation}
The Dirichlet-to-Neumann map is defined by
\begin{equation}
\label{DtN}
m(\lambda) := \frac{1}{s_{\lambda}(1)} \left( \begin{matrix} -c_{\lambda}(1) & 1 \\ 1 & -s'_{\lambda}(1)  \end{matrix} \right)
\end{equation}
with the property that for $\psi_{\lambda, \vec{e}}$ as in \eqref{representsol}
\begin{equation}
\left( \begin{matrix} \psi^{\prime}_{\lambda, \vec{e}}(i(\vec{e})) \\ -\psi^{\prime}_{\lambda, \vec{e}}(t(\vec{e})) \end{matrix} \right) = m(\lambda) \left( \begin{matrix} \psi_{\lambda, \vec{e}}(i(\vec{e})) \\ \psi_{\lambda, \vec{e}}(t(\vec{e})) \end{matrix} \right).
\end{equation}
For the second component, the constancy of the Wronskian is used.
Since $V(t)$ is assumed to be even, the intuitive relation 
\begin{equation}
\label{cossinrel}
c_{\lambda}(1)=s'_{\lambda}(1)
\end{equation} 
remains also true for non-zero potentials. 

For $\lambda\notin \sigma(H^D)$, by expressing $c_{\lambda}(1)$ in terms of $\psi_{\lambda,1, \vec{e}}$ and $\psi_{\lambda,2, \vec{e}}$, it follows immediately that
\begin{equation}\label{eta=c}
\eta(\lambda)=s^\prime_{\lambda}(1).
\end{equation}

\subsubsection{Relation to Hill operators}\label{Hillsec}
Using the potential $V(t)$ \eqref{V01}, we define the $\mathbb{Z}$-periodic Hill potential $V_{\text{Hill}} \in L^2_{\text{loc}}(\mathbb{R}).$
\begin{equation}
V_{\text{Hill}}(t):=V(t\ (\operatorname{mod}1)),
\end{equation}
for $t \in \mathbb{R}.$ The associated self-adjoint Hill operator on the real line is given by
\begin{equation}
\begin{split}
H_{\text{Hill}}&: \mathcal{H}^2(\mathbb{R}) \subset L^2(\mathbb{R}) \rightarrow L^2(\mathbb{R})  \\
H_{\text{Hill}}&\psi:=-\psi''+ V_{\text{Hill}} \psi.
\end{split}
\end{equation}
Then $c_{\lambda},s_{\lambda}\in \mathcal{H}^2(0,1)$, extending naturally to $\mathcal{H}^2_{\text{loc}}(\mathbb{R})$, become solutions to
\begin{equation}
\label{classicalsol}
H_{\text{Hill}}\psi= \lambda \psi.
\end{equation}

The monodromy matrix associated with $H_{\text{Hill}}$ is the matrix valued function 
\begin{equation}
Q(\lambda):= \left( 
\begin{matrix}
  c_{\lambda}(1) & s_{\lambda}(1)  \\
  c_{\lambda}'(1) & s_{\lambda}'(1)  
 \end{matrix}\right)
 \end{equation}
 and depends by standard ODE theory holomorphically on $\lambda.$
Its normalized trace 
\begin{equation}\label{defDelta}
\Delta({\lambda}):= \frac{\operatorname{tr}(Q(\lambda))}{2}=s_{\lambda}'(1)
\end{equation}
is called the Floquet discriminant. In the simplest case when $V_{\text{Hill}}=0,$ the Floquet discriminant is just $\Delta(\lambda)=\cos\left(\sqrt{\lambda}\right)$ for $\lambda \ge 0.$ 

By the well-known spectral decomposition of periodic differential operators on
the line \cite{RS4}, the spectrum of the Hill operator is purely absolutely continuous and satisfies
\begin{equation}
\label{Hillspectrum}
\sigma(H_{\text{Hill}})= \left\{\lambda \in \mathbb{R}: \left\lvert \Delta(\lambda) \right\rvert \le 1 \right\}=\bigcup_{n=1}^{\infty} [\alpha_n,\beta_n]
\end{equation}
where $B_n:=[\alpha_n,\beta_n]$ denotes the $n$-th Hill band with $\beta_n \le \alpha_{n+1}$. We have $\Delta\vert_{\operatorname{int}(B_n)}'(\lambda) \neq 0.$

Putting (\ref{eta=c}) and (\ref{defDelta}) together, we get the following relation
\begin{align}
\Delta(\lambda)=\eta(\lambda), \text{ for }\lambda \notin \sigma(H^D),
\end{align}
that connects the Hill spectrum with the spectrum of the graphene Hamiltonian.

Also, if $\lambda\in \sigma (H^D)$, then by the symmetry of the potential, the Dirichlet eigenfunction are either even or odd with respect to $\frac{1}{2}$. Thus, Dirichlet eigenvalues can only be located at the edges of the Hill bands. Namely,
\begin{align}\label{DirichletHilledge}
\Delta(\lambda)=\pm 1, \text{  for  }\lambda\in \sigma(H^D).
\end{align}

\subsubsection{Spectral decomposition}
The singular continuous spectrum of $H$ is empty by the direct integral decomposition \eqref{directint} \cite{GeNi}. 
Due to Thomas \cite{T} there is the characterization, stated also in \cite[Corollary $6.11$]{K}, of the pure point spectrum of fibered operators: 
$\lambda$ is in the pure point spectrum iff the set $\{ k \in \mathbb{T}^{*}_2; \lambda_j(k)=\lambda \}$ has positive measure where $\lambda_j(k)$ is the $j$-th eigenvalue of $H^k$.
Away from the Dirichlet spectrum, the condition $\mathbb{R} \ni \lambda=\lambda_j(k)$ is by \eqref{condition} equivalent to $\Delta(\lambda)^2 = \frac{\left\lvert 1+e^{ik_1}+e^{ik_2} \right\rvert^2}{9}.$ Yet, the level-sets of this function are of measure zero. The spectrum of $H$ away from the Dirichlet spectrum is therefore purely absolutely continuous.
The Dirichlet spectrum coincides with the point spectrum of $H$ and is spanned by so-called loop states that consist of six Dirichlet eigenfunctions wrapped around each hexagon of the lattice \cite[Theorem $3.6$(v)]{KP}.
Hence, the spectral decomposition in the case without magnetic field is given by
\begin{theo}
\label{zero magnetic potential}
The spectra of $\sigma(H)$ and $\sigma(H_{\text{Hill}})$ coincide as sets. Aside from the Dirichlet contribution to the spectrum, $H$ has absolutely continuous spectrum as in Fig.\ref{Fig:firstbands} with conical cusps at the points (Dirac points) where two bands on each Hill band meet. The Dirichlet spectrum is contained in the spectrum of $H$, is spanned by loop states supported on single hexagons, and is thus infinitely degenerated. 
\end{theo}

\begin{figure}
\centerline{\includegraphics[height=9cm]{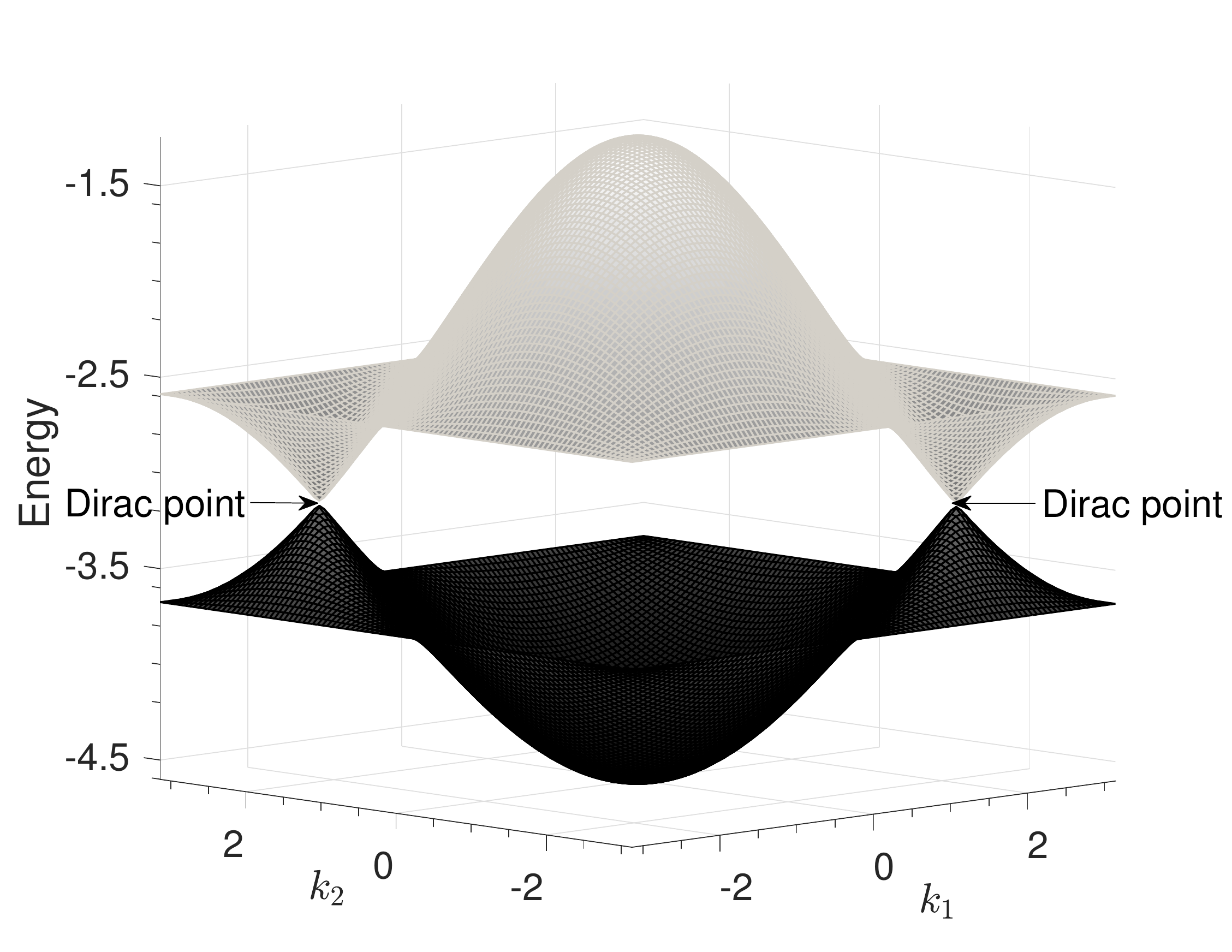}} 
\caption{The first two bands of the Schr\"odinger operator on the graph with Mathieu potential $V(t)=20\cos(2\pi t)$ and no magnetic field showing the characteristic conical Dirac points where the two differently colored bands touch. The two bands are differently colored. \label{Fig:firstbands}}
\end{figure}

\subsection{One-dimensional quasi-periodic Jacobi matrices}\label{preJacobi}
\

The proof of the main Theorems will involve the study of a one-dimensional quasi-periodic Jacobi matrix. 
We include several general facts that will be useful.

Let $H_{\Phi, \theta}\in \mathcal{L}(l^2(\Z))$ be a quasi-periodic Jacobi matrix, that is given by
\begin{equation}\label{defH}
(H_{\Phi, \theta}u)_m=c\left(\theta+m\frac{\Phi}{2\pi} \right)u_{m+1}+\overline{c\left(\theta+(m-1)\frac{\Phi}{2\pi}\right)}u_{m-1}+v\left(\theta+m \frac{\Phi}{2\pi}\right)u_m.
\end{equation}

Let $\Sigma_{\Phi, \theta}$ be the spectrum of $H_{\Phi, \theta}$, and $\Sigma_{\Phi}=\bigcup_{\theta\in \T_1}\Sigma_{\Phi, \theta}$. It is a well known result that 
for irrational $\frac{\Phi}{2\pi}$, the set $\Sigma_{\Phi, \theta}$ is independent of $\theta$, thus $\Sigma_{\Phi, \theta}=\sigma_{\Phi}$.
It is also well known that, for any $\Phi$, $\Sigma_{\Phi}$ has no isolated points.

\subsubsection{Transfer matrix and Lyapunov exponent}\label{deftransfersec}
\

We assume that $c(\theta)$ has finitely many zeros (counting multiplicity), and label them as $\theta_1, \theta_2,..., \theta_m$. 
\footnote{In our concrete model, $c(\theta)=1+e^{-2\pi i\theta}$, see \eqref{def:cv}, hence has a single zero $\theta_1=1/2$.}
Let $\Theta:=\cup_{j=1}^m \cup_{k\in \Z}\left\lbrace \theta_j+k\frac{\Phi}{2\pi}\right\rbrace$, in particular if $\frac{\Phi}{2\pi}\in \Q$, then $\Theta$ is a finite set in $\T$.

For $\theta\notin \Theta$, the eigenvalue equation $H_{\Phi, \theta}u=\lambda u$ has the following dynamical reformulation:
\begin{align*}
\left(
\begin{matrix}
u_{n+1}\\
u_n
\end{matrix}
\right)=
A^{\lambda}\left(\theta+n\frac{\Phi}{2\pi} \right)
\left(
\begin{matrix}
u_n\\
u_{n-1}
\end{matrix}
\right),
\end{align*}
where 
\begin{align*}
\mathrm{GL}(2, \C)\ni A^{\lambda}(\theta)=
\frac{1}{c(\theta)}
\left(
\begin{matrix}
\lambda-v(\theta)\ &-\overline{c(\theta-\frac{\Phi}{2\pi})}\\
c(\theta)\ &0
\end{matrix}
\right)
\end{align*}
is called the {\it transfer matrix}.
Let 
\begin{align*}
A^{\lambda}_{n}(\theta)=A^{\lambda}(\theta+(n-1)\frac{\Phi}{2\pi})\cdots A^{\lambda}(\theta+\frac{\Phi}{2\pi})A^{\lambda}(\theta)
\end{align*}
be the {\it n-step transfer matrix}.

We define the Lyapunov exponent of $H_{\Phi, \theta}$ at energy $\lambda$ as
\begin{align}\label{defLE}
L(\lambda, \Phi):=\lim_{n\rightarrow\infty}\frac{1}{n}\int_{\T_1}\log{\|A^{\lambda}_n(\theta)\|}\ \mathrm{d}\theta.
\end{align}

By a trivial bound $\|A\|^2\geq |\det{A}|$, which comes from the fact $A$ is a $2\times 2$ matrix, we get
\begin{align}
L(\lambda, \Phi)\geq \lim_{n\to\infty}\frac{1}{2n}\int_{\T_1}\log\left(\frac{|c(\theta-\frac{\Phi}{2\pi})|}{|c(\theta+(n-1)\frac{\Phi}{2\pi}|}\right)\ \mathrm{d}\theta=0.
\end{align}

\subsubsection{Normalized transfer matrix}
Let $|c(\theta)|=\sqrt{c(\theta)\overline{c(\theta)}}$.
We introduce the {\it normalized transfer matrix}:
\begin{align*}
\mathrm{SL}(2, \R)\ni \tilde{A}^{\lambda}(\theta)=
\frac{1}{\sqrt{|c(\theta)||c(\theta-\frac{\Phi}{2\pi})|}}
\left(
\begin{matrix}
\lambda-v(\theta)\ &-|c(\theta-\frac{\Phi}{2\pi})|\\
|c(\theta)|\ &0
\end{matrix}
\right)
\end{align*}
and the {\it n-step normalized transfer matrix} $\tilde{A}^{\lambda}_{n}(\theta)$.

The following connection between $A^{\lambda}$ and $\tilde{A}^{\lambda}$ is clear:
\begin{align}\label{conjtildeAA}
\tilde{A}^{\lambda}(\theta)=
\frac{c(\theta)}{\sqrt{|c(\theta)||c(\theta-\frac{\Phi}{2\pi})|}}
\left(
\begin{matrix}
1\ &0\\
0\ &\sqrt{\frac{\overline{{c}(\theta)}}{c(\theta)}}
\end{matrix}
\right)
A^{\lambda}(\theta)
\left(
\begin{matrix}
1\ &0\\
0\ &\sqrt{\frac{\overline{{c}(\theta-\frac{\Phi}{2\pi})}}{c(\theta-\frac{\Phi}{2\pi})}}
\end{matrix}
\right)^{-1}.
\end{align}

When $\frac{\Phi}{2\pi}=\frac{p}{q}$ is rational, \eqref{conjtildeAA} yields
\begin{align}\label{tildeAqAq}
\tr(\tilde{A}^{\lambda}_q(\theta))=\frac{\prod_{j=0}^{q-1}c(\theta+j\frac{p}{q})}{\prod_{j=0}^{q-1}|c(\theta+j\frac{p}{q})|}\tr (A^{\lambda}_q(\theta)).
\end{align}
Let 
\begin{align}\label{defD}
D^{\lambda}(\theta)=c(\theta)A^{\lambda}(\theta)=\left(
\begin{matrix}
\lambda-v(\theta)\ &-\overline{c(\theta-\frac{\Phi}{2\pi})}\\
c(\theta)\ &0
\end{matrix}
\right)
\end{align}
and $D_n^\lambda(\theta)=D^\lambda(\theta+(n-1)\frac{\Phi}{2\pi})\cdots D^\lambda(\theta+\frac{\Phi}{2\pi})D^\lambda(\theta)$.
Then when $\frac{\Phi}{2\pi}=\frac{p}{q}$ is rational, (\ref{tildeAqAq}) becomes
\begin{align}\label{tildeAqDq}
\tr(\tilde{A}^{\lambda}_q(\theta))=\frac{\tr(D^{\lambda}_q(\theta))}{\prod_{j=0}^{q-1}|c(\theta+j\frac{p}{q})|}.
\end{align}
Note that although $A^{\lambda}_n(\theta)$ is not well-defined for $\theta\in \Theta$, $D^{\lambda}_n(\theta)$ is always well-defined.


\subsection{Continued fraction expansion}
Let $\alpha\in \R\setminus \Q$, then $\alpha$ has the following continued fraction expansion 
\begin{align*}
\alpha=a_0+\frac{1}{a_1+\frac{1}{a_2+\frac{1}{a_3+\cdots}}},
\end{align*}
with $a_0$ being the integer part of $\alpha$ and $a_n\in \N^+$ for $n\geq 1$.


For $\alpha\in (0,1)$, let the reduced rational number
\begin{align}\label{defpnqn}
\frac{p_n}{q_n}=\frac{1}{a_1+\frac{1}{a_2+\frac{1}{\cdots+\frac{1}{a_n}}}}
\end{align}
be the continued fraction approximants of $\alpha$. 

The following property of continued fraction expansions is well-known:
\begin{align}\label{alpha-pnqn}
|\alpha-\frac{p_n}{q_n}|\leq \frac{1}{q_{n}q_{n+1}}.
\end{align}


\subsection{Lidskii inequalities}
Let $A$ be an $n\times n$ self-adjoint matrices, let $E_1(A)\geq E_2(A)\geq \cdots \geq E_n(A)$ be its eigenvalues.
Then, for the eigenvalues of the sum of two self-adjoint matrices, we have
\begin{align}\label{Lid}
\begin{cases}
E_{i_1}(A+B)+\cdots +E_{i_k}(A+B)\leq E_{i_1}(A)+\cdots +E_{i_k}(A)+E_1(B)+\cdots +E_k (B)\\
E_{i_1}(A+B)+\cdots +E_{i_k}(A+B)\geq E_{i_1}(A)+\cdots +E_{i_k}(A)+E_{n-k+1} (B)+\cdots +E_n (B)
\end{cases}
\end{align}
for any $1\leq i_1<\cdots <i_k\leq n$. These two inequalities are called {\it Lidskii inequality} and {\it dual Lidskii inequality}, respectively.

\section{Magnetic Hamiltonians on quantum graphs}\label{secwithmagnetic}
\subsection{Magnetic potential}

\

Given a vector potential $A(x)=A_1(x_1, x_2)\ dx_1+A_2(x_1, x_2)\ dx_2 \in \Omega^1(\mathbb{R}^2)$, the scalar vector potential $A_{\vec{e}} \in C^{\infty}(\vec{e})$ along edges 
$\vec{e} \in \mathcal{E}(\Lambda)$ is obtained by evaluating the form $A$ on the graph along the vector field generated by edges $\vec{[e]} \in \mathcal{E}(W_{\Lambda})$
\begin{equation}\label{defAe}
A_{\vec{e}}(x) := A(x)\left(\vec{[e]}_1 \partial_1 + \vec{[e]}_2 \partial_2 \right).
\end{equation}
The integrated vector potentials are defined as $\beta_{\vec{e}}:=\int_{\vec{e}} A_{\vec{e}}(x) dx$ for $\vec{e} \in \mathcal{E}(\Lambda).$ 
\begin{Assumption}
The magnetic flux $\Phi$ through each hexagon $Q$ of the lattice 
\begin{equation}\label{defPhi}
\Phi:= \int_{Q} dA
\end{equation}
is assumed to be constant.
\end{Assumption} 
Let us mention that the assumption above is equivalent to the following equation, in terms of the integrated vector potentials
\begin{align}\label{constantfluxbeta}
\beta_{\gamma_1, \gamma_2, \vec{f}}-\beta_{\gamma_1, \gamma_2+1, \vec{h}}+\beta_{\gamma_1, \gamma_2+1, \vec{g}}-\beta_{\gamma_1-1, \gamma_2+1,\vec{f}}+\beta_{\gamma_1-1, \gamma_2+1,\vec{h}}-\beta_{\gamma_1, \gamma_2, \vec{g}}=\Phi,
\end{align}
for any $\gamma_1, \gamma_2\in \Z$.

\begin{ex}
The vector potential $A \in \Omega^1(\mathbb{R}^2)$ of a homogeneous magnetic field $B  \in \Omega^2(\mathbb{R}^2)$ 
\begin{equation}
B(x)=B_0 \ dx_1 \wedge dx_2
\end{equation}
can be chosen as
\begin{equation}
\label{homogeneousA}
A(x):=B_0 x_1 \ dx_2.
\end{equation}
This scalar vector potential is invariant under $\vec{b}_2$-translations. The integrated vector potentials $\beta_{\vec{e}}$ are given by
\begin{equation}
\label{betas}
\beta_{\gamma_1, \gamma_2, \vec{f}}=  \frac{\Phi}{2}\left(\gamma_1+\frac{1}{6} \right),\ \beta_{\gamma_1, \gamma_2, \vec{g}}=0, \text{ and } \beta_{\gamma_1, \gamma_2, \vec{h}}=-\beta_{\gamma_1, \gamma_2, \vec{f}},
\end{equation}
where, in this case, the magnetic flux through each hexagon is $\Phi=\frac{3\sqrt{3}}{2}B_0$. 
\end{ex}

\subsection{Magnetic differential operator and modified Peierls' substitution}

In terms of the magnetic differential operator $(D^B \psi)_{\vec{e}}:= - i \psi_{\vec{e}}' - A_{\vec{e}}\psi_{\vec{e}} $, the 
Schr\"odinger operator modeling graphene in a magnetic field reads
\begin{equation}
\begin{split}
\label{magop}
H^B&:D(H^B) \subset L^2(\mathcal{E}\left(\Lambda\right)) \rightarrow L^2(\mathcal{E}\left(\Lambda\right))\\
(H^B& \psi)_{\vec{e}}:=(D^B D^B \psi)_{\vec{e}}+V_{\vec{e}}\psi_{\vec{e}},
\end{split}
\end{equation}
and is defined on
\begin{equation}
\begin{split}
\label{magdom}
D(H^B):= \Bigg\{&\psi \in \mathcal{H}^2(\mathcal{E}\left(\Lambda\right)): \psi_{\vec{e}_1}(v)=\psi_{\vec{e}_2}(v) \text{ for any } \vec{e}_1,\vec{e}_2 \in \mathcal{E}_{v}(\Lambda) \\
&\text{ and } \sum_{\vec{e} \in \mathcal{E}_{v}(\Lambda)} \left(D^B\psi\right)_{\vec{e}}(v) = 0 \Bigg\}.
\end{split}
\end{equation}

Let us first introduce a unitary operator $U$ on $L^2(\mathcal{E}(\Lambda))$, defined as
\begin{equation}
U\psi_{\gamma_1, \gamma_2, \vec{e}}=\zeta_{\gamma_1, \gamma_2}\psi_{\gamma_1, \gamma_2, \vec{e}}\ \text{  for  } \vec{e}=\vec{f},\vec{g},\vec{h},
\end{equation}
the factors $\zeta_{\gamma_1, \gamma_2}$ are defined as follows. 
First, choose a path $p(\cdot): \N\rightarrow \Z^2$ connecting $(0,0)$ to $(\gamma_1, \gamma_2)$ with
\begin{equation}
\label{path}
p(0)=(0,0)\ \text{  and  }\ p(|\gamma_1|+|\gamma_2|)=(\gamma_1, \gamma_2).
\end{equation}
Note that (\ref{path}) implies that both components of $p(\cdot)$ are monotonic functions.
Then we define $\zeta_{\gamma_1, \gamma_2}$ recursively through the following relations along $p(\cdot)$:
\begin{equation}
\begin{split}
\label{zetaH}
\zeta_{0, 0}&=1, \\
\zeta_{\gamma_1+1, \gamma_2}&=e^{i\beta_{\gamma_1, \gamma_2, \vec{f}}\ -i\beta_{\gamma_1+1, \gamma_2, \vec{g}}} \zeta_{\gamma_1, \gamma_2}, \\
\zeta_{\gamma_1, \gamma_2+1}&=e^{i\beta_{\gamma_1, \gamma_2, \vec{f}}\ -i\beta_{\gamma_1, \gamma_2+1, \vec{h}}\ -i\Phi \gamma_1}\zeta_{\gamma_1, \gamma_2}.
\end{split}
\end{equation}
Due to (\ref{constantfluxbeta}), it is easily seen that the definition of $\zeta_{\gamma_1, \gamma_2}$ is independent of the choice of $p(\cdot)$, hence is well-defined.

The unitary Peierls' substitution is the multiplication operator
\begin{equation}
\begin{split}
\label{Peierl}
P:& L^2(\mathcal{E}\left(\Lambda\right)) \rightarrow L^2(\mathcal{E}\left(\Lambda\right)) \\ 
&(\psi_{\vec{e}}) \mapsto \left( \left(\vec{e}\ni x\mapsto e^{i \int_{i(\vec{e}) \rightarrow x} A_{\vec{e}}(s)ds}\right) \psi_{\vec{e}} \right),
\end{split}
\end{equation}
where $i(\vec{e})\rightarrow x$ denotes the straight line connecting $i(\vec{e})$ with $x \in \vec{e}$.
It reduces the magnetic Schr\"odinger operator to non-magnetic ones with magnetic contribution moved into boundary condition, with multiplicative factors at terminal edges given by $e^{i\beta_{\vec{e}}}$.

We will define a modified Peierls' substitution that allows us to reduce the number of non-trivial multiplicative factors to one, by taking 
\begin{equation}
\begin{split}
\label{modifiedP}
\tilde{P}=PU.
\end{split}
\end{equation} 
It transforms $H^B$ into
\begin{equation}
\label{LambdaB}
\Lambda^B:=\left(- \frac{d^2}{dt_{\vec{e}}^2}+V_{\vec{e}} \right)_{\vec{e} \in \mathcal{E}(\Lambda)}={\tilde{P}}^{-1} H^B \tilde{P}.
\end{equation}
The domain of $\Lambda^B$ is
\begin{equation}
\begin{split}
\label{DomLambdaB}
D({\Lambda}^B)=\Bigg\{ &\psi\in \mathcal{H}^2(\mathcal{E}(\Lambda)):\ \text{any } \vec{e}_1, \vec{e}_2\in \mathcal{E}(\Lambda) \text{ with } i(\vec{e}_1)=i(\vec{e}_2)=v \text{ satisfy} \\
&\ \psi_{\vec{e}_1}(v)=\psi_{\vec{e}_2}(v) \text{ and } \sum_{i(\vec{e})=v} \psi^\prime_{\vec{e}}(v)=0; \text{ whilst at edges for which }
\\
& \ t(\vec{e}_1)=t(\vec{e}_2)=v, e^{i\tbe_{\vec{e}_1}}\psi_{\vec{e}_1}(v)=e^{i\tbe_{\vec{e}_2}}\psi_{\vec{e}_2}(v) \text{ and } \sum_{t(\vec{e})=v}e^{i\tbe_{\vec{e}}}\psi^\prime_{\vec{e}}(v)=0 \Bigg\},
\end{split}
\end{equation}
where 
\begin{align}\label{deftildebeta}
\tbe_{\gamma_1, \gamma_2, \vec{g}}\equiv \tbe_{\gamma_1, \gamma_2, \vec{f}}\equiv 0\ \text{ and }\ \tbe_{\gamma_1, \gamma_2, \vec{h}}=-\Phi \gamma_1.
\end{align}

Thus, the problem reduces to the study of non-magnetic Schr\"odinger operators with magnetic contributions moved into the boundary conditions.

Observe that the magnetic Dirichlet operator 
\begin{equation}
\begin{split}
\label{magopd}
H^{D,B}&:\bigoplus_{\vec{e} \in \mathcal{E}(\Lambda)}\left(\mathcal{H}_0^1(\vec{e})\cap \mathcal{H}^2(\vec{e}) \right) \subset L^2(\mathcal{E}\left(\Lambda\right)) \rightarrow L^2(\mathcal{E}\left(\Lambda\right)) \\
(H^{D,B}& \psi)_{\vec{e}}:=(D^B D^B \psi)_{\vec{e}}+V_{\vec{e}}\psi_{\vec{e}}
\end{split}
\end{equation}
is by the (modified) Peierls' substitution unitary equivalent to the Dirichlet operator without magnetic field
\begin{equation}
H^D=\tilde{P}^{-1}H^{D,B}\tilde{P}=P^{-1}H^{D,B}P.
\end{equation} 
Consequently, the spectrum of the Dirichlet operator $H^D$ is invariant under perturbations by the magnetic field.

\section{Main lemmas}\label{mainlemmassec}
First, let us introduce the following two-dimensional tight-binding Hamiltonian
\begin{equation}
\label{Q-op}
Q_{\Lambda}(\Phi) := \frac{1}{3} \left(\begin{matrix} 0 && 1+\tau_0+\tau_1 \\ \left(1 +\tau_0+\tau_1 \right)^* && 0 \end{matrix} \right)
\end{equation}
with translation operators $\tau_0, \tau_1 \in \mathcal{L}(l^2(\mathbb{Z}^2; \mathbb{C}))$ which for $\gamma \in \mathbb{Z}^2$ and $u \in l^2(\mathbb{Z}^2;\mathbb{C})$ are defined as
\begin{align}
\label{translatop}
(\tau_0(u))_{\gamma_1, \gamma_2}:= u_{\gamma_1-1,\gamma_2} \text{ and }
(\tau_1(u))_{\gamma_1, \gamma_2}:= e^{-i \Phi \gamma_1}u_{\gamma_1,\gamma_2-1}.
\end{align}

The following lemma connects the spectrum of $H^B$ with $\sigma(Q_{\Lambda})$. With $\Delta(\lambda)$ defined in (\ref{defDelta}), we have
\begin{lemm}\label{HBQLam}
A number $\lambda\in \rho(H^{D})$ lies in $\sigma(H^B)$ iff $\Delta(\lambda)\in \sigma(Q_{\Lambda}(\Phi))$. Such $\lambda$ is in the point spectrum of $H^B$ iff $\Delta(\lambda)\in \sigma_p(Q_{\Lambda}(\Phi))$.
\end{lemm}
\begin{rem}
We will show in Lemma \ref{spectrumQlambda} that $\sigma_p(Q_{\Lambda}(\Phi))$ is empty, thus $H^B$ has no point spectrum away from $\sigma(H^D)$.
\end{rem}

Lemma \ref{measure0cor} below shows $\sigma(Q_{\Lambda}(\Phi))$ is a zero-measure Cantor set for irrational flux $\frac{\Phi}{2\pi}$,
Lemma \ref{measureestlem} gives a measure estimate for rational flux,
and Lemma \ref{Hausdorff} provides an upper bound on the Hausdorff dimension of the spectrum of $Q_{\Lambda}(\Phi)$.
These three lemmas prove the topological structure part of Theorem \ref{T3}.

\begin{lemm}\label{measure0cor}
For $\frac{\Phi}{2\pi} \in \R\setminus \Q$, $\sigma(Q_{\Lambda}(\Phi))$ is a zero-measure Cantor set.
\end{lemm}

\begin{lemm}\label{measureestlem}
If $\frac{\Phi}{2\pi}=\frac{p}{q}$ is a reduced rational number, then we have
\begin{align*}
|\sigma(Q_{\Lambda}(\Phi))|\leq \frac{8\sqrt{6 \pi}}{9\sqrt{q}}.
\end{align*}
\end{lemm}

\begin{lemm}\label{Hausdorff}
For generic $\Phi$, the Hausdorff dimension of $\sigma(Q_{\Lambda}(\Phi))$ is $\leq \frac{1}{2}$.
\end{lemm}

\section{Proof of Lemma \ref{measure0cor}}\label{secproofofLemmakey2}
\subsection{Symmetric property of $Q_{\Lambda}$}
\begin{lemm}\label{Qsym}
The spectrum of $Q_{\Lambda}$ has the following properties:
\begin{enumerate}
    \item $\sigma(Q_{\Lambda}(\Phi))$ is symmetric with respect to $0$.
    \item $0\in \sigma(Q_{\Lambda}(\Phi))$.
\end{enumerate}
\end{lemm}
\begin{proof}
(1). Conjugating $Q_{\Lambda}$ in \eqref{Q-op} by
\begin{equation}
\Omega =
\left(\begin{matrix}
-\operatorname{id} & 0\\
0 & \operatorname{id}  \\
\end{matrix}\right)\end{equation}
shows that $\sigma(Q_{\Lambda}(\Phi))$ is symmetric with respect to $0$ \cite[Prop. $3.5$]{KL}.\bigbreak

(2). If we view $Q_{\Lambda}(\Phi)$ as an operator-valued function of the flux $\Phi,$ then 
\begin{equation}
\Phi \mapsto \langle Q_{\Lambda}(\Phi)x,y \rangle,
\end{equation}
for $x,y \in c_{00}$ arbitrary, is analytic and $Q_{\Lambda}$ therefore is a bounded analytic map. If there was ${\Phi_0}/{2\pi} \in \mathbb{R} \backslash \mathbb{Q}$ where $Q_{\Lambda}(\Phi_0)$ was invertible, then $Q_{\Lambda}(\Phi)$ would also be invertible in a sufficiently small neighborhood of $\Phi_0$ (e.g. \cite[Ch. $7.1$]{Ka}). Yet, in \cite[Prop. $4.1$]{HKL} it has been shown that for rational $\Phi/2\pi$, the Dirac points are in the spectrum, i.e. $0 \in \sigma(Q_{\Lambda}(\Phi)).$ Thus, by density $0 \in \sigma(Q_{\Lambda}(\Phi))$, independent of $\Phi \in \mathbb{R}.$
\end{proof}

\subsection{Reduction to the one-dimensional Hamiltonian}

\

Relating the spectrum of $Q_{\Lambda}$ to that of $Q_{\Lambda}^2$, we obtain the following characterization of $\sigma(Q_{\Lambda})$.
\begin{lemm}
\label{spectrumQlambda}
\begin{enumerate}
\item The spectrum of the operator $Q_{\Lambda}(\Phi)$ as a set is given by 
\begin{equation}\label{QLambdaHalpha}
\sigma(Q_{\Lambda}(\Phi))= \pm \sqrt{\frac{\bigcup_{\theta \in \mathbb{T}_1}\sigma(H_{\Phi, \theta})}{9}+\frac{1}{3}} \bigcup \left\{0\right\}.
\end{equation}
where $H_{\Phi, \theta} \in \mathcal{L}( l^2(\mathbb{Z}) )$ is the one-dimensional quasi-periodic Jacobi matrix defined as in (\ref{defH})
with 
\begin{align}\label{def:cv}
c(\theta)=1+e^{-2\pi i \theta},\ \  \text{and}\ \ \ v(\theta)=2\cos{2\pi \theta}.
\end{align}
\item $Q_{\Lambda}(\Phi)$ has no point spectrum.
\end{enumerate}
\end{lemm}

\begin{proof}
(1). Let $A:=\frac{1}{3}\left(1+\tau_0+\tau_1\right)$. Then squaring the operator $Q_{\Lambda}(\Phi)$ yields
\begin{equation}
\label{Diracoperator}
Q_{\Lambda}^2(\Phi)=\left(\begin{matrix} AA^* & 0 \\ 0 & A^*A \end{matrix} \right).
\end{equation}
The spectral mapping theorem implies that $\sigma(Q_{\Lambda}^2(\Phi))=\sigma(Q_{\Lambda}(\Phi))^2$ and from Lemma \ref{spectrumQlambda} we conclude that $\sigma(Q_{\Lambda}(\Phi))= \pm \sqrt{\sigma(Q_{\Lambda}^2(\Phi))}.$
Clearly, the operator $AA^* \vert_{\operatorname{ker}(A^*)^{\perp}}$ and $A^*A \vert_{\operatorname{ker}(A)^{\perp}}$ are unitarily equivalent. Thus, the spectrum can be expressed by 
\begin{equation}
\sigma(Q_{\Lambda}(\Phi))=\pm \sqrt{\sigma(AA^*)} \cup \left\{0\right\}
\end{equation} where we are able to use either of the two ($AA^*$ or $A^*A$) since $0 \in \sigma(Q_{\Lambda}(\Phi))$ due to Lemma \ref{Qsym}.  

Then, it follows that
\begin{align}
AA^* &= \frac{\operatorname{id}}{3}+\frac{\operatorname{id}}{9}\underbrace{\left((\tau_0+\tau_0^*)+(\tau_1+\tau_1^*)+\tau_0\tau_1^* + \tau_1\tau_0^*\right)}_{=:H_{\Phi}}.
\end{align}
Observe that 
\begin{align}
H_{\Phi}\psi_{m,n}=
&\psi_{m-1,n}+\psi_{m+1,n}+e^{- i \Phi m } \psi_{m,n-1}+e^{i \Phi m} \psi_{m,n+1} \nonumber\\
&+e^{ i \Phi (m-1)} \psi_{m-1,n+1}+e^{-i \Phi m}\psi_{m+1,n-1}.
\end{align}
Since $H_{\Phi}$ is invariant under discrete translations in $n$, the operator is unitarily equivalent to the direct integral operator
\begin{equation}
\int_{\mathbb{T}_1}^{\oplus} H_{\Phi, \theta}\ d \theta, 
\end{equation}
which gives the claim. \bigbreak

(2). It follows from a standard argument that the two dimensional operator $H_{\Phi}$ has no point spectrum.
Indeed, assume $H_{\Phi}$ has point spectrum at energy $E$, then $H_{\Phi, \theta}$ would have the same point spectrum $E$ for a.e. $\theta\in \T_1$. 
This implies the integrated density of states of $H_{\Phi, \theta}$ has a jump discontinuity at $E$, which is impossible. 
Therefore the point spectrum of $H_{\Phi}$ is empty, hence the same holds for $Q_{\Lambda}(\Phi)$.
\end{proof}

Lemma \ref{measure0cor} follows as a direct consequence of (\ref{QLambdaHalpha}) and the following Theorem \ref{measure0thm}. 
Let $\Sigma_{\Phi}$ be defined as in Section \ref{preJacobi}.
\begin{theo}\label{measure0thm}
For $\frac{\Phi}{2\pi} \in \R\backslash \Q$, $\Sigma_{\Phi}$ is a zero-measure Cantor set.
\end{theo}
We will prove Theorem \ref{measure0thm} in the next section.

\section{Proof of Lemmas \ref{measureestlem},   \ref{Hausdorff}, and Theorem \ref{measure0thm}}\label{secHausdorff}
For a set $U$, let $\mathrm{dim}_H(U)$ be its Hausdorff dimension.

We will need the following three lemmas.

First, we have
\begin{lemm}\label{rationalest}
Let $\frac{\Phi}{2\pi}=\frac{p}{q}$ be a reduced rational number, then $\Sigma_{2\pi p/q}$ is a union of $q$ (possibly touching) bands with $|\Sigma_{2\pi p/q}|<\frac{16\pi}{3q}$.
\end{lemm}

Lemma \ref{rationalest} will be proved in subsections \ref{even} and \ref{odd} after some further preparation. The following lemma addresses the continuity of the spectrum $\Sigma_{\Phi}$ in $\Phi$, extending a result of \cite{AvMS} (see Proposition 7.1 therein) from quasiperiodic Schr\"odinger operators to Jacobi matrices.
\begin{lemm}\label{continuity}
There exists absolute constants $C_1, C_2>0$ such that if $\lambda \in \Sigma_{\Phi}$ and $|\Phi-\Phi^\prime|<C_1$, then there exists $\lambda^\prime\in \Sigma_{\Phi^\prime}$ such that
\begin{align*}
|\lambda-\lambda^\prime| \le C_2|\Phi-\Phi'|^{\frac{1}{2}}.
\end{align*}
\end{lemm}
We will prove Lemma \ref{continuity} in the appendix.

The next Lemma gives a way to bound the Hausdorff dimension from above.
\begin{lemm}\label{Last:lemma}(Lemma 5.1 of \cite{Last})
Let $S\subset \R$, and suppose that $S$ has a sequence of covers: $\{S_n\}_{n=1}^{\infty}$, $S\subset S_n$, such that each $S_n$ is a union of $q_n$ intervals, $q_n\to \infty$ as $n\to \infty$, and for each $n$, 
\begin{align*}
|S_n|<\frac{C}{q_n^{\beta}},
\end{align*}
where $\beta$ and $C$ are positive constants, then 
\begin{align*}
\mathrm{dim}_H(S)\leq \frac{1}{1+\beta}.
\end{align*}
\end{lemm}

\subsection*{Proof of Lemma \ref{measureestlem}}

This is a quick consequence of Lemma \ref{rationalest}.
It is clear that for any $\epsilon>0$, we have
\begin{align*}
\sqrt{\Sigma_{2 \pi p/q}+3}\subseteq [\ 0, \sqrt{\epsilon}\ ]\bigcup \sqrt{\left(\Sigma_{2 \pi p/q}+3\right) \bigcap\ (\epsilon, \infty)}.
\end{align*}
Hence by Lemma \ref{rationalest}, we have
\begin{align*}
|\sqrt{\Sigma_{2 \pi p/q}+3}|\leq \sqrt{\epsilon}+\frac{|\Sigma_{2\pi p/q}|}{2\sqrt{\epsilon}}\leq \sqrt{\epsilon}+\frac{8\pi}{3\sqrt{\epsilon} q}.
\end{align*}
Optimizing in $\epsilon$ leads to 
\begin{align*}
|\sqrt{\Sigma_{2 \pi p/q}+3}|\leq \frac{4\sqrt{6\pi}}{3\sqrt{q}}.
\end{align*}
Then \eqref{QLambdaHalpha} implies
\begin{align}
|\sigma(Q_{\Lambda}(2\pi p/q))| \leq \frac{8\sqrt{6 \pi}}{9\sqrt{q}}. 
\end{align}
\begin{flushright}
\qed
\end{flushright}

\subsection*{Proof of Lemma \ref{Hausdorff}}
We will show that if $\frac{\Phi}{2\pi}$ is an irrational obeying
\begin{align}\label{genericcondition}
q_n^4 \left| \frac{\Phi}{2\pi}-\frac{p_n}{q_n}\right| <C,
\end{align}
for some constant $C$, and a sequence of reduced rationals $\{p_n/q_n\}$ with $q_n\to \infty$, then $\dim_H(\sigma(Q_{\Lambda}(\Phi)))\leq 1/2$.

Without loss of generality, we may assume $\frac{\Phi}{2\pi}\in (0,1)$.

First, by \eqref{QLambdaHalpha}, we have that 
\begin{align*}
\mathrm{dim}_H(\sigma(Q_{\Lambda}(\Phi)))=\sup_{k\geq 2}\ \mathrm{dim}_H\left(\pm \sqrt{\left(\frac{\Sigma_{\Phi}}{9}+\frac{1}{3}\right)\cap [\frac{1}{k}, 1]}\right),
\end{align*}
where we used a trivial bound $\|H_{\Phi, \theta}\|\leq 6.$
Hence it suffices to show that for each $k\geq 2$, 
\begin{align}\label{HDk}
\mathrm{dim}_H\left(\sqrt{\left(\frac{\Sigma_{\Phi}}{9}+\frac{1}{3}\right)\cap [\frac{1}{k}, 1]}\right)\leq \frac{1}{2}.
\end{align}

The rest of the argument is similar to that of \cite{Last}. By Lemma \ref{continuity}, taking any $\lambda\in \Sigma_{\Phi}$, for $n\geq n_0$, there exists $\lambda^\prime\in \Sigma_{2\pi p_n/q_n}$ such that 
$|\lambda-\lambda^\prime|\le \tilde{C}_2 |\frac{\Phi}{2\pi}-\frac{p_n}{q_n}|^{\frac{1}{2}}$. 
This means $\Sigma_{\Phi}$ is contained in the $\tilde{C}_2 |\frac{\Phi}{2\pi}-\frac{p_n}{q_n}|^{\frac{1}{2}}$ neighbourhood of $\Sigma_{2\pi p_n/q_n}$.
By Lemma \ref{rationalest}, $\Sigma_{2\pi p_n/q_n}$ has $q_n$ (possibly touching) bands with total measure $|\Sigma_{2\pi p_n/q_n}|\leq \frac{16\pi}{3q_n}$.
Hence 
$\Sigma_{\Phi}$ has cover $S_n$ such that $S_n$ is a union of (at most) $q_n$ intervals with total measure
\begin{align}\label{measureestSn}
|S_n|\leq \frac{16\pi}{3q_n}+ 2\tilde{C}_2 q_n \left|\frac{\Phi}{2\pi}-\frac{p_n}{q_n}\right|^{\frac{1}{2}}.
\end{align}
Since $q_n^4 \left|\frac{\Phi}{2\pi}-\frac{p_n}{q_n}\right|\leq C$, we have, by \eqref{measureestSn},
\begin{align}
|S_n|\leq \frac{16\pi}{3q_n}+\frac{2C_2\sqrt{C}}{q_n}=:\frac{\tilde{C}}{q_n}.
\end{align}
This implies $\left(\frac{\Sigma_{\Phi}}{9}+\frac{1}{3} \right)\cap [\frac{1}{k}, 1]$ has cover $\tilde{S}_n$ such that $\tilde{S}_n$ is a union of (at most) $q_n$ intervals with total measure
\begin{align}\label{measureesttildeSn}
|\tilde{S}_n|\leq \frac{\sqrt{k}\tilde{C}}{2q_n}.
\end{align}
Then Lemma \ref{Last:lemma} yields \eqref{HDk}.
\qed
\subsection{Proof of Theorem \ref{measure0thm}}

\
Note that Lemmas \ref{rationalest} and \ref{continuity} already imply zero measure (and thus Cantor nature) of the spectrum for fluxes $\alpha$ with unbounded coefficients in the continued fraction expansion, thus for a.e. $\alpha,$ by an argument similar to that used in the proof of Lemma \ref{Hausdorff}.  However extending the result to the remaining measure zero set this way would require a slightly stronger continuity in Lemma \ref{continuity}, which is not available. We circumvent this by combining quantization of localization-type arguments, singularity-induced absence of absolutely continuous spectrum, and Kotani theory for Jacobi matrices.

Let $\Sigma_{ac}(H_{\Phi, \theta})$ be the absolutely continuous spectrum of $H_{\Phi, \theta}$. 
Let $L(\lambda, \Phi)$ be the Lyapunov exponent of $H_{\Phi, \theta}$ at energy $\lambda$, as defined in \eqref{defLE}.
For a set $U\subseteq \R$, let $\overline{U}^{ess}$ be its essential closure.

First, we are able to give a characterization of the Lyapunov exponent on the spectrum.
\begin{prop}\label{LE=0}
For $\frac{\Phi}{2\pi}\in \R\setminus \Q$, 
$L(\lambda, \Phi)=0$ if and only if $\lambda\in \Sigma_{\Phi}$.
\end{prop}
The proof of this is similar to that for the almost Mathieu operator as given in \cite{Global} and the extended Harper's model \cite{JM}.
The general idea is to complexify $\theta$ to $\theta+i\epsilon$, and obtain asymptotic behavior of the Lyapunov exponent when $|\epsilon|\to\infty$.
Convexity and quantization of the acceleration (see Theorem 5 of \cite{Global}) then bring us back to the $\epsilon=0$ case.
We will leave the proof to the appendix.

Exploiting the fact that $c(\theta)=1+e^{-2\pi i \theta}$ has a real zero $\theta_1=\frac{1}{2}$, we have
\begin{prop}\label{sing}(\cite{JDom}, see also Proposition 7.1 of \cite{JM})
For $\frac{\Phi}{2\pi}\in \R\setminus \Q$, and a.e. $\theta\in \T_1$, $\Sigma_{ac}(H_{\Phi, \theta})$ is empty.
\end{prop}

Hence our operator $H_{\Phi, \theta}$ has zero Lyapunov exponent on the spectrum and empty absolutely continuous spectrum. Celebrated Kotani theory identifies the essential closure of the set of zero Lyapunov exponents with the absolutely continuous spectrum, for general ergodic Schr\"odinger operators. This has been extended to the case of non-singular (that is $|c(\cdot)|$  uniformly bounded away from zero) Jacobi matrices in Theorem 5.17 of \cite{GTeschl}. In our case $|c(\cdot)|$ is not  bounded away from zero, however a careful inspection of the proof of Theorem 5.17 of \cite{GTeschl} shows that it holds under a weaker requirement: $\log{(|c(\cdot)|)} \in L^1$.
Namely, let $H_{c,v}(\theta)$ acting on $\ell^2(\Z)$ be an ergodic Jacobi matrix, 

$$ (H_{c,v}(\theta)u)_m=c(T^m\theta)u_{m+1}+\overline{c(T^{m-1}\theta)}u_{m-1}+v(T^m\theta)u_m$$

where $c:M\to \C$, $v:M\to \R$, are bounded measurable functions, and $T:M\to M$ is an ergodic map. Let $L_{c,v}(\lambda)$ be the corresponding Lyapunov exponent. We have 
\begin{theo}\label{Kotani}(Kotani theory) Assume $\log{(|c(\cdot)|)} \in L^1(M)$.
Then for a.e. $\theta\in M$, $\Sigma_{ac}(H_{c,v}(\theta))=\overline{\{\lambda:\ L_{c,v}(\lambda)=0\}}^{ess}$.
\end{theo}
{\bf Proof.} The proof of Theorem 5.17 of \cite{GTeschl} works verbatim. \qed\\
In our concrete model, $\log{(|c(\theta)|)}=\log{(2|\cos{\pi\theta}|)}\in L^1(\T_1)$, thus Theorem \ref{Kotani} applies, and combining with Propositions \ref{LE=0}, \ref{sing}, it follows that 
$\Sigma_{\Phi}$ must be a zero measure set.

\qed

The rest of this section will be devoted to proving Lemma \ref{rationalest}.

\subsection{Quick Observations about $H_{2\pi p/q, \theta}$}
\

Let $A^{\lambda}(\cdot), \tilde{A}^{\lambda} (\cdot), D^{\lambda}(\cdot), \Theta$ be defined as in Section \ref{deftransfersec}.
We start with several quick observations about $H_{2\pi p/q, \theta}$.
\begin{obs}
The sampling function $c(\theta)=0$ yields a unique solution $\theta=\frac{1}{2}$ \text{(}mod $1$\text{)}, hence $\Theta=\frac{1}{2}+\frac{1}{q}\Z$. 
Then, 
\begin{itemize}
\item for $\theta\notin \Theta$, we have $c(\theta+n\frac{p}{q})\neq 0$ for any $n\in \Z$
\item for $\theta\in \Theta$, there exists $k_0\in \{0,1,...,q-1\}$ such that $c(\theta+n\frac{p}{q})=0$ if and only if $n\equiv k_0$ (mod $q$).
\end{itemize}
\end{obs}

Note that $|c(\theta)|=2|\cos{\pi\theta}|$, so a simple computation yields that $\prod_{j=0}^{q-1}|c(\theta+j\frac{p}{q})|=2|\sin{\pi q(\theta+\frac{1}{2})}|$. Thus (\ref{tildeAqDq}) becomes 
\begin{align}\label{tildeAqDq2}
\tr(\tilde{A}^{\lambda}_q(\theta))=\frac{\tr(D^{\lambda}_q(\theta))}{2|\sin{\pi q(\theta+\frac{1}{2})}|}.
\end{align}

We have the following characterization of $\Sigma_{2\pi p/q, \theta}$.
\

\subsubsection{Case 1.} 
If $\theta\in \Theta$, we have the following
\begin{obs}
For $\theta\in \Theta$, the infinite matrix $H_{2\pi p/q, \theta}$ is decoupled into copies of the following block matrix $M_q$ of size $q$:
\begin{align}\label{Thetaqblock}
\left(
\begin{matrix}
&v(\frac{1}{2})\ \ &\overline{c(\frac{1}{2}-\frac{p}{q})}\ \ &\ \ &\ \ & \ \ &\ \ &\\
&c(\frac{1}{2}-\frac{p}{q})\ \ &v(\frac{1}{2}-\frac{p}{q})\ \ &\ddots\ \ &\ \ & \ \ &\ \ &\\
&\ddots    &  &\ddots & & & &\\
& & & & & & &\\
\\
\\
& & & & & &v(\frac{1}{2}-(q-2)\frac{p}{q}) &\overline{c(\frac{1}{2}-(q-1)\frac{p}{q})}\\
& & & & & &c(\frac{1}{2}-(q-1)\frac{p}{q}) &v(\frac{1}{2}-(q-1)\frac{p}{q})
\end{matrix}
\right)
\end{align}
Thus 
\begin{align}\label{sigma=det}
\Sigma_{2\pi p/q, \theta}=\{\text{eigenvalues of }M_q\},\ \ \ \text{for } \theta\in \Theta. 
\end{align}
\end{obs}

\subsubsection{Case 2.} If $\theta\notin \Theta$, by Floquet theory, we have
\begin{align}\label{Spq1}
\Sigma_{2\pi p/q, \theta}=\{\lambda:\ |\tr(\tilde{A}^{\lambda}_{\frac{p}{q}, q}(\theta))| \leq 2 \}.
\end{align}
Furthermore, the set $\{\lambda : \tr \tilde{A}^{\lambda}_{\frac{p}{q}, q}(\theta)=2\cos{2\pi \nu}\}$ contains $q$ individual points (counting multiplicities), which are eigenvalues of the following $q\times q$ matrix $M_{q, \nu}$:
\begin{align}\label{defMqbeta}
M_{q, \nu}(\theta)=
\left(
\begin{matrix}
v(\theta) &|c(\theta-\frac{p}{q})| & & &  & &e^{2\pi i \nu} |c(\theta)|\\
|c(\theta-\frac{p}{q})|   &v(\theta-\frac{p}{q})  &\ddots  &          &           & &\\
 &\ddots &\ddots  & &\ & & &\
\\
\\
 & &  & & &v(\theta-(q-2)\frac{p}{q}) &|c(\theta-(q-1)\frac{p}{q})|\\
e^{-2\pi i \nu} |c(\theta)|         &          &               &                      &     &|c(\theta-(q-1)\frac{p}{q})|   &v(\theta-(q-1)\frac{p}{q})
\end{matrix}
\right)
\end{align}

Combining (\ref{Spq1}) with (\ref{tildeAqDq2}), we arrive at an alternative representation
\begin{align}\label{Spq2}
\Sigma_{2\pi p/q, \theta}=\left\lbrace \lambda : |\tr(D^{\lambda}_q(\theta)) | \leq 4|\sin{\pi q (\theta+\frac{1}{2})}|\right\rbrace .
\end{align}

\subsection{Key lemmas}\label{chamberssec}

\

Let
\begin{align}\label{defdq}
d_q(\theta)=\tr(D^{\lambda}_q(\theta)).
\end{align}
We have
\begin{lemm}
(a Chambers' type formula)
For all $\theta\in \T_1$, we have
\begin{align}\label{dqChambers}
d_q(\theta)=-2\cos{2\pi q\theta}+G_q (\lambda),
\end{align}
where $G_q(\lambda)$ (defined by (\ref{dqChambers})) is independent of $\theta$.
\end{lemm}
\begin{proof}
It is easily seen that $d_q(\cdot)$ is a $1/q$-periodic function, thus 
\begin{align*}
d_q(\theta)=G_q(\lambda)+a_q e^{2\pi i q\theta}+a_{-q} e^{-2\pi i q\theta},
\end{align*}
in which the $G_q(\lambda)$ part is independent of $\theta$.
One can easily compute the coefficients $a_q, a_{-q}$, and get $a_q=a_{-q}=1$.
\end{proof}

\begin{lemm}\label{det=tr}
For $\theta\in \Theta$, 
\begin{align}
\det{(\lambda\cdot \mathrm{Id} -M_q(\theta))}=\tr(D^{\lambda}_q(\theta)).
\end{align}
\end{lemm}
The proof of this lemma will be included in the appendix.

Combining (\ref{Spq2}), (\ref{sigma=det}) and Lemma \ref{det=tr} with the fact that $|\sin{\pi q(\theta+\frac{1}{2})}|=0$ for $\theta\in \Theta$, we arrive at
\begin{align}
\Sigma_{2\pi p/q, \theta}=\{\lambda:\ |\tr(D^{\lambda}_q(\theta))|\leq 4|\sin{\pi q(\theta+\frac{1}{2})}| \}
\end{align}
holds {\it uniformly} for $\theta\in \T_1$.

By Lemma \ref{dqChambers}, we get the following alternative characterization of $\Sigma_{2\pi p/q, \theta}$.
\begin{align}\label{Spq3}
\Sigma_{2\pi p/q, \theta}=\left\lbrace \lambda : -4|\sin{\pi q(\theta+\frac{1}{2})}|+2\cos{2\pi q\theta}\leq G_q (\lambda)\leq 4|\sin{\pi q(\theta+\frac{1}{2})}|+2\cos{2\pi q\theta} \right\rbrace.
\end{align}
Let us denote $L_q(\theta):=4|\sin{\pi q(\theta+\frac{1}{2})}|+2\cos{2\pi q\theta}$, and $l_q(\theta):=-4|\sin{\pi q(\theta+\frac{1}{2})}|+2\cos{2\pi q\theta}$. 
Then (\ref{Spq3}) translates into
\begin{align}\label{Spq5}
\Sigma_{2\pi p/q, \theta}=\{ \lambda : l_q(\theta)\leq G_q(\lambda)\leq L_q(\theta)\}.
\end{align}
This clearly implies
\begin{align}\label{Spq4}
\Sigma_{2\pi p/q}=\{ \lambda : \min_{\T_1} l_q(\theta) \leq G_q (\lambda) \leq \max_{\T_1} L_q(\theta) \}.
\end{align}

Since $G_q(\lambda)$ is a polynomial of $\lambda$ of degree $q$, (\ref{Spq4}) implies $\Sigma_{2\pi p/q}$ has $q$ (possibly touching) bands.

The following lemma provides estimates of $|\Sigma_{2\pi p/q, \theta}|$.
\begin{lemm}\label{estlemma}
We have
\begin{align*}
|\Sigma_{2\pi p/q, \theta}|\leq 4|c(\theta)|.
\end{align*}
\end{lemm}
\begin{proof}
Let us point out that, due to (\ref{Spq1}) and the explanation below it, 
\begin{align}\label{Spq6}
\begin{cases}
\{\lambda:\ G_q(\lambda)=L_q(\theta)\}=\{\text{eigenvalues of }M_{q, 0}(\theta)\} \\
\{\lambda:\ G_q(\lambda)=l_q(\theta)\}\ =\{\text{eigenvalues of }M_{q, \frac{1}{2}}(\theta)\}
\end{cases}
\end{align}

Let $\{\lambda_i(\theta)\}_{i=1}^{q}$ be eigenvalues of $M_{q, 0}(\theta)$, labelled in the increasing order. 
Let $\{\tilde{\lambda}_i{(\theta)}\}_{i=1}^q$ be eigenvalues of $M_{q, \frac{1}{2}}(\theta)$, labelled also in the decreasing order.
Then by (\ref{Spq5}) and (\ref{Spq6}),
\begin{align}\label{Spqevenqest2}
|\Sigma_{\frac{2\pi p}{q}, \theta}|
=&\sum_{k=1}^q (-1)^{q-k} \left( \lambda_k (\theta)-\tilde{\lambda}_k{(\theta)} \right)\\
=&\sum_{k=1}^{[\frac{q+1}{2}]} \left(\lambda_{q-2k+2}{(\theta)}-\tilde{\lambda}_{q-2k+2}{(\theta)}\right) - \sum_{k=1}^{[\frac{q-1}{2}]} \left(\lambda_{q-2k+1}{(\theta)}- \tilde{\lambda}_{q-2k+1}{(\theta)} \right). \notag
\end{align}

Consider the difference matrix
\begin{align*}
M_{q, 0}(\theta)-M_{q, \frac{1}{2}}(\theta)
=
\left(
\begin{matrix}
& & & &  & &2|c(\theta)|\\
 & &  & & & & &
\\
\\
 & &  & & & &\\
2|c(\theta)|        &          &               &                      &     &   &
\end{matrix}
\right)
\end{align*}
whose eigenvalues we denote by $\{E_i(\theta)\}_{i=1}^{q}$, namely,
\begin{align*}
E_{1}(\theta)=-2|c(\theta)|< 0=E_{2}(\theta)=\cdots =E_{q-1}(\theta)=0< 2|c(\theta)|=E_{q}{(\theta)}.
\end{align*}

By Lidskii inequalities (\ref{Lid}), we have
\begin{align}\label{qevenlid1}
\sum_{k=1}^{[\frac{q+1}{2}]} \lambda_{q-2k+2}{(\theta)} \leq \sum_{k=1}^{[\frac{q+1}{2}]} \tilde{\lambda}_{q-2k+2}{(\theta)}
+\sum_{k=1}^{[\frac{q+1}{2}]} E_{q-k+1}{(\theta)} =\sum_{k=1}^{[\frac{q+1}{2}]} \tilde{\lambda}_{q-2k+2}{(\theta)}+2|c(\theta)|,
\end{align}
and
\begin{align}\label{qevenlid2}
\sum_{k=1}^{[\frac{q-1}{2}]}\lambda_{q-2k+1}{(\theta)} \geq \sum_{k=1}^{[\frac{q-1}{2}]}\tilde{\lambda}_{q-2k+1}{(\theta)} 
+\sum_{k=1}^{[\frac{q-1}{2}]} E_{k}{(\theta)}= \sum_{k=1}^{[\frac{q-1}{2}]}\tilde{\lambda}_{q-2k+1}{(\theta)}  -2|c(\theta)|.
\end{align}

Hence combining (\ref{Spqevenqest2}) with (\ref{qevenlid1}) (\ref{qevenlid2}), we get,
\begin{align}
|\Sigma_{\frac{2\pi p}{q}, \theta}|\leq 4|c(\theta)|.
\end{align}
\end{proof}

\subsection{Proof of Lemma \ref{rationalest} for even $q$}\label{even}\

For sets/functions that depend on $\theta$, we will sometimes substitute $\theta$ in the notation with $A\subseteq \T_1$, if corresponding sets/functions are constant on $A$.

Since $q$ is even, 
a simple computation shows 
\begin{align*}
\begin{cases}
\max_{\T_1} L_q(\theta)=L_q(\frac{6\Z+1}{6q})=L_q(\frac{6\Z+5}{6q})=3,\\
\\
\min_{\T_1} l_q(\theta)=l_q(\frac{2\Z+1}{2q})=-6.
\end{cases}
\end{align*} 
A simple computation also shows $l_q(\frac{6\Z+1}{6q})=-1$ and $L_q(\frac{2\Z+1}{2q})=2$.
Thus we have, by (\ref{Spq4}),
\begin{align*}
\Sigma_{2\pi p/q}
=&\{\lambda : -6\leq G_q(\lambda)\leq 3\}\\
=&\{\lambda : -6\leq G_q(\lambda)\leq 2\} \bigcup \{\lambda : -1\leq G_q(\lambda)\leq 3\}\\
=&\Sigma_{\frac{2\pi p}{q}, \frac{2\Z+1}{2q}}\bigcup \Sigma_{\frac{2\pi p}{q}, \frac{6\Z+1}{6q}}.
\end{align*}
This implies
\begin{align}\label{Spqevenqest1}
|\Sigma_{2\pi p/q}| \leq |\Sigma_{\frac{2\pi p}{q}, \frac{2\Z+1}{2q}}|+|\Sigma_{\frac{2\pi p}{q}, \frac{6\Z+1}{6q}}|.
\end{align}
Now it remains to estimate $|\Sigma_{\frac{2\pi p}{q}, \frac{2\Z+1}{2q}}|$ and $|\Sigma_{\frac{2\pi p}{q}, \frac{6\Z+1}{6q}}|$.
Since $q$ is even, let us consider $\Sigma_{\frac{2\pi p}{q}, \frac{q+1}{2q}}$ and $\Sigma_{\frac{2\pi p}{q}, \frac{1}{6q}}$.

By Lemma \ref{estlemma}, we have
\begin{align}\label{Spqevenqest3}
\begin{cases}
|\Sigma_{\frac{2\pi p}{q}, \frac{q+1}{2q}}|\leq 4|c(\frac{q+1}{2q})|<\frac{4\pi}{q},\\
|\Sigma_{\frac{2\pi p}{q}, \frac{1}{6q}}|\leq 4|c(\frac{1}{6q})|< \frac{4\pi }{3q}.
\end{cases}
\end{align}
Hence putting \eqref{Spqevenqest1}), \eqref{Spqevenqest3} together, we have
\begin{align}
|\Sigma_{\frac{2\pi p}{q}}|< \frac{16\pi}{3q}.
\end{align}

\subsection{Proof of Lemma \ref{rationalest} for odd $q$}\label{odd}\

Since the proof for odd $q$ is very similar to that for even $q$, we only sketch the steps here.\


For odd $q$, similar to \eqref{Spqevenqest1}, we have



\begin{align}\label{Spqoddqest1}
|\Sigma_{\frac{2\pi p}{q}}|\leq |\Sigma_{\frac{2\pi p}{q}, \frac{3\Z+1}{3q}}|+|\Sigma_{\frac{2\pi p}{q}, \frac{\Z}{q}}|.
\end{align}
By Lemma \ref{estlemma}, we have
\begin{align}\label{Spqoddqest3}
\begin{cases}
|\Sigma_{\frac{2\pi p}{q}, \frac{3q-1}{6q}}| \leq 4|c(\frac{3q-1}{6q})|< \frac{4\pi}{3q},\\
|\Sigma_{\frac{2\pi p}{q}, \frac{q-1}{2q}}|\leq 4|c(\frac{q-1}{2q})| < \frac{4\pi }{q}.
\end{cases}
\end{align}
Hence putting (\ref{Spqoddqest1}), (\ref{Spqoddqest3}) together, we have
\begin{align}
|\Sigma_{\frac{2\pi p}{q}}|< \frac{16\pi}{3q}.
\end{align}
$\hfill{} \Box$

\section{Proof of Lemma \ref{HBQLam}}\label{KreinRed}
Using ideas from \cite{P} and \cite{BGP}, we can express the resolvent of the operator $\Lambda^B$ \eqref{LambdaB} by Krein's resolvent formula in terms of the resolvent of the Dirichlet Hamiltonian and the resolvent of $Q_{\Lambda}$.

For this we need to introduce a few concepts first.
The $l^2$-space on the vertices $l^2(\mathcal{V}(\Lambda))$ carries the inner product
\begin{equation}
\langle f,g \rangle :=\sum_{v \in \mathcal{V}(\Lambda)} 3 f(v) \overline{g(v)}
\end{equation}
where the factor three accounts for the number of incoming or outgoing edges at each vertex.

A convenient method from classical extension theory required to state Krein's resolvent formula, and thus to link the magnetic Schr\"odinger operator $H^B$ with an effective Hamiltonian, is the concept of boundary triples.
\begin{defi}
\label{triple}
Let $\mathcal{T}: D(\mathcal{T}) \subset \mathscr{H} \rightarrow \mathscr{H}$ be a closed linear operator on the Hilbert space $\mathscr{H}$, then the triple $(\pi,\pi',\mathscr{H}'),$ with $\mathscr{H}'$ being another Hilbert space and $\pi, \pi': D(\mathcal{T})\rightarrow \mathscr{H}'$, is a boundary triple for $\mathcal{T}$, if 
\begin{itemize}
\item Green's identity holds on $D(\mathcal{T}),$ i.e. for all $\psi,\varphi \in D(\mathcal{T})$
\begin{equation}
 \langle \psi,\mathcal{T}\varphi \rangle_{\mathscr{H}} - \langle \mathcal{T} \psi, \varphi \rangle_{\mathscr{H}}  = \langle \pi \psi,\pi'\varphi \rangle_{\mathscr{H}'}-\langle \pi' \psi,\pi\varphi \rangle_{\mathscr{H}'}.
\end{equation}
\item $\operatorname{ker}(\pi,\pi')$ is dense in $\mathscr{H}.$
\item $(\pi,\pi'):D(\mathcal{T}) \rightarrow \mathscr{H}' \oplus  \mathscr{H}'$ is a linear surjection.
\end{itemize}
\end{defi}
The following lemma applies this concept to our setting.
\begin{lemm}
\label{TB-Lemm}
The operator $\mathcal{T}^B:D(\mathcal{T}^B) \subset L^2(\mathcal{E}\left(\Lambda\right)) \rightarrow L^2(\mathcal{E}\left(\Lambda\right))$ acting as the maximal Schr\"odinger operator \eqref{maximaloperator} on every edge with domain
\begin{align}\label{DommaximalTB}
D(\mathcal{T}^B):= \Bigg\{\psi \in \mathcal{H}^2(\mathcal{E}\left(\Lambda\right)): &\text{ any }\vec{e}_1, \vec{e}_2 \in \mathcal{E}_v(\Lambda) \text{ such that } i(\vec{e}_1)=i(\vec{e}_2)=v \text{ satisfy } \nonumber\\
& \ \psi_{\vec{e}_1}(v)=\psi_{\vec{e}_2}(v) \text{ and if  }t(\vec{e}_1)=t(\vec{e}_2)=v, \nonumber\\
&\text{ then } e^{  i\tbe_{\vec{e}_1}}\psi_{\vec{e}_1}(v)=e^{ i\tbe_{\vec{e}_2} }\psi_{\vec{e}_2}(v)\Bigg\}
\end{align}
is closed. The maps $\pi,\pi'$ on $D(\mathcal{T}^B)$ defined by
\begin{align}
\pi(\psi)(v)&:=\frac{1}{3}
  \Bigg(\sum_{i(\vec{e})=v} \psi_{\vec{e}}(v)+\sum_{t(\vec{e})=v} e^{ i\tbe_{\vec{e}}} \psi_{\vec{e}}(v)\Bigg)
 \nonumber\\
\pi'(\psi)(v)&:= \frac{1}{3}
  \Bigg(\sum_{i(\vec{e})=v} \psi_{\vec{e}}'(v)-\sum_{t(\vec{e})=v} e^{ i\tbe_{\vec{e}}} \psi_{\vec{e}}'(v) \Bigg)
\end{align}
form together with $\mathscr{H}':=l^2(\mathcal{V}(\Lambda))$ a boundary triple associated to $\mathcal{T}^B$. 
\end{lemm}
\begin{proof}
The proof follows the same strategy as in \cite{P}.
The operator $\mathcal{T}^B$ is closed iff its domain is a closed subspace (with respect to the graph norm) of the domain of some closed extension of $\mathcal{T}^B$. 
Such a closed extension is given by $\bigoplus_{\vec{e} \in \mathcal{E}(\Lambda)}\mathcal{H}_{\vec{e}}$ on $\mathcal{H}^2(\mathcal{E}\left(\Lambda\right)).$ 
To see that $D(\mathcal{T}^B)$ is a closed subspace of $\mathcal{H}^2(\mathcal{E}\left(\Lambda\right)),$ observe that in terms of continuous functionals
\begin{align}
l_{\vec{e}_i, \vec{e}_j}&: \mathcal{H}^2(\mathcal{E}\left(\Lambda\right)) \rightarrow \mathbb{C},\ \ l_{\vec{e}_i, \vec{e}_j}(\psi)= \psi_{\vec{e}_i}(i(\vec{e}_i))-\psi_{\vec{e}_j}(i(\vec{e}_j)) \nonumber \\
k_{\vec{e}_i, \vec{e}_j}&: \mathcal{H}^2(\mathcal{E}\left(\Lambda\right)) \rightarrow \mathbb{C},\ \ k_{\vec{e}_i, \vec{e}_j}(\psi)= e^{  i\tbe_{\vec{e}_i}}\psi_{\vec{e}_i}(t(\vec{e}_i))-e^{  i\tbe_{\vec{e}_j}}\psi_{\vec{e}_j}(t(\vec{e}_j)) 
\end{align}
we obtain
\begin{equation}
D(\mathcal{T}^B)=\bigcap_{\vec{e}_i,\vec{e}_j \in \mathcal{E}(\Lambda) \text{ with }i(\vec{e}_i)=i(\vec{e}_j)} \operatorname{ker}\left(l_{\vec{e}_i,\vec{e}_j}\right) \cap \bigcap_{\vec{e}_i,\vec{e}_j \in \mathcal{E}(\Lambda) \text{ with }t(\vec{e}_i)=t(\vec{e}_j)} \operatorname{ker}\left(k_{\vec{e}_i,\vec{e}_j}\right)
\end{equation}
which proves closedness of $\mathcal{T}^B$.
Green's identity follows directly from integration by parts on the level of edges.
The denseness of $\operatorname{ker}(\pi,\pi')$ is obvious since this space contains $\bigoplus_{\vec{e} \in \mathcal{E}(\Lambda)} C_c^{\infty}(\vec{e})$.
To show surjectivity, it suffices to consider a single edge. On those however, the property can be established by explicit constructions as in Lemma $2$ in \cite{P}. 
\end{proof}

Any boundary triple for $\mathcal{T}$ as in Def. \ref{triple} and any self-adjoint relation $A \subseteq \mathscr{H}'\bigoplus \mathscr{H}'$ gives rise \cite{Sch} to a self-adjoint restriction $\mathcal{T}_A$ of $\mathcal{T}$ with domain
\begin{equation}
D(\mathcal{T}_A)=\left\{\psi \in D(\mathcal{T}): (\pi(\psi), \pi'(\psi))\in A \right\}.
\end{equation}
The restriction of $\mathcal{T}^B$ satisfying Dirichlet type boundary conditions on every edge is obtained by selecting $A_1:=\left\{0\right\} \oplus l^2(\mathcal{V}(\Lambda))$ and coincides with $H^D$ \eqref{Dirichlet}. The operator $\Lambda^B$ \eqref{LambdaB} is recovered from $\mathcal{T}^B$ by picking the relation $A_2:=l^2(\mathcal{V}(\Lambda)) \oplus \left\{0\right\}.$

\begin{defi}\label{defgamma}
Given the boundary triple for $\mathcal{T}^B$ as above, the gamma-field 
$\gamma: \rho(H^D) \rightarrow \mathcal{L}(l^2(\mathcal{V}(\Lambda)),L^2(\mathcal{E}\left(\Lambda\right)))$
is given by $\gamma(\lambda):= \left(\pi \vert_{\operatorname{ker}(\mathcal{T}^B - \lambda )} \right)^{-1}$ 
and the Weyl function $M(\cdot, \Phi): \rho(H^D) \rightarrow \mathcal{L}(l^2(\mathcal{V}(\Lambda)))$ is defined as $ M(\lambda, \Phi):=\pi' \gamma(\lambda).$
\end{defi}
A computation shows that those maps are well-defined.
The resolvents of $H^D=\mathcal{T}_{A_1}^B$ and $\Lambda^B=\mathcal{T}_{A_2}^B$ are then related by Krein's resolvent formula \cite[Theorem $14.18$]{Sch}.
\begin{theo}
\label{KRF}
Let $(l^2(\mathcal{V}(\Lambda)),\pi,\pi')$ be the boundary triple for $\mathcal{T}^B$ and $\gamma,M$ as above, then for $\lambda \in \rho(H^D)\cap \rho(\Lambda^B)$ there is also a bounded inverse of $M(\lambda, \Phi)$ and
\begin{equation}
\label{Kreinresolvform}
(\Lambda^B-\lambda)^{-1}-(H^D-\lambda)^{-1} = -\gamma(\lambda)M(\lambda, \Phi)^{-1} \gamma(\overline{\lambda})^*.
\end{equation}
In particular, $\sigma(\Lambda^B)\backslash \sigma(H^D) = \{\lambda \in \mathbb{R} \cap \rho(H^D):\ 0 \in \sigma(M(\lambda, \Phi))\}$. For $\lambda \in \rho(H^D)$  we have $\gamma(\lambda)\operatorname{ker}(M(\lambda, \Phi)) = \operatorname{ker}(\Lambda^B-\lambda)$. This implies that both null-spaces are of equal dimension.
\end{theo}

\begin{lemm}\label{explicitKLambda}
For the operator $\mathcal{T}^B$, the gamma-field $\gamma$ and Weyl function $M$ can be explicitly written in terms of the solutions $s_{\lambda},c_{\lambda}$ \eqref{classicalsol} on an arbitrary edge $\vec{e} \in \mathcal{E}(\Lambda)$ for $\lambda \in \rho(H^D)$ and  $z \in l^2(\mathcal{V}(\Lambda))$ by
\begin{equation}
\label{gfield}
(\gamma(\lambda)z)_{\vec{e}}(x) =
 \frac{\left(s_{\lambda}(1)c_{\lambda, \vec{e}}(x)-s_{\lambda, \vec{e}}(x)c_{\lambda}(1)\right)z(i(\vec{e}))
+e^{-i\tbe_{\vec{e}}}s_{\lambda, \vec{e}}(x) z(t(\vec{e}))}{s_{\lambda}(1)}
\end{equation}
and 
\begin{equation}\label{Mlambdarep}
M(\lambda, \Phi) = \frac{K_{\Lambda}(\Phi)-\Delta(\lambda)}{s_{\lambda}(1)} 
\end{equation}
where 
\begin{equation}
\label{discreter}
(K_\Lambda (\Phi) z)(v):=\frac{1}{3} \left(\sum_{\vec{e}:\ i(\vec{e})=v } e^{-i  \tbe_{\vec{e}}} z(t(\vec{e})) +  \sum_{\vec{e}:\ t(\vec{e})=v } e^{i \tbe_{\vec{e}}} z(i(\vec{e})) \right)
\end{equation} defines an operator in $\mathcal{L}(l^2(\mathcal{V}(\Lambda)))$ with $\left\lVert K_{\Lambda}(\Phi) \right\rVert \le 1.$
\end{lemm}
\begin{proof}
For $\lambda \in \rho(H^D)$ and $z \in l^2(\mathcal{V}(\Lambda))$ we define for $\vec{e} \in \mathcal{E}(\Lambda)$ arbitrary
\begin{equation}
\psi_{\vec{e}}:=(\gamma(\lambda)z)_{\vec{e}} = ((\pi \vert_{\operatorname{ker}(\mathcal{T}^B-\lambda)})^{-1}z)_{\vec{e}}
\end{equation}
with $\psi:=(\psi_{\vec{e}}).$
In particular, $\psi_{\vec{e}}$ is the solution to $-\psi''_{\vec{e}} + V_{\vec{e}}\psi_{\vec{e}} = \lambda \psi_{\vec{e}}$ with the following boundary condition: $\psi_{\vec{e}}(i(\vec{e}))=z(i(\vec{e}))$ and $\psi_{\vec{e}}(t(\vec{e}))=e^{-i\tbe_{\vec{e}} }z(t(\vec{e}))$. The representation \eqref{gfield} is then an immediate consequence of \eqref{representsol}. 

The expression for the Weyl function on the other hand, follows from the Dirichlet-to-Neumann map \eqref{DtN}.

\begin{align}
\label{Weylrep}
(M(\lambda, \Phi)z)(v) &= (\pi' \gamma(\lambda) z)(v)  \nonumber \\
&= \frac{1}{3} \left(\sum_{\vec{e}:\ i(\vec{e})=v } \psi_{\vec{e}}'(v) - \sum_{\vec{e}:\  t(\vec{e})=v} e^{i \tbe_{\vec{e}}} \psi_{\vec{e}}'(v) \right) \nonumber \\
&=  \frac{(K_\Lambda (\Phi) z)(v)}{s_{\lambda}(1)}- \left( \underbrace{\frac{c_{\lambda}(1)}{s_{\lambda}(1)} \delta_{v \in i(\mathcal{V}(\Lambda))}z(v) + \frac{s_{\lambda}'(1)}{s_{\lambda}(1)}\delta_{v \in t(\mathcal{V}(\Lambda))}z(v)}_{= \frac{s_{\lambda}'(1)}{s_{\lambda}(1)} z(v)}\right) \nonumber \\
&= \frac{\left(K_\Lambda (\Phi) z-s_{\lambda}'(1)z \right)(v) }{s_{\lambda}(1)},
\end{align}
here we used \eqref{cossinrel}. 
The formula (\ref{Mlambdarep}) then follows from (\ref{Weylrep}) and (\ref{defDelta}).
Since $i(\Lambda)\cap t(\Lambda)=\emptyset$, we have $\|K_{\Lambda}(\Phi)\|\leq 1$.
\end{proof}


Since all vertices are integer translates of either of the two vertices $r_0,r_1 \in W_{\Lambda}$ by basis vectors $\vec{b}_1, \vec{b}_2$, we conclude that $l^2(\mathcal{V}(\Lambda)) \simeq l^2(\mathbb{Z}^2;\mathbb{C}^2)$. 
Our next Lemma shows $K_{\Lambda}(\Phi)$ and $Q_{\Lambda}(\Phi)$ are unitary equivalent under this identification.

\begin{lemm}\label{QLKL}
$K_{\Lambda}(\Phi)$ is unitary equivalent to operator $Q_{\Lambda}(\Phi)$.
\end{lemm}
\begin{proof}
By (\ref{deftildebeta}), (\ref{discreter}),
\begin{equation*}
\begin{cases}
(K_{\Lambda}(\Phi)z)(\gamma_1, \gamma_2, r_0)= \frac{1}{3} \left(z(\gamma_1, \gamma_2, r_1) +z(\gamma_1-1, \gamma_2, r_1) + e^{-i \Phi \gamma_1}z(\gamma_1, \gamma_2-1, r_1)\right),\\
(K_{\Lambda}(\Phi)z)(\gamma_1, \gamma_2, r_1)= \frac{1}{3} \left(z(\gamma_1, \gamma_2, r_0)+z(\gamma_1+1, \gamma_2, r_0) +e^{i\Phi \gamma_1}z(\gamma_1, \gamma_2+1, r_0)\right).
\end{cases}
\end{equation*}
In order to transform $K_{\Lambda}$ to $Q_{\Lambda}$ we use the unitary identification $
W:l^2(\mathcal{V}(\Lambda)) \rightarrow l^2(\mathbb{Z}^2, \mathbb{C}^2)$ 
\begin{equation}
\label{id1}
(Wz)_{\gamma_1, \gamma_2}:= \left( \begin{matrix} z(\gamma_1, \gamma_2, r_0) \ ,z(\gamma_1, \gamma_2, r_1 )\end{matrix} \right)^T.
\end{equation}

This way, $Q_{\Lambda}(\Phi) = WK_{\Lambda}(\Phi) W^*.$
\end{proof}

\begin{rem}
In terms of $a \in l^2(\mathbb{Z}^2,\mathbb{C}^2)$ defined as
\begin{align}
\label{a}
a_{(0,0)}&:=\frac{1}{3} \left(\begin{matrix} 0 && 1 \\ 1 && 0 \end{matrix}\right), \ 
a_{(0,1)}:=\frac{1}{3} \left(\begin{matrix} 0 && 1 \\ 0 && 0 \end{matrix}\right) \nonumber\\
a_{(1,0)}&:=\frac{1}{3} \left(\begin{matrix} 0 && 1 \\ 0 && 0 \end{matrix}\right), \
a_{(0,-1)}:=\frac{1}{3} \left(\begin{matrix} 0 && 0 \\ 1 && 0 \end{matrix}\right) \nonumber\\
a_{(-1,0)}&:=\frac{1}{3} \left(\begin{matrix} 0 && 0 \\ 1 && 0 \end{matrix}\right)
 ,\text{ and } a_{\gamma}:=0 \text{ for other }\gamma \in \mathbb{Z}^2, 
\end{align}
we can express \eqref{Q-op} in the compact form
\begin{equation}
\label{qlambda}
Q_{\Lambda}(\Phi)=\sum_{\gamma \in \mathbb{Z}^2; \lvert \gamma \rvert \le 1}a_{\gamma}(\tau_0)^{\gamma_1} (\tau_1)^{\gamma_2}.
\end{equation} 
This operator has already been studied, in different contexts, for rational flux quanta in \cite{KL}, \cite{HKL}, and \cite{AEG}.
\end{rem}

Finally, we point out that Lemma \ref{HBQLam} follows from a combination of Theorem \ref{KRF}, Lemma \ref{explicitKLambda} and Lemma \ref{QLKL}. $\hfill{} \Box$


\section{Spectral analysis}\label{speanasec}
This section is devoted to complete spectral analysis of $H^B$.

In view of Lemmas \ref{HBQLam} and \ref{Qsym}, an important technical fact is:
\begin{lemm}
\label{operatornorm}
The operator norm of $Q_{\Lambda}(\Phi)$ for non-trivial flux quanta $\Phi \notin 2 \pi \mathbb{Z}$ is strictly less than $1$.
\end{lemm}

Indeed, then, away from the Dirichlet spectrum $\sigma(H^D)$, which are located on the edges of the Hill bands (\ref{DirichletHilledge}), we have the following characterization of $\sigma(H^B)$. Let $B_n$ and $\Delta$ be defined as in Section \ref{Hillsec}. 

\begin{lemm}\label{simpleproperties}
For the magnetic Schr\"odinger operator $H^B$, the following properties hold.
\begin{enumerate}
    \item The level of the Dirac points $\Delta\vert_{\operatorname{int}(B_n)}^{-1}(0)$ always belongs to the spectrum of $H^B$, i.e. $0 \in \Delta\vert_{\operatorname{int}(B_n)}(\sigma(H^B))$.
    \item $\lambda \in \Delta\vert_{\operatorname{int}(B_n)}(\sigma(H^B))$ iff $-\lambda \in \Delta\vert_{\operatorname{int}(B_n)}(\sigma(H^B)).$ Consequently, the property $\Delta'(0)\neq 0$ implies that locally with respect to the Dirac points, the spectrum of $H^B$ is symmetric.
    \item $H^B$ has no point spectrum away from $\sigma(H^D)$.
    \item For non-trivial flux $\Phi\notin 2\pi \Z$, $H^B$ has purely continuous spectrum bounded away from $\sigma(H^D)$.
\end{enumerate}
\end{lemm}

Combining Lemma \ref{simpleproperties} with Lemma \ref{Hausdorff}, we get 
\begin{lemm}\label{HausdorffofH}
For generic $\Phi$, $\dim_H(\sigma^{\Phi})\leq \frac{1}{2}$.
\end{lemm}
\subsubsection*{Proof of Lemma \ref{HausdorffofH}}
Lemma \ref{Hausdorff} and (2) of Lemma \ref{simpleproperties} together imply that
for generic $\Phi$,
\begin{align*}
\dim_H\left(\Delta\vert_{\operatorname{int}(B_n)}^{-1}(\sigma(Q_{\Lambda}(\Phi ))\right)\leq \frac{1}{2}.
\end{align*}
Hence since 
\begin{align*}
\sigma^{\Phi}=\sigma(H^D)\bigcup \left(\cup_{n\in \N}\ \Delta\vert_{\operatorname{int}(B_n)}^{-1}(\sigma(Q_{\Lambda}(\Phi ))\right),
\end{align*} 
we have
\begin{align*}
\dim_H(\sigma^{\Phi})\leq \sup\left\lbrace\dim_H(\sigma(H^D)),\ \sup_{n\in \N}\ \dim_H\left(\Delta\vert_{\operatorname{int}(B_n)}^{-1}(\sigma(Q_{\Lambda}(\Phi ))\right) \right\rbrace\leq \frac{1}{2}.
\end{align*}
\qed

\subsubsection*{Proof of Lemma \ref{simpleproperties}}
(1), (2) follow from a quick combination of Lemma \ref{HBQLam} and Lemma \ref{Qsym}, and (3) follows from Part (2) of Lemma \ref{spectrumQlambda}. (4) is a corollary of Lemma \ref{operatornorm} and (3). $\hfill{} \Box$

\subsubsection*{Alternate proof of Lemma \ref{operatornorm}}
Without loss of generality $\Phi\in (0, 2\pi)$.
By \eqref{QLambdaHalpha}, it suffices to show $\|H_{\Phi, \theta}\|<c_{\Phi}<6$ for some constant $c_{\Phi}$ independent of $\theta\in \T_1$.
Let us take $\varphi\in \ell^2(\Z)$ with $\|\varphi\|_{\ell^2(\Z)}=1$. 
Consider 
\begin{align*}
(H_{\Phi, \theta}\varphi)_n
=&\underbrace{c\left(\theta+n\frac{\Phi}{2\pi}\right)\varphi_{n+1}}_{=:(h_1\varphi)_n}+\underbrace{\overline{c\left(\theta+(n-1)\frac{\Phi}{2\pi}\right)}\varphi_{n-1}}_{=:(h_2\varphi)_n}+\underbrace{v\left(\theta+n\frac{\Phi}{2\pi}\right)\varphi_n}_{=:(h_3\varphi)_n},
\end{align*}
in which $h_1, h_2, h_3\in \mathcal{L}(\ell^2(\Z))$.
Hence 
\begin{align*}
\|H_{\Phi, \theta}\varphi\|_{\ell^2(\Z)}^2
\leq &3\left(\|h_1\varphi\|_{\ell^2(\Z)}^2+\|h_2\varphi\|_{\ell^2(\Z)}^2+\|h_3\varphi\|_{\ell^2(\Z)}^2\right)\\
\leq &3\ \sup_{n\in \Z}\left(\left|c\left(\theta+(n-1)\frac{\Phi}{2\pi}\right)\right|^2+\left|\overline{c\left(\theta+n\frac{\Phi}{2\pi}\right)}\right|^2+\left|v\left(\theta+n\frac{\Phi}{2\pi}\right)\right|^2\right)\\
\leq & 12\sup_{\theta\in \T_1}\left(\sin^2 \pi\left(\theta-\frac{\Phi}{2\pi}\right)+\sin^2 (\pi \theta)+\cos^2 (2\pi \theta)\right)\\
=:&c_{\Phi}^2<36.
\end{align*}
\qed

In order to investigate further the Dirichlet spectrum and spectral decomposition of the continuous spectrum into absolutely and singular continuous parts, we start with constructing magnetic translations.

\subsection{Magnetic translations}

\

In general, $\Lambda^B$ does not commute with lattice translations $T^{st}_{\gamma}$. 
Yet, there is a set of modified translations that do still commute with $\Lambda^B,$ although they in general no longer commute with each other.
We define those magnetic translations $T^{B}_{\gamma} : L^2(\mathcal{E}\left(\Lambda\right)) \rightarrow L^2(\mathcal{E}\left(\Lambda\right))$ as unitary operators given by
\begin{equation}
\label{MagTra}
(T^{B}_{\gamma} \psi)_{\vec{e}} := u^B_{\gamma}(\vec{e}) (T^{\text{st}}_{\gamma}\psi)_{\vec{e}}
\end{equation}
for any $\psi:=(\psi_{\vec{e}})_{\vec{e} \in \mathcal{E}(\Lambda)} \in L^2(\mathcal{E}\left(\Lambda\right))$ and $\gamma \in \mathbb{Z}^2.$
The lattice translation $T^{st}_{\gamma}$ is defined by $(T^{st}_{\gamma}\psi)_{\vec{e}}(x)=\psi_{\vec{e}-\gamma_1\vec{b}_1-\gamma_2\vec{b}_2}(x-\gamma_1 \vec{b}_1-\gamma_2\vec{b}_2)$ as before.
The function $u^B_{\gamma}$ is constant on each copy of the fundamental domain, and defined as follows
\begin{align}\label{defuBgamma}
u^B_{\gamma}(\tilde{\gamma}_1, \tilde{\gamma}_2, \vec{[e]})= e^{i\Phi \gamma_1 \tilde{\gamma}_2},\ \ \ \text{for } \vec{[e]}=\vec{f}, \vec{g}\text{ or }\vec{h}, \ \ \gamma=(\gamma_1, \gamma_2)\in \Z^2.
\end{align}

For $\vec{e}=\tilde{\gamma}_1\vec{b}_1+\tilde{\gamma}_2\vec{b}_2+\vec{[e]}$ we have
\begin{align}\label{TmuTgamma}
(T^B_{\mu}T^B_{\gamma}\psi)_{\vec{e}}=&u_{\mu}(\vec{e})u_{\gamma}(\vec{e}-\mu_1 \vec{b}_1-\mu_2 \vec{b}_2) \psi_{\vec{e}-\mu_1 \vec{b}_1-\mu_2 \vec{b}_2-\gamma_1 \vec{b}_1-\gamma_2 \vec{b}_2} \nonumber\\
=&e^{i\Phi\mu_1 \tilde{\gamma}_2+i\Phi\gamma_1(\tilde{\gamma}_2-\mu_2)}\psi_{\vec{e}-\mu_1 \vec{b}_1-\mu_2 \vec{b}_2-\gamma_1 \vec{b}_1-\gamma_2 \vec{b}_2}, \nonumber \\
(T^B_{\gamma} T^B_{\mu}\psi)_{\vec{e}}=&u_{\gamma}(\vec{e})u_{\mu}(\vec{e}-\gamma_1 \vec{b}_1-\gamma_2 \vec{b}_2) \psi_{\vec{e}-\mu_1 \vec{b}_1-\mu_2 \vec{b}_2-\gamma_1 \vec{b}_1-\gamma_2 \vec{b}_2} \nonumber\\
=&e^{i\Phi \gamma_1 \tilde{\gamma}_2+i\Phi \mu_1(\tilde{\gamma}_2-\gamma_2)}\psi_{\vec{e}-\mu_1 \vec{b}_1-\mu_2 \vec{b}_2-\gamma_1 \vec{b}_1-\gamma_2 \vec{b}_2}.
\end{align}
Hence 
\begin{align}\label{commrel}
T^B_{\mu}T^B_{\gamma}=T^B_{\gamma}T^B_{\mu}e^{i\Phi \omega(\mu, \gamma)},
\end{align}
where $\omega(\mu,\gamma):=\mu_1\gamma_2-\mu_2\gamma_1$ is the standard symplectic form on $\mathbb{R}^2.$
It also follows from (\ref{TmuTgamma}) that $(T^B_{\gamma})^{-1}=e^{-i\Phi \gamma_1 \gamma_2} T^B_{-\gamma}$.
One can also check that 
\begin{align}\label{TB*}
(T^B_{\gamma})^*=e^{-i\Phi \gamma_1\gamma_2} T^B_{-\gamma}=(T^B_{\gamma})^{-1},
\end{align}
hence $T^B_{\gamma}$ is unitary.

By the definition (\ref{MagTra}), (\ref{defuBgamma}), it is clear that for any $\psi\in L^2(\mathcal{E}(\Lambda))$, 
\begin{align}\label{TBcommutesHe}
\frac{d}{dt} T^B_{\gamma}\psi=T^B_{\gamma}\frac{d}{dt}\psi \text{ and } VT^B_{\gamma}\psi=T^B_{\gamma}V\psi.
\end{align}

In order to make sure $D(\Lambda^BT^B_{\gamma})=D(T^B_{\gamma}\Lambda^B)$, 
it suffices to check $T^B_{\gamma}(D(\Lambda^B))=D(\Lambda^B)$, which translates into
\begin{align}\label{TBcommutesLambdaBtranslates}
\left\lbrace
\begin{matrix}
u_{\gamma}(\vec{e}_1)=u_{\gamma}(\vec{e}_2)\ \text{ whenever } i(\vec{e}_1)=i(\vec{e}_2)\\
\\
\frac{e^{i\tbe_{\vec{e}_2}}u_{\gamma}(\vec{e}_2)}{e^{i\tbe_{\vec{e}_1}}u_{\gamma}(\vec{e}_1)}=\frac{e^{i\tbe_{\vec{e}_2-\gamma_1\vec{b}_1-\gamma_2 \vec{b}_2}}}{e^{i\tbe_{\vec{e}_1-\gamma_1 \vec{b}_1-\gamma_2 \vec{b}_2}}}\ \text{ whenever } t(\vec{e}_1)=t(\vec{e}_2).
\end{matrix}
\right.
\end{align}
This, by (\ref{deftildebeta}) is in turn equivalent to the following: for any $\tilde{\gamma}_1, \tilde{\gamma}_2\in \Z$:
\begin{align*}
\left\lbrace
\begin{matrix} 
u_{\gamma}(\tilde{\gamma}_1, \tilde{\gamma}_2, \vec{f})=u_{\gamma}(\tilde{\gamma}_1, \tilde{\gamma}_2, \vec{g})
=u_{\gamma}(\tilde{\gamma}_1, \tilde{\gamma}_2, \vec{h})\\
\\
u_{\gamma}(\tilde{\gamma}_1, \tilde{\gamma}_2, \vec{f})=u_{\gamma}(\tilde{\gamma}_1+1, \tilde{\gamma}_2, \vec{g})=e^{-i\Phi \gamma_1}u_{\gamma}(\tilde{\gamma}_1, \tilde{\gamma}_2+1, \vec{h})
\end{matrix}
\right.
\end{align*}
The definition of $u^B_{\gamma}$ (\ref{defuBgamma}) clearly satisfies this requirement. 

Therefore, although magnetic translations do not necessarily commute with one another \eqref{commrel}, they commute with $\Lambda^B$
\begin{equation}\label{TBcommutesLambda}
T_{\gamma}^B \Lambda^B = \Lambda^B T_{\gamma}^B.
\end{equation}

One can check that $T^B_{\gamma}(D(\mathcal{T}^B))=D(\mathcal{T}^B)$ is also equivalent to \eqref{TBcommutesLambdaBtranslates}, hence 
by \eqref{TBcommutesHe},
\begin{align}\label{TBcommutescalTB}
T_{\gamma}^B \mathcal{T}^B=\mathcal{T}^B T_{\gamma}^B.
\end{align}
Note that we also have 
\begin{align}\label{TBcommutesHD}
T_{\gamma}^B H^D=H^D T_{\gamma}^B,
\end{align}
which is due to (\ref{TBcommutesHe}) and $T_{\gamma}^B (\mathcal{H}_0^1(\vec{e}))=\mathcal{H}_0^1(\vec{e}-\gamma_1\vec{b}_1-\gamma_2\vec{b}_2)$ for any $\vec{e}\in \mathcal{E}(\Lambda)$.

We will now study the structure of eigenfunctions and the nature of the continuous spectrum of $H^B$.

\subsection{Dirichlet spectrum}\label{Dirichsec}

\

In this subsection, we will study the energies belonging to the Dirichlet spectrum $\sigma(H^D)$. 
Lemma \ref{Dirichlet contribution} below shows that $\sigma(H^D)$ is contained in the point spectrum of $H^B$, hence the {\it only} point spectrum of $H^B$, due to Part (3) of Lemma \ref{simpleproperties}.

Consider a compactly supported simply closed loop, which is a path with vertices of degree $2$ enclosing $q$ hexagons, see e.g. Fig. \ref{Fig:loop}. 
Then this loop passes (proceeding in positive direction from the center of an edge $\vec{e}_1$ such that the first vertex we reach is $t(\vec{e}_1)$) $n:=2+4q$ edges $\vec{e}_1,..., \vec{e}_{n}$ in $\mathcal{E}(\Lambda).$
For a solution vanishing outside this loop,
the boundary conditions imposed by \eqref{DomLambdaB} on the derivatives can be represented in a matrix equation
\begin{equation}
T_{\Phi}(n)\psi^\prime(n)=0,
\end{equation}
where
\begin{equation}
\label{Matrix}
T_{\Phi}(n):= \left(\begin{matrix} 
e^{i\tbe_{\vec{e}_1}} & e^{i\tbe_{\vec{e}_2}} & 0 & 0 & 0 &\cdots & 0 \\
0 & 1 & 1 &0 & 0 & \cdots& 0 \\
0 &0 &e^{i\tbe_{\vec{e}_3}} & e^{i \tbe_{\vec{e}_4}} & 0 &\cdots & 0 \\
\vdots & \vdots & 0 & \ddots & \ddots  & 0 & 0  \\ 
0 & 0 & 0 &0 & 0  & e^{i \tbe_{\vec{e}_{n-1}}} & e^{i \tbe_{\vec{e}_n}} \\ 
1 & 0 & 0 &0 & 0  & 0 & 1 \\ 
\end{matrix}\right)\ \mathrm{and}\
\psi^\prime(n):=
\left(\begin{matrix}
\psi^\prime_{\vec{e}_1}(1)\\
\psi^\prime_{\vec{e}_2}(1)\\
\psi^\prime_{\vec{e}_3}(1)\\
\vdots\\
\psi^\prime_{\vec{e}_{n-1}}(1)\\
\psi^\prime_{\vec{e}_n}(1)
\end{matrix}
\right).
\end{equation}

\begin{rem}
\label{rankofmat}
We observe that $T_{\Phi}(n)$ can be row-reduced to an upper triangular matrix with diagonal 
\begin{align*}
  &(e^{i\tbe_{\vec{e}_1}}, 1, e^{i\tbe_{\vec{e}_3}}, 1, e^{i\tbe_{\vec{e}_5}},...,1, e^{i\tbe_{\vec{e}_{n-1}}}, 1-e^{i\sum_{j=1}^n (-1)^j \tbe_{\vec{e}_j}})\\
=&(e^{i\tbe_{\vec{e}_1}}, 1, e^{i\tbe_{\vec{e}_3}}, 1, e^{i\tbe_{\vec{e}_5}},...,1, e^{i\tbe_{\vec{e}_{n-1}}}, 1-e^{iq\Phi}),
\end{align*} 
where $q$ is the number of enclosed hexagons.
Hence $\operatorname{rank}(T_{\Phi}(n))=n$ iff $ q \Phi \notin 2\pi \mathbb{Z}$ and $\operatorname{rank}(T_{\Phi}(n))=n-1$ otherwise.
\end{rem}

\begin{lemm}
\label{Dirichlet contribution}
The Dirichlet eigenvalues $\lambda \in \sigma(H^D)$ are contained in the point spectrum of $H^B$. 
\end{lemm}
\begin{proof}
For $\Phi \in 2\pi \mathbb{Z}$ the statement is known \cite[Theroem $3.6$]{KP}, thus we focus on $\Phi \notin 2\pi \mathbb{Z}.$ By unitary equivalence, it suffices to construct an eigenfunction to $\Lambda^B.$
We will construct an eigenfunction on two adjacent hexagons $\Gamma$ as in Fig. \ref{Fig:doublehex}. Thus, $q=2$, the total number of edges is $m=11,$ of which $n=10$ are on the outer loop.
Let us denote the slicing edge by $\vec{e}$ and the edges on the outer loop by $\vec{e}_1, \vec{e}_2,..., \vec{e}_{10}$ (see Fig. \ref{Fig:doublehex}).
Recall that $s_{\lambda, \vec{e}}$ is the Dirichlet eigenfunction on $\vec{e}$.

By Remark \ref{rankofmat}, for $2 \Phi \in 2\pi \mathbb{Z}$, operator $T_{\Phi}(10)$ has a non-trivial nullspace. We could take 
\begin{equation}
a=(a_j) \in \operatorname{ker}\left(T_{\Phi}(10)\right)\backslash\{0\},
\end{equation}
and an eigenfunction $\psi$ on $\Gamma$ such that $\psi_{\vec{e}}=0$ and $\psi_{\vec{e}_j}=a_j s_{\lambda, \vec{e}_j}$.

If $2\Phi \notin 2\pi \mathbb{Z}$, we take a vector $y \in \mathbb{C}^{10}$ such that $y_2=-1$, $y_7=-e^{i\tbe_{\vec{e}}}$ and $y_j=0$ otherwise.
Since in this case $T_{\Phi}(10)$ is invertible, there exists a unique solution $a=(a_j)$ to the following equation:
\begin{equation}\label{10y}
T_{\Phi}(10) a = y.
\end{equation}
Let us take $\psi$ on $\Gamma$ such that $\psi_{\vec{e}_j}=a_j s_{\lambda, \vec{e}_j}$ and $\psi_{\vec{e}}=s_{\lambda, \vec{e}}$, then one can easily check $\psi$ is indeed an eigenfunction on $\Gamma$.
\end{proof}

\begin{center}
\begin{figure}
\label{doublehexstate}
\centerline{\includegraphics[height=7cm]{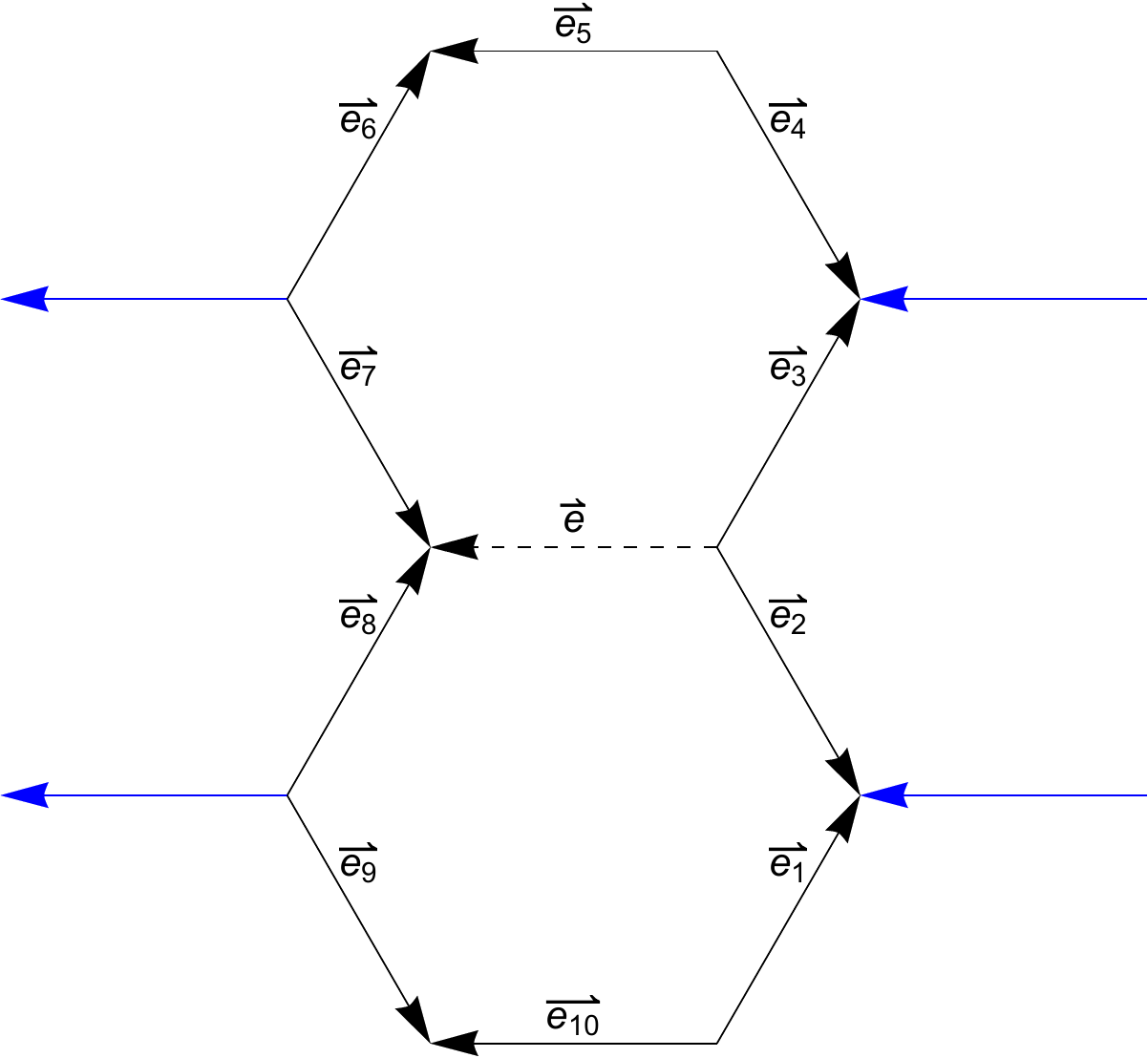}} 
\caption{Black arrows describe the double hexagon with slicing edge $\vec{e}$ indicated by the dashed arrow. \label{Fig:doublehex}}
\end{figure}
\end{center}

As a corollary of Lemma \ref{operatornorm} and Lemma \ref{Dirichlet contribution}, we have the following:
\begin{corr}
The spectrum of $H^B$ must always have open gaps for $\Phi \notin 2 \pi \mathbb{Z}$ at the edges of the Hill bands. 
\end{corr}
\begin{rem}
If the magnetic flux is trivial, i.e. $\Phi \in 2 \pi \mathbb{Z}$, then there do not have to be gaps. In particular, for zero potential in the non-magnetic case discussed in Theorem \ref{zero magnetic potential} all gaps of the absolutely continuous spectrum are closed and $\sigma_{\text{ac}}(H^B)=[0,\infty)$.
\end{rem}

The next lemma concerns the general feature of eigenspace of $H^B$.
Before proceeding, let us introduce the degree of a vertex in order to distinguish different types of eigenfunctions.
\begin{defi}
An eigenfunction is said to have a vertex of degree $d$ if there is a vertex with exactly $d$ adjacent edges on which the eigenfunction does not vanish. 
\end{defi}

\begin{lemm}\label{eigenspace}
For the point spectrum of $H^B$ it follows that
\begin{enumerate}
    \item Every eigenspace of $H^B$ is infinitely degenerated.
    \item Eigenfunctions of $H^B$ vanish at every vertex and are thus eigenfunctions of $H^D$ as well.
    \item Eigenfunctions of $H^B$ with compact support cannot have vertices of degree $1$. In particular, they must contain loops and the boundary edges of their support form loops as well. 
\end{enumerate}
\end{lemm}
\begin{proof}
(1). This follows immediately using magnetic translations (\ref{MagTra}). 
Assume there was a finite-dimensional eigenspace of $H^B.$  
Because magnetic translations commute with $H^B$, they leave the eigenspaces of $H^B$ invariant.
Magnetic translations are unitary, thus there is for any magnetic translation a normalized eigenfunction $\psi$ with eigenvalue $\lambda$ on the unit circle in $\mathbb{C}$. For $\psi$, there is a sufficiently large ball $B(0,R)$ such that 
\begin{equation}
\label{eq1}
\left\lVert \psi \right\rVert_{L^2(\mathcal{E}(\Lambda) \cap B(0,R))}>1-\varepsilon.
\end{equation}
Upon $n-$fold application of the magnetic translation, the point $0$ gets translated to some point $x_n$ whereas the eigenfunction $\psi$ acquires only a complex phase $\lambda^n$. Thus, \eqref{eq1} still holds and we must also have that 
\begin{equation}
\label{eq2}
\left\lVert \psi \right\rVert_{L^2(\mathcal{E}(\Lambda) \cap B(x_n,R))}>1-\varepsilon.
\end{equation}
Yet, there exists $n$ such that $B(0,R)\cap B(x_n,R)=\emptyset.$ Therefore, \eqref{eq1} and \eqref{eq2} cannot hold at the same time for arbitrarily large $n$. This contradicts the existence of an eigenfunction to magnetic translations and thus the existence of a finite-dimensional eigenspace.
 \bigbreak

(2). If there is an eigenfunction to $H^B$ with eigenvalue $\lambda$ that does not vanish at a vertex, by (modified) Peierls' substitution \eqref{modifiedP}, there is one to $\Lambda^B,$ denoted as $\varphi,$  as well. We may expand the function in local coordinates on every edge $\vec{e} \in \mathcal{E}(\Lambda)$ as $\varphi_{\vec{e}}=a_{\vec{e}} c_{\lambda, \vec{e}} + b_{\vec{e}} s_{\lambda, \vec{e}}$ according to \eqref{classicalsol}. Recall also that the Dirichlet eigenfunction $s_{\lambda}$ is either even or odd. Thus, using \eqref{cossinrel} we conclude that $\left\lvert c_{\lambda}(0) \right\rvert= \left\lvert c_{\lambda}(1) \right\rvert$ and thus $\varphi$ cannot be compactly supported. In particular, $\varphi$ has the same absolute value at any vertex by boundary conditions \eqref{DomLambdaB}. Due to 
\begin{equation}
\sum_{\vec{e} \in \mathcal{E}(\Lambda)} \left\lvert \varphi_{\vec{e}}(i(\vec{e})) \right\rvert^2 \le \left\lVert \varphi \right\rVert^2_{\mathcal{H}^2}< \infty
\end{equation}
$\varphi$ has to vanish at every vertex. Thus $\varphi$ is also an eigenfunction to $H^{D}.$

(3) clearly follows from (2) and \eqref{DomLambdaB}.
\end{proof}

\subsubsection{Dirichlet spectrum for rational flux quanta}

\

In this section, the flux quanta are assumed to be reduced fractions $\frac{\Phi}{2\pi} =\frac{p}{q}.$

If magnetic fields are absent, the point spectrum is spanned by hexagonal simply closed loop states, i.e. states supported on a single hexagon \cite{KP}. 
We will see in the following that similar statements remain true in the case of rational flux quanta and derive such a basis as well.
The natural extension of loop states supported on a single hexagon, in the case of magnetic fields, are simply closed loops enclosing an area $\frac{q\Phi}{B_0}$ rather than just $\frac{\Phi}{B_0}$, see Fig. \ref{Fig:loop}.
\begin{center}
\begin{figure}
\centerline{\includegraphics[height=6cm]{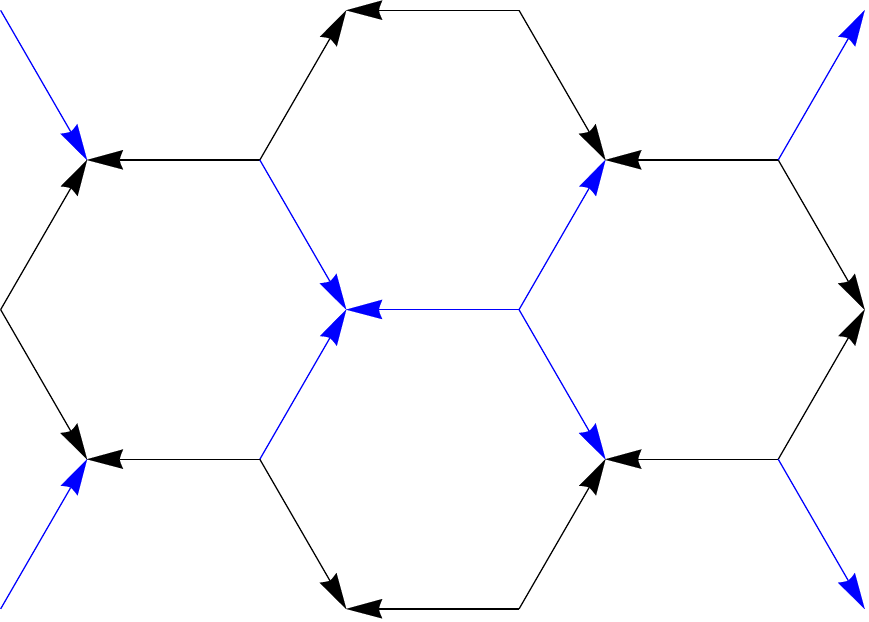}} 
\caption{Simply closed loop state supported on black arrows encloses area $\frac{4\Phi}{B_0}$\label{Fig:loop}.}
\end{figure}
\end{center}
\begin{lemm}
\label{Dirichletrat}
Any simply closed loop enclosing an area of $\frac{q \Phi}{B_0}$ has a unique (up to normalization) eigenfunction of $H^B$ supported on it. 
\end{lemm}
\begin{proof}
The existence of eigenfunctions on simply closed loops enclosing this flux follows directly from the non-trivial kernel of \eqref{Matrix}, see Remark \ref{rankofmat}. Due to $\operatorname{dim}(\operatorname{ker}(T_{\Phi}))=1$, such eigenfunctions are also unique (up to normalization). 
\end{proof}

\begin{lemm}
\label{compsuppstat}
The nullspaces $\operatorname{ker}(H^B-\lambda)$ where $\lambda \in \sigma(H^D)$ are generated by compactly supported eigenfunctions.
\end{lemm}
\begin{proof}
Unitary equivalence allows us to work with $\Lambda^B$ rather than $H^B$. 
Without loss of generality, we assume that the Dirichlet eigenfunction to $\lambda$ is even. 
Due to Lemma \ref{eigenspace}, eigenfunctions of $\Lambda^B$ to Dirichlet eigenvalues vanish at every vertex. Thus, on every edge $\vec{e}\in \mathcal{E}(V)$, they are of the form $\varphi_{\vec{e}}=a_{\vec{e}}s_{\lambda, \vec{e}}$ for some $a_{\vec{e}}$.

Let $\varphi$ be such a function. We define the sequence $(u(v))_{v \in \mathcal{V}(\Lambda)}$ as follows
\begin{align}
\left\lbrace
\begin{matrix}
u(\gamma_1, \gamma_2, r_0):= \varphi_{\gamma_1, \gamma_2, \vec{g}}'(\gamma_1, \gamma_2, r_0) \nonumber \\
u(\gamma_1, \gamma_2, r_1):= \varphi_{\gamma_1, \gamma_2, \vec{f}}'(\gamma_1, \gamma_2, r_1).
\end{matrix}
\right.
\end{align}
Observe that the sequence $(u(v))$ determines the eigenfunction on every edge. Indeed, $a_{\gamma_1, \gamma_2, \vec{g}}=u(\gamma_1, \gamma_2, r_0)$ and $a_{\gamma_1, \gamma_2, \vec{f}}=u(\gamma_1, \gamma_2, r_1)$, since $s^\prime_{\lambda}(1)=s^\prime_{\lambda}(0)$.
At the same time, $a_{\gamma_1, \gamma_2, \vec{h}}$ can be determined in two different ways, one for each endpoint, from the boundary condition \eqref{DomLambdaB}. 
Let us now introduce an operator $A \in \mathcal{L}( l^2(\mathcal{V}(\Lambda)))$ that has precisely the sequences $(u(v))$ with matching boundary conditions for $a_{\gamma_1, \gamma_2, \vec{h}}$ in its kernel.
Then,
\begin{align}
(Au)(\gamma_1, \gamma_2, r_0)
&:= u(\gamma_1, \gamma_2, r_0)+u(\gamma_1, \gamma_2, r_1) \nonumber\\
&\ \ \ \ -e^{2\pi i \frac{p\gamma_1}{q}} \Big(u(\gamma_1+1, \gamma_2-1, r_0)+u(\gamma_1, \gamma_2-1, r_1) \Big) \text{ and } \nonumber\\
(Au)(\gamma_1, \gamma_2, r_1)&:=0.
\end{align}
The operator $A$ is then a $\mathbb{Z}^2$-periodic finite-order difference operator.
Any eigenfunction $\varphi$ satisfying $(\Lambda^B - \lambda)\varphi=0$ leads by standard arguments to a square-summable sequence $(u(v))$ as defined above in the nullspace of $A$. Conversely, any such element in the nullspace of $A$ uniquely defines an eigenfunction $\varphi=a_{\vec{e}} s_{\lambda, \vec{e}}$ to $\Lambda^B.$
Theorem $8$ in \cite{K2} implies then that the nullspace of $A$ is generated by sequences in $c_{00}(\mathcal{V}(\Lambda))$. It suffices now to observe that those compactly supported sequences also give rise to compactly supported eigenfunctions to conclude the claim.
\end{proof}

\begin{lemm}
Let $\Phi \notin 2 \pi \mathbb{Z}.$ The eigenspaces are spanned by the set of double hexagonal states Fig. \ref{Fig:doublehex}.
\end{lemm}
\begin{proof}
By Lemma \ref{eigenspace}, all eigenfunctions vanish at every vertex. Compactly supported eigenfunctions are dense in the eigenspace by the previous Lemma \ref{compsuppstat}. Thus, it suffices, as in the non-magnetic \cite{KP} case, to show that any compactly supported eigenfunction is a linear combination of double hexagonal states. Let $\varphi$ be a compactly supported eigenfunction of $\Lambda^B$ to some Dirichlet eigenvalue $\lambda$. Consider an edge $\vec{d} \in \mathcal{E}(\Lambda)$ on the boundary loop of the support of $\varphi$. It exists due to (3) of Lemma \ref{eigenspace}. The boundary loop, which cannot  be just a loop around a single hexagon, as this one does not support such eigenfunctions, necessarily encloses a double hexagon $\Gamma$, as in Fig. \ref{Fig:doublehex}, which contains the chosen edge $\vec{d}.$ Then, there is by the proof of Lemma \ref{Dirichlet contribution} a state $\psi$ on $\Gamma$ so that the wavefunction $\psi_{\vec{d}}$ on $\vec{d}$ coincides with $\varphi_{\vec{d}}.$ Subtracting $\psi$ from $\varphi$ leaves us with an eigenfunction to $\Lambda^B$ that encloses at least one single hexagon less than $\psi.$ Thus, iterating this procedure shows that compactly supported eigenfunctions are spanned by double hexagonal states which implies the claim.
\end{proof}

\subsubsection{Dirichlet spectrum for irrational flux quanta}\


After proving Theorem \ref{T4} for rational flux quanta, we now prove the analogous result for irrational magnetic fluxes.
We start by introducing the following definition.
\begin{defi}
The Hilbert space $l^2(\mathcal{E}(\Lambda))$ is defined as 
\begin{equation}
l^2(\mathcal{E}(\Lambda)):=\left\{ z: \mathcal{E}(\Lambda) \rightarrow \mathbb{C},\  \left\lVert z \right\rVert_{l^2(\mathcal{E}(\Lambda))}^2:=\sum_{\vec{e} \in \mathcal{E}(\Lambda)} \left \lvert z(\vec{e}) \right \rvert^2 < \infty \right\}.
\end{equation}
\end{defi}
\begin{theo}
\label{irratdirichlet}
The double hexagonal states generate the eigenspaces of Dirichlet spectrum of $H^B$ for irrational flux quanta.
\end{theo}
We will give a proof of this theorem after a couple of auxiliary observations. For this entire discussion to follow we consider a fixed $\lambda \in \sigma(H^D).$
\begin{defi}
We denote the closed $L^2(\mathcal{E}(\Lambda))$ subspace generated by linear combinations of all double hexagonal states on the entire graph $\Lambda$ by $DH_{\mathcal{E}(\Lambda)}(\Phi).$
\end{defi}
There is a countable orthonormal system of states $V(\Phi)\subset DH_{\mathcal{E}(\Lambda)}(\Phi)$ such that  
\begin{equation}
\overline{\operatorname{span}(V(\Phi))}= DH_{\mathcal{E}(\Lambda)}(\Phi). 
\end{equation}
We may label elements of $V(\Phi)$ by $\varphi_{\gamma}(\Phi)$ with $\gamma \in \mathbb{Z}^2$.
Without loss of generality, $\varphi_{\gamma}(\Phi)$ can be chosen to depend analytically on $\Phi\in (0,1)$.
Every element $\varphi_{\gamma}(\Phi) \in V(\Phi)$ is due to Lemma \ref{eigenspace} of the form
\begin{equation}
\label{GramSchmidtstates}
\varphi_{\gamma}(\Phi)=\sum_{\vec{e} \in \mathcal{E}(\Lambda)} \varphi_{\gamma,\vec{e}}(\Phi) s_{\lambda,\vec{e}}
\end{equation}
because it is an element of $\operatorname{ker}(H^B-\lambda).$

Now assume that the statement of Theorem \ref{irratdirichlet} does not hold, this is equivalent to saying that  $Z(\Phi):= \operatorname{ker}(H^B-\lambda) \cap DH_{\mathcal{E}(\Lambda)}(\Phi)^{\perp}$ is not the zero space, i.e. there are eigenfunctions not spanned by double hexagonal states. Our goal is to characterize $Z(\Phi)$ as the nullspace of a suitable operator we define next. 
\begin{defi}
Let $A(\Phi) \in \mathcal{L}(l^2(\mathcal{E}(\Lambda)))$ on $u\in l^2(\mathcal{E}(\Lambda))$ for $\gamma \in \mathbb{Z}^2$ be given by 
\begin{equation}
\begin{split}
\label{defA}
(A(\Phi)u)(\gamma,\vec{f})&:=u(\gamma,\vec{f})+u(\gamma,\vec{g})+u(\gamma,\vec{h}) \\
(A(\Phi)u)(\gamma,\vec{g})&:=u(\gamma_1, \gamma_2-1 ,\vec{f})+u(\gamma_1+1, \gamma_2-1,\vec{g})+e^{-i \Phi \gamma_1 } u(\gamma_1, \gamma_2,\vec{h}) \\
(A(\Phi)u)(\gamma,\vec{h})&:=\left\langle u, (\varphi_{\gamma,\vec{e}}(\Phi)) \right\rangle_{l^2(\mathcal E(\Lambda))}.
\end{split}
\end{equation}
\end{defi}
\begin{rem}
The first two lines of this definition resemble the boundary conditions for the derivatives at outgoing/incoming vertices \eqref{magopd} and with the third line we monitor the orthogonality of $\sum_{\vec{e} \in \mathcal{E}(\Lambda)} u_{\vec{e}} s_{\lambda,\vec{e}}$ to $DH_{\mathcal{E}(\Lambda)}(\Phi).$
\end{rem}
In particular, there is an isometric isomorphism $\eta \in \mathcal{L}( \operatorname{ker}(A(\Phi)), Z(\Phi))$ with
\begin{equation}
\label{iso}
\eta(u):=\sum_{\vec{e} \in \mathcal{E}(\Lambda)} \frac{u_{\vec{e}}}{\left\lVert s_{\lambda,\vec{e}}\right\rVert_{L^2(\vec{e})}} s_{\lambda,\vec{e}}.
\end{equation}
We observe that by Lemma \eqref{compsuppstat} and the isomorphism \eqref{iso} the operator $A(\Phi)$ is injective for $\frac{\Phi}{2\pi} \in \mathbb{Q} \cap (0,1)$. 
To prove Theorem \ref{irratdirichlet} we only need the following Lemma:
\begin{lemm}
\label{sur}
The operator $A(\Phi)$ is surjective for $\frac{\Phi}{2\pi} \in  (0,1)$. In particular, for any $(a(\vec{e}))\in l^2(\mathcal{E}(\Lambda))$, there exists $(u(\vec{e}))\in l^2(\mathcal{E}(\Lambda))$ such that $A(\Phi)u=a$ and 
\begin{align}\label{surnorm}
\|u\|_{l^2(\mathcal{E}(\Lambda))}\leq \frac{C}{|1-e^{-i\Phi}|} \|a\|_{l^2(\mathcal{E}(\Lambda))}
\end{align}
holds for a universal constant $C$.
\end{lemm}
Combining Lemma \ref{sur} with the already established injectivity result, we have $A(\Phi)$ is continuously invertible for $\frac{\Phi}{2\pi}\in \mathbb{Q}\cap (0,1)$ with the following control of its norm
\begin{align}\label{inversenorm}
\|A(\Phi)^{-1}\|\leq \frac{C}{|1-e^{-i\Phi}|}.
\end{align}

Now let us give the proof of Theorem \ref{irratdirichlet}, assuming the result of Lemmas \ref{sur}.
\subsection*{Proof of Theorem \ref{irratdirichlet}}
Since $\|A(\Phi)\|$ is uniformly bounded by a constant and $\Phi \mapsto \langle x,A(\Phi)y \rangle$ is analytic for $x,y \in c_{00}(\mathcal{E}(\Lambda))$,
$A(\Phi)$ is an analytic operator in $\Phi$. 
Thus for any $\frac{\tilde{\Phi}}{2\pi} \in (0,1)$, there exists $\epsilon_1(\tilde{\Phi})$ and $C(\tilde{\Phi})$ such that 
\begin{align}\label{Lip}
\|A(\Phi)-A(\tilde{\Phi})\|\leq C(\tilde{\Phi})|\Phi-\tilde{\Phi}|,\ \ \ \ \ \mathrm{for}\ |\Phi-\tilde{\Phi}|<\epsilon_1(\tilde{\Phi}).
\end{align}
Also by (\ref{inversenorm}), for any irrational $\frac{\tilde{\Phi}}{2\pi}\in (0,1)$ and rational $\frac{\Phi}{2\pi}$ with $|\Phi-\tilde{\Phi}|<\epsilon_2(\tilde{\Phi})$, we have
\begin{align}
\|A(\Phi)^{-1}\|\leq \frac{2C}{|1-e^{-i\tilde{\Phi}}|}.
\end{align}
Hence, taking $\frac{\Phi}{2\pi}\in \mathbb{Q}\cap (0,1)$ with $|\Phi-\tilde{\Phi}|<\min(\epsilon_1(\tilde{\Phi}), \epsilon_2(\tilde{\Phi}), \frac{|1-e^{-i\tilde{\Phi}}|}{2C(\tilde{\Phi})C})$, we would get
\begin{align*}
\|A(\Phi)^{-1}(A(\tilde{\Phi})-A(\Phi))\|<1.
\end{align*}
This implies that
\begin{align*}
A(\tilde{\Phi})=A(\Phi)\left(\mathrm{Id}+A(\Phi)^{-1}(A(\tilde{\Phi})-A(\Phi))\right)
\end{align*}
is invertible.
Thus, we conclude that also for irrational fluxes $\operatorname{ker}(A(\Phi))=\left\{0\right\}$ and by $\eqref{iso}$ therefore $Z(\Phi)=\left\{0\right\}$ which shows the claim. $\hfill{} \Box$


\subsection*{Proof of Lemma \ref{sur}}

We prove this Lemma by showing that there is a sufficiently sparse set of elements in $l^2(\mathcal{E}(\Lambda))$ that gets mapped under $A(\Phi)$ on the standard basis of $l^2(\mathcal{E}(\Lambda)).$

Let $\alpha_{\vec{e},(\gamma,\vec{h})}:= \varphi_{\gamma,\vec{e}} \left\lVert s_{\lambda,e} \right\rVert^2_{L^2(\vec{e})}.$ Since functions $\varphi_{\gamma}$ satisfy the continuity conditions \eqref{magopd} and form an $L^2$ orthonormal system, we obtain the standard basis vectors $\delta_{\bullet,(\gamma,\vec{h})} \in l^2(\mathcal{E}(\Lambda))$ under $A(\Phi)$
\begin{equation}
\begin{split}
\label{basis3}
(A(\Phi)\alpha_{\bullet,(\gamma,\vec{h})})(\gamma',\vec{f})&:=0, \\
(A(\Phi)\alpha_{\bullet,(\gamma,\vec{h})})(\gamma',\vec{g})&:=0, \text{ and }\\
(A(\Phi)\alpha_{\bullet,(\gamma,\vec{h})})(\gamma',\vec{h})&:=\delta_{\gamma,\gamma'}.
\end{split}
\end{equation}
\begin{figure}
\includegraphics[width=0.4\textwidth]{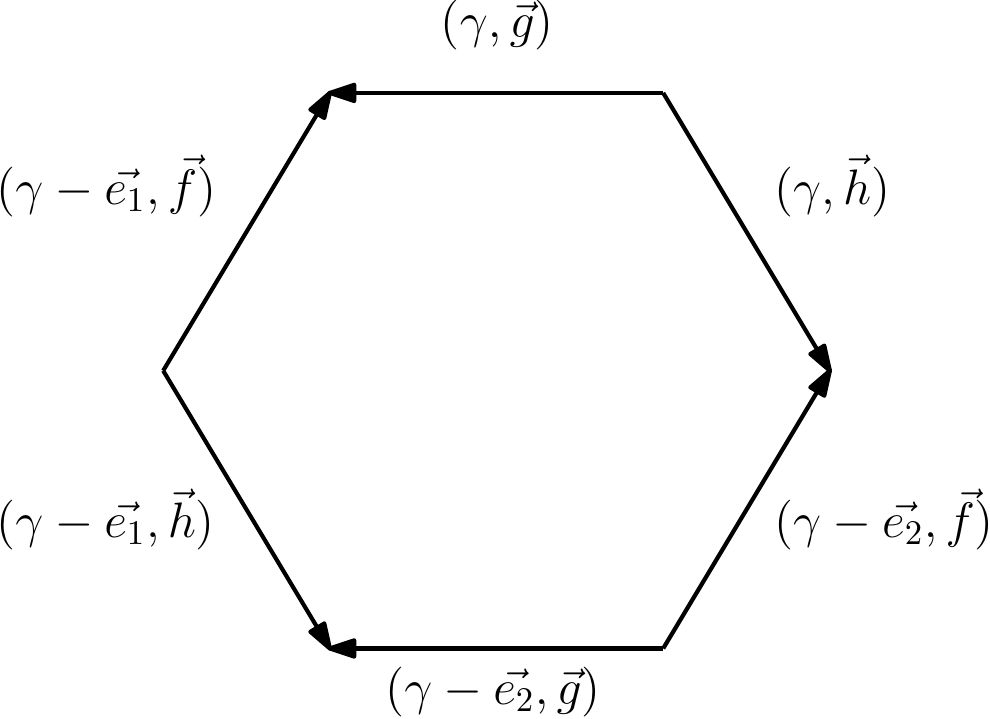}
\caption{Labelling of hexagon $\Gamma_{\gamma}$. \label{Comb}}
\end{figure}
To obtain also the remaining basis vectors, let us define $L^2$ functions $\tilde{\psi}_{(\gamma,\vec{f})}$ and  $\tilde{\psi}_{(\gamma,\vec{g})}$ supported on a single hexagon $\Gamma_{\gamma}$ as shown in Figure \ref{Comb}. The indices of $\tilde{\psi}_{(\gamma,[\vec{e}])}$ are chosen to indicate the standard basis vectors $\delta_{\bullet,(\gamma,[\vec{e}])} \in l^2(\mathcal E(\Lambda))$ in the range of $A(\Phi)$ that we will construct from those functions. To define $\tilde{\psi}_{(\gamma,\vec{f})}$ and $\tilde{\psi}_{(\gamma,\vec{g})}$, we introduce coefficients $\zeta_{\bullet,{(\gamma,\vec{f})}}$ and $\zeta_{\bullet,{(\gamma,\vec{g})}}$ such that $\tilde{\psi}_{(\gamma,\vec{f})}:=\sum_{\vec{e} \in \mathcal{E}(\Gamma_{\gamma})} \zeta_{\vec{e},(\gamma,\vec{f})} s_{\lambda,\vec{e}}$ and $\tilde{\psi}_{(\gamma,\vec{g})}:=\sum_{\vec{e} \in \mathcal{E}(\Gamma_{\gamma})} \zeta_{\vec{e},(\gamma,\vec{g})} s_{\lambda,\vec{e}},$ respectively.

We do this in such a way that all continuity conditions for $\tilde{\psi}_{(\gamma,\vec{f})}$ at the vertices of $\Gamma_{\gamma}$ are satisfied up to a single one at the (initial) vertex $i((\gamma,\vec{g}))=i((\gamma,\vec{h}))$. We define for fixed $\vec{e}=(\gamma,\vec{f})$
\begin{equation}
\begin{split}
\label{coeff}
&\zeta_{(\gamma,\vec{h}),\vec{e}}:=\frac{1}{1-e^{-i\Phi}}, \
 \zeta_{(\gamma-e_2,\vec{f}),\vec{e}}:=\frac{-e^{-i\Phi \gamma_1}}{1-e^{-i\Phi}}, \ 
\zeta_{(\gamma-e_2,\vec{g}),\vec{e}}:=\frac{e^{-i\Phi \gamma_1}}{1-e^{-i\Phi}}, 
\\ 
&\zeta_{(\gamma-e_1,\vec{h}),\vec{e}}:=\frac{-e^{-i\Phi}}{1-e^{-i\Phi}}, \ 
\zeta_{(\gamma-e_1,\vec{f}),\vec{e}}:=\frac{e^{-i\Phi }}{1-e^{-i\Phi}},\ 
\zeta_{(\gamma,\vec{g}),\vec{e}}:=\frac{-e^{-i\Phi}}{1-e^{-i\Phi}}
\end{split}
\end{equation}
and all other $\zeta_{\bullet,\vec{e}}$ are taken to be zero.
Since for $\tilde{\psi}_{(\gamma,\vec{f})}$ all but one continuity conditions are satisfied, we obtain for the first two components of \eqref{defA}
\begin{align}
(A(\Phi)\zeta_{\bullet,(\gamma,\vec{f})})(\gamma',\vec{f}):=\delta_{\gamma,\gamma'} \text{ and } (A(\Phi)\zeta_{\bullet,(\gamma,\vec{f})})(\gamma',\vec{g}):=0.
\end{align}
To ensure that we also get constant zero in the third component of \eqref{defA}, we project onto the orthogonal complement of the double hexagonal states $\psi_{(\gamma,\vec{f})}:=\widetilde{\psi}_{(\gamma,\vec{f})}-P_{DH_{\mathcal{E}(\Lambda)}(\Phi)}\widetilde{\psi}_{(\gamma,\vec{f})}$ where $P_{DH_{\mathcal{E}(\Lambda)}(\Phi)}$ is the orthogonal projection onto $DH_{\mathcal{E}(\Lambda)}(\Phi).$ Let now $\alpha_{\vec{e},(\gamma,\vec{f})}$ be such that
\begin{equation}
\psi_{(\gamma,\vec{f})}=\sum_{\vec{e} \in \mathcal{E}(\Gamma_{\gamma})} \alpha_{\vec{e},(\gamma,\vec{f})} s_{\lambda,\vec{e}},
\end{equation} 
then it follows that
\begin{align}
\label{basis1}
(A(\Phi)\alpha_{\bullet,(\gamma,\vec{f})})(\gamma',\vec{f})&:=\delta_{\gamma,\gamma'}, \nonumber\\
(A(\Phi)\alpha_{\bullet,(\gamma,\vec{f})})(\gamma',\vec{g})&:=0, \text{ and } \nonumber\\
(A(\Phi)\alpha_{\bullet,(\gamma,\vec{f})})(\gamma',\vec{h})&:=0.
\end{align}

Similarly, we choose coefficients $\zeta_{\bullet,\vec{e}}$ with $\vec{e}=(\gamma,\vec{g})$, such that the boundary conditions are satisfied up to the one at the (terminal) vertex $t((\gamma,\vec{h}))=t((\gamma-e_2,\vec{f}))$
\begin{align}
\label{coeff2}
& \zeta_{(\gamma-e_2,\vec{f}),\vec{e}}:=\frac{1}{1-e^{-i\Phi}}, \ 
\zeta_{(\gamma-e_2,\vec{g}),\vec{e}}:=\frac{-1}{1-e^{-i\Phi}}, 
\ \zeta_{(\gamma-e_1,\vec{h}),\vec{e}}:=\frac{e^{i\Phi(\gamma_1-1)}}{1-e^{-i\Phi}} \nonumber\\
& \zeta_{(\gamma-e_1,\vec{f}),\vec{e}}:=\frac{-e^{i\Phi(\gamma_1-1)}}{1-e^{-i\Phi}},\
 \zeta_{(\gamma,\vec{g}),\vec{e}}:=\frac{e^{i\Phi(\gamma_1-1)}}{1-e^{-i\Phi}}, \  
 \zeta_{(\gamma,\vec{h}),\vec{e}}:=\frac{-e^{i\Phi(\gamma_1-1)}}{1-e^{-i\Phi}}
\end{align}
and all other coefficients $\zeta_{\bullet,\vec{e}}$ equal to zero. Thus, we get for the first two components of \eqref{defA}
\begin{align}
(A(\Phi)\zeta_{\bullet,(\gamma,\vec{g})})(\gamma',\vec{f})= 0 \text{ and }
(A(\Phi)\zeta_{\bullet,(\gamma,\vec{g})})(\gamma',\vec{g})=\delta_{\gamma,\gamma'}.
\end{align}
To ensure that we also get constant zero in the third component of \eqref{defA}, we project again on the orthogonal complement of the double hexagonal states $\psi_{(\gamma,\vec{g})}:=\widetilde{\psi}_{(\gamma,\vec{g})}-P_{DH_{\mathcal{E}(\Lambda)}(\Phi)}\widetilde{\psi}_{(\gamma,\vec{g})}$. Let now $\psi_{(\gamma,\vec{g})}=\sum_{\vec{e} \in \mathcal{E}(\Gamma_{\gamma})} \alpha_{\vec{e},(\gamma,\vec{g})} s_{\lambda,\vec{e}}$, then 
\begin{equation}
\begin{split}
\label{basis2}
(A(\Phi)\alpha_{\bullet,(\gamma,\vec{g})})(\gamma',\vec{f})&:=0 \\
(A(\Phi)\alpha_{\bullet,(\gamma,\vec{g})})(\gamma',\vec{g})&:=\delta_{\gamma,\gamma'}, \text{ and}  \\
(A(\Phi)\alpha_{\bullet,(\gamma,\vec{g})})(\gamma',\vec{h})&:=0.
\end{split}
\end{equation}

Hence, we obtained in \eqref{basis3}, \eqref{basis1}, and \eqref{basis2} sequences 
\begin{equation}
\label{eq:Riesz basis}
\left\{\alpha_{\bullet,(\gamma,\vec{f})}, \alpha_{\bullet,(\gamma,\vec{g})}, \text{ and } \alpha_{\bullet,(\gamma,\vec{h})} ; \gamma \in \mathbb{Z}^2   \right\} 
\end{equation} 
in $l^2(\mathcal{E}(\Lambda))$ that get mapped under $A(\Phi)$ onto the standard unit basis of $l^2(\mathcal{E}(\Lambda)).$

To conclude surjectivity of $A(\Phi)$ from this, it suffices to show that for all $(a({\vec{e}})) \in l^2(\mathcal{E}(\Lambda))$ we can bound $u(\vec{e}):=\sum_{\vec{d}  \in \mathcal{E}(\Lambda)}  a({\vec{d}}) \ \alpha_{\vec{e},\vec{d}}$ as follows
\begin{equation}
\label{crestimate}
\left\lVert u \right\rVert^2_{l^2(\mathcal{E}(\Lambda))}  \le \frac{C^2}{\left\lvert 1-e^{-i\Phi} \right\rvert^2} \sum_{\vec{e} \in \mathcal{E}(\Lambda)} \left\lvert a({\vec{e}}) \right\rvert^2.
\end{equation}
We then define 
\begin{equation}
\begin{split}
\sigma_{\vec{e}}=\sum_{\vec{d} \in \mathcal E(\Lambda); [\vec{d}] \neq \vec{h}} a({\vec{d}}) \ \alpha_{\vec{e},\vec{d}} \text{  and  } \nu_{\vec{e}}= \sum_{\vec{d} \in \mathcal E(\Lambda); [\vec{d}] \neq \vec{h}} a({\vec{d}}) \ \zeta_{\vec{e},\vec{d}}. 
\end{split}
\end{equation}

Since $\psi_{(\gamma,\vec{f})},\psi_{(\gamma,\vec{g})}\in DH_{\mathcal{E}(\Lambda)}(\Phi)^{\perp}$ and $(\varphi_{\gamma})$ forms an orthonormal system in $DH_{\mathcal{E}(\Lambda)}(\Phi)$, to prove \eqref{crestimate} it suffices to show 
\begin{equation}
\label{crestimate2}
\left\lVert  \sigma \right\rVert^2_{l^2(\mathcal{E}(\Lambda))}  \le \frac{C^2}{\left\lvert 1-e^{-i\Phi} \right\rvert^2}  \sum_{\vec{e} \in \mathcal{E}(\Lambda);[\vec{e}] \neq \vec{h} } \left\lvert a({\vec{e}}) \right\rvert^2.
\end{equation}
Due to $\left\lVert  \sigma \right\rVert_{l^2(\mathcal{E}(\Lambda))} \le \left\lVert  \nu \right\rVert_{l^2(\mathcal{E}(\Lambda))}+ \left\lVert \sigma-\nu \right\rVert_{l^2(\mathcal{E}(\Lambda))}$ we may establish estimate \eqref{crestimate} for each term on the right-hand side of the triangle inequality, individually.

For two edges $ \vec{d}, \vec{e} \in \mathcal{E}(\Lambda)$ we define a function $M(\vec{d},\vec{e}):=1$ if there are $\gamma,\gamma' \in \mathbb Z^2$ and two hexagons $\Gamma_{\gamma}, \Gamma_{\gamma'}$ satisfying $\Gamma_{\gamma} \cap \Gamma_{\gamma'}\neq \emptyset$ such that $\vec{d} \in \Gamma_{\gamma}$ and $\vec{e} \in\Gamma_{\gamma'},$ and $M(\vec{d},\vec{e}):=0$ otherwise. Choosing $\tau_1$ such that $\sum_{\vec{d} \in \mathcal E(\Lambda); [\vec{d}] \neq \vec{h}} M(\vec{d},\vec{e})\le \tau_1$ for any $\vec{e} \in \mathcal E(\Lambda),$ then
\begin{equation}
\begin{split}
\left\lVert  \nu \right\rVert^2_{l^2(\mathcal{E}(\Lambda))}
&\le \sum_{\vec{d},\vec{e} \in \mathcal{E}(\Lambda); [\vec{d}],[\vec{e}] \neq \vec{h}}   \left\lvert a({\vec{d}}) \right\rvert \left\lvert a({\vec{e}}) \right\rvert \underbrace{\left\lVert \zeta_{\bullet,\vec{d}} \right\rVert_{l^2(\mathcal{E}(\Lambda))} \left\lVert \zeta_{\bullet,\vec{e}} \right\rVert_{l^2(\mathcal{E}(\Lambda))} }_{\le \frac{7}{\left\lvert 1-e^{-i\Phi} \right\rvert^2}} M(\vec{d},\vec{e}) \\
&\le \frac{7 \tau_1 }{\left\lvert 1-e^{-i\Phi} \right\rvert^2} \sum_{\vec{e} \in \mathcal{E}(\Lambda);[\vec{e}] \neq \vec{h}} \left\lvert a({\vec{e}}) \right\rvert^2.
\end{split}
\end{equation}
For the second term $\left\lVert \sigma-\nu \right\rVert_{l^2(\mathcal{E}(\Lambda))}$, we use that functions $\tilde{\psi}_{(\gamma,[e])}$ with $[e] \neq [h]$ are supported on hexagons $\Gamma$ and can therefore only overlap with finitely many linearly independent double hexagonal states. Thus, we define a function $N$ with $N(\vec{d}, \vec{e}):=1$ if $\vec{d},\vec{e}$ belong to two hexagons $\Gamma_{\gamma}, \Gamma_{\gamma'}$ for which there are two double hexagons $\Gamma_{1},\Gamma_{2}$ with the property that all intersections $\Gamma_{\gamma} \cap \Gamma_{1}$, $\Gamma_{1} \cap \Gamma_{2}$, $\Gamma_{2}\cap \Gamma_{\gamma'}$ are not empty. Otherwise, we set $N(\vec{d}, \vec{e}):=0.$ Choosing $\tau_2$ such that $\sum_{\vec{d} \in \mathcal E(\Lambda); [\vec{d}] \neq \vec{h}} N(\vec{d}, \vec{e})\le \tau_2$ for any $\vec{e} \in \mathcal E(\Lambda),$ then
\begin{equation}
\begin{split}
\left\lVert \sigma-\nu \right\rVert_{l^2(\mathcal{E}(\Lambda))}^2 
&=\sum_{\vec{d},\vec{e} \in \mathcal{E}(\Lambda);[\vec{d}],[\vec{e}]  \neq \vec{h}} \frac{N( \vec{d},\vec{e}) a({\vec{d}})\  \overline{a({\vec{e}})}}{\left\lVert s_{\lambda} \right\rVert_{L^2((0,1))}^2} \left\langle P_{DH_{\mathcal{E}(\Lambda)}(\Phi)}\tilde{\psi}_{\vec{d}}, P_{DH_{\mathcal{E}(\Lambda)}(\Phi)}\tilde{\psi}_{\vec{e}} \right\rangle_{L^2(\mathcal{E}(\Lambda))}  \\
&\le \sum_{\vec{d},\vec{e} \in \mathcal{E}(\Lambda);[\vec{d}],[\vec{e}]  \neq \vec{h}} \left\vert a({\vec{d} }) \right\vert \left\lvert a({\vec{e}}) \right\rvert  \left\lVert \zeta_{\bullet,\vec{d}} \right\rVert_{l^2(\mathcal{E}(\Lambda))} \left\lVert \zeta_{\bullet,\vec{e}} \right\rVert_{l^2(\mathcal{E}(\Lambda))} N( \vec{d},\vec{e})    \\
& \le \frac{7\tau_2}{\left\lvert 1-e^{-i\Phi} \right\rvert^2} \sum_{\vec{e} \in \mathcal{E}(\Lambda);[\vec{e}] \neq \vec{h}} \left\lvert a({\vec{e}}) \right\rvert^2. 
\end{split}
\end{equation}
$\hfill{} \Box$

\subsection{Absolutely continuous spectrum for rational flux quanta}\label{acsec}
\begin{lemm}
\label{rationalcase}
For $\frac{\Phi}{2\pi}=\frac{p}{q} \in \mathbb{Q},$ the spectrum of $H^B$ away from the Dirichlet spectrum is absolutely continuous and has possibly touching, but non-overlapping band structure. An interval $I \subset [-1,1]$ is a band of $Q_{\Lambda}(\Phi)$ if and only if its pre-image under $\Delta$, on each fixed band of the Hill operator, is a band of $H^B.$ 
\end{lemm}
\begin{proof}

Note that $\sigma_p^{\Phi}\backslash \sigma(H^D)=\emptyset$ is due to (3) of Lemma \ref{simpleproperties}.
The absence of singular continuous spectra is clear because $H^B$ is invariant under magnetic translations \eqref{MagTra}, some of which commute with one another in the case of rational flux quanta. In particular, $T_{0,q}^B$ and $T_{q,0}^B$ always generate a group of commutative magnetic translations due to \eqref{commrel}. Thus, standard arguments from Floquet-Bloch theory show that $H^B$ has no singular continuous spectrum \cite{GeNi}.

Now we will show the non-overlapping band structure. We recall that $\Lambda^B$, $\mathcal{T}^B$, $H^D$ all commute with magnetic translations $T^B_{\mu}$ \eqref{TBcommutesLambda} \eqref{TBcommutescalTB} \eqref{TBcommutesHD} where we assume that $\mu \in q \mathbb{Z}$. Consequently, $T_{\mu}^B$ leaves all eigenspaces of those operators invariant.

For $\lambda \in \rho(H^D)$ we define discrete magnetic translations $\tau_{\mu}^B: l^2(\mathcal{V})\rightarrow l^2(\mathcal{V})$,
\begin{equation}
\label{dismatra}
\tau_{\mu}^B:= \pi(\lambda)T_{\mu}^B \gamma(\lambda),
\end{equation}
with $\gamma$ as in Def. \ref{defgamma}.
Recall now that $\gamma(\lambda)=(\pi|_{\mathrm{ker}(\mathcal{T}^B-\lambda)})^{-1}$ maps given vertex values onto a function in $\operatorname{ker}(\mathcal{T}^B - \lambda)$ with $\mathcal{T}^B$ as in Lemma \ref{TB-Lemm}. The operator $T_{\mu}^B$ translates this function and multiplies it by a function taking values on the unit circle of $\mathbb{C}$. Hereupon, $\pi(\lambda)$ recovers the translated and phase shifted vertex values. Thus, the definition of those discrete translation operators is in fact independent of $\lambda \in \rho(H^D)$ and they form a family of unitary operators with $\left(\tau_{\mu}^B\right)^*=\tau_{-\mu}^B$, due to (\ref{TB*}). See also \cite[Proof of Lemma $4.9$]{BZ}.

Since $\Lambda^B$, $H^D$ both commute with $T_{\mu}^B$, Krein's formula \eqref{Kreinresolvform} implies that for $\lambda\in \rho(H^D)\cap \rho(\Lambda^B)$,
\begin{equation}
T^B_{\mu} \left(\gamma(\lambda)M(\lambda, \Phi)^{-1} \gamma(\overline{\lambda})^*\right) = \left(\gamma(\lambda)M(\lambda, \Phi)^{-1} \gamma(\overline{\lambda})^*\right)T^B_{\mu}.
\end{equation}
Multiplying with $\pi(\lambda)$ and $\pi(\overline{\lambda})^*$ from both sides respectively, it follows that
\begin{align}
\label{commutativityM}
\tau_{\mu}^B M(\lambda, \Phi)^{-1}
&=M(\lambda, \Phi)^{-1}\left( \gamma(\overline{\lambda})^* T^B_{\mu} \pi(\overline{\lambda})^* \right) \nonumber\\
&=M(\lambda, \Phi)^{-1}\left( \pi(\overline{\lambda}) T^B_{-\mu}\gamma(\overline{\lambda}) \right)^{*} \nonumber\\
&=M(\lambda, \Phi)^{-1} \tau^B_{\mu},
\end{align}
thus $M(\lambda, \Phi)^{-1}$ commutes with discrete magnetic translations.

To see that $\Delta$ restricts on every Hill band $B_n$ (because $\Delta\vert_{B_n}$ is one-to-one) to an isomorphism from each band of $Q_{\Lambda}(\Phi)$ to a unique band of $\Lambda^B,$ it suffices to note that the preceding calculation shows that Krein's formula \eqref{Kreinresolvform} holds true for the Floquet-Bloch transformed operators. Let $U^{\operatorname{cont}}$ be the Gelfand transform generated by translations $T_{\mu}^B$ and $U^{\operatorname{discrete}}$ the Gelfand transform generated by discrete translations $\tau_{\mu}^B,$ i.e. for $k$ in the Brioullin zone
\begin{align}
(U^{\operatorname{cont}}f)(k,x)&:=\sum_{\mu \in q\mathbb{Z}^2} \left(T_{\mu}^Bf\right)(x)e^{i\langle k,\mu \rangle} \text{ and } \nonumber\\
(U^{\operatorname{discrete}}\xi)(k,\delta)&:=\sum_{\mu \in q\mathbb{Z}^2} \left(\tau_{\mu}^B\xi\right)(\delta)e^{i\langle k,\mu \rangle}.
\end{align} 
Using \eqref{dismatra} and $\gamma(\lambda)\pi(\lambda)=\operatorname{id}_{\mathcal{V}(\Lambda)}$ we obtain from 
\begin{equation}
\tau_{\mu}^B \gamma(\overline{\lambda})^*=\left(\tau_{-\mu}^B\right)^*\gamma(\overline{\lambda})^*=\gamma(\overline{\lambda})^*T_{\mu}^B \pi(\overline{\lambda})^*\gamma(\overline{\lambda})^*=\gamma(\overline{\lambda})^*T_{\mu}^B,
\end{equation} 
where we used that $T_{\mu}^B$ preserves $\operatorname{ker}(\mathcal{T^B}-\overline{\lambda})$ and $\gamma(\overline{\lambda})^*\lvert_{\operatorname{ker}(\mathcal{T^B}-\overline{\lambda})^{\perp}}=0,$ on some fundamental domain $W_{\Lambda}^{\Phi}$
\begin{equation}
\begin{split}
-\left(U^{\operatorname{cont}}\gamma(\lambda)M(\lambda, \Phi)^{-1} \gamma(\overline{\lambda})^{*}f\right)(k,x)
&=-\gamma(\lambda)(k)\left(U^{\operatorname{discrete}}M(\lambda, \Phi)^{-1}\gamma(\overline{\lambda})^{*}f\right)(k,x) \\
&=-\gamma(\lambda)(k)M(\lambda, \Phi)(k)^{-1}\left(U^{\operatorname{discrete}}\gamma(\overline{\lambda})^{*}f\right)(k,x) \\
&=-\gamma(\lambda)(k)M(\lambda, \Phi)(k)^{-1}\gamma(\overline{\lambda})(k)^*(U^{\operatorname{cont}}f)(k,x).
\end{split}
\end{equation}
where
$\gamma(\lambda)(k) \in \mathcal{L}(l^2(W_{\Lambda}^{\Phi}),L^2(W_{\Lambda}^{\Phi}))$ and $M(\lambda, \Phi)(k) \in \mathcal{L}(l^2(W_{\Lambda}^{\Phi}))$ are the restrictions of $\gamma(\lambda)$ and $M(\lambda, \Phi)$ on $l^2(W_{\Lambda}^{\Phi})$ satisfying Floquet boundary conditions. 
Both $\Lambda^B$ and $H^D$ commute with translations $T_{\mu}^B$, and thus fiber upon conjugation by $U^{\operatorname{cont}},$ so that for $\lambda \in \rho(\Lambda^B(k)) \cap \rho(H^D)$ Krein's formula remains true for each $k$
\begin{equation}
\label{Kreinresolvform2}
(\Lambda^B(k)-\lambda)^{-1}-(H^D\vert_{L^2(W_{\Lambda}^{\Phi})}-\lambda)^{-1} = -\gamma(\lambda)(k)M(\lambda, \Phi)(k)^{-1} \gamma(\overline{\lambda})(k)^*.
\end{equation}
In particular, for $\lambda \in \rho(H^D)$
\begin{equation}
\gamma(\lambda)(k) \operatorname{ker}(M(\lambda, \Phi)(k))= \operatorname{ker}(\Lambda^B(k)-\lambda).
\end{equation}
Therefore bands of $Q_{\Lambda}(\Phi)$ are in one-to-one correspondence with bands of $\Lambda^B$, and thus also with bands of $H^B$.
That the bands of $Q_{\Lambda}(\Phi)$ do not overlap is shown in Section $6$ of \cite{HKL}. Thus, the unique correspondence among bands of $Q_{\Lambda}(\Phi)$ and $H^B$ shows that the non-overlapping of bands holds true for $H^B$ as well. 
\end{proof}
\begin{rem}
\label{closedgaps}
For $\frac{\Phi}{2\pi}= \frac{1}{2}$ the spectral bands of $Q_{\Lambda}(\Phi)$ are touching and given by \cite{HKL} 
\begin{equation}
\left[-\sqrt{\frac{2}{3}},-\sqrt{\frac{1}{3}}\right], \ \left[-\sqrt{\frac{1}{3}},0\right],\ \left[0,\sqrt{\frac{1}{3}}\right], \text{ and } \left[\sqrt{\frac{1}{3}}, \sqrt{\frac{2}{3}}\right].
\end{equation}
Thus, by Lemma \ref{rationalcase} the bands of $H^B$ on each Hill band are touching as well, see Fig. \ref{Fig:touchingbands}. Bands belonging to different Hill bands do, as a rule for $\Phi \in (0,2\pi)$, not touch by Lemma \ref{operatornorm}. \\
In the case of $\frac{\Phi}{2\pi}= \frac{1}{3}$ however, only the bands at the Dirac points touch, see also Fig. \ref{Fig:notouchingbands}. The touching at the Dirac points is always satisfied by Lemma \ref{simpleproperties}.
\end{rem}
\begin{rem}
The spectrum of $Q_{\Lambda}(\Phi)$ for rational $\frac{\Phi}{2\pi}=\frac{p}{q}$ are precisely the eigenvalues \cite{HKL} of 
\begin{equation}
\frac{1}{3}\left(\begin{matrix} 0 & \operatorname{id}_{\mathbb{C}^{q}}+e^{ik_1} J_{p,q} +e^{ik_2} K_q \\
\operatorname{id}_{\mathbb{C}^{q}}+e^{-ik_1}J_{p,q}^* + e^{-i k_2} K_{q}^*  & 0 \end{matrix}\right)
\end{equation}
for $k \in \mathbb{T}_2^*$ where $J_{p,q}:=\left(\delta_{m,n}e^{2\pi i(m-1)\frac{p}{q}}\right)_{mn}$ and $(K_q)_{mn}:=1$ if $n = (m+1) \operatorname{mod} q$ and $0$ otherwise.
\end{rem}
\begin{rem}
Similar to the Hofstadter butterflies for discrete tight binding operators, the explicit spectrum for rational flux quanta allows us to plot the spectrum of $H^B$ for different rational flux quanta in Fig. \ref{Fig:butterfly}.
\end{rem}

\begin{figure}
\centerline{\includegraphics[height=9cm]{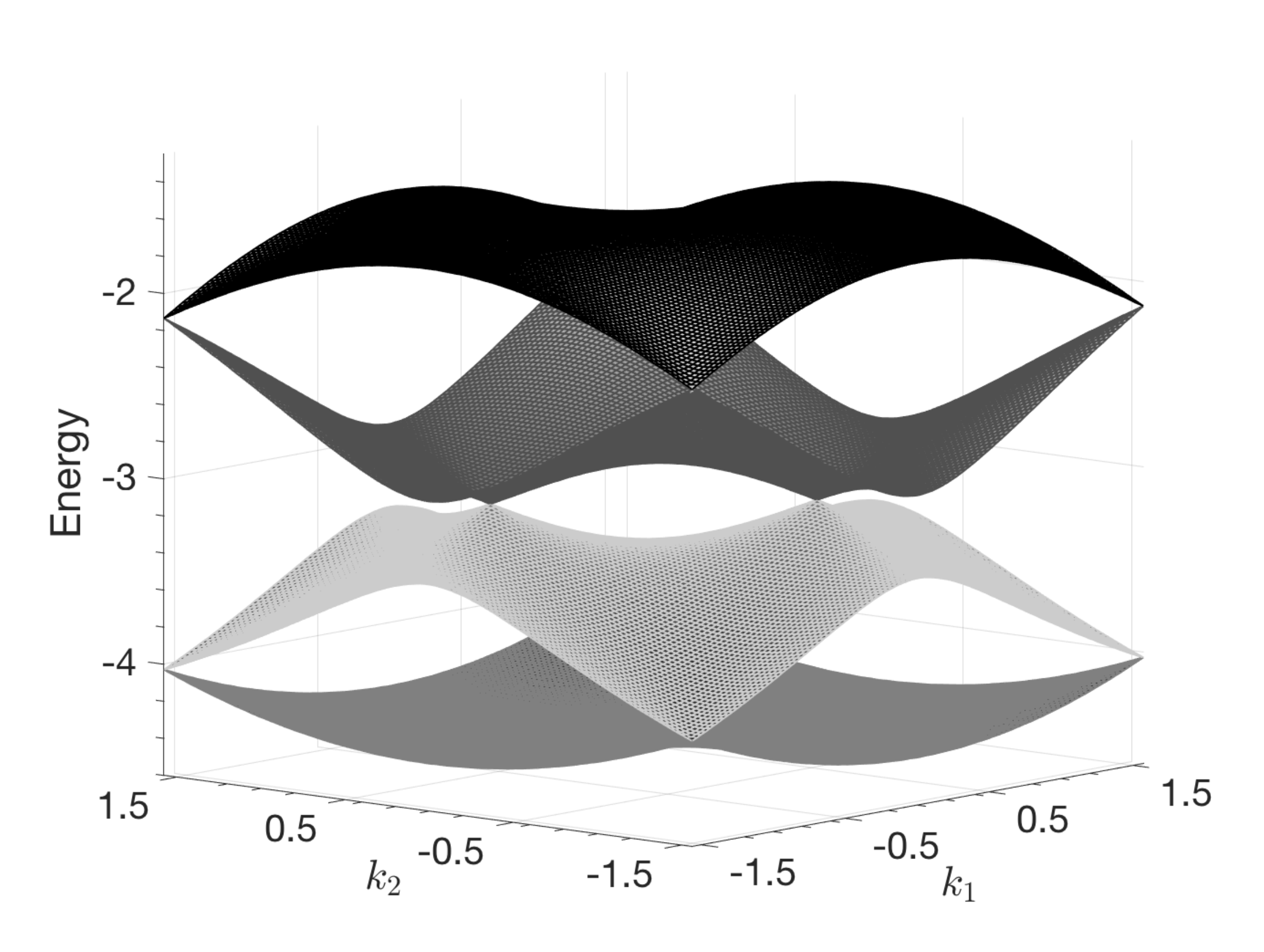}} 
\caption{Touching bands for $\frac{\Phi}{2\pi}=\frac{1}{2}$ on the first Hill band of a Schr\"odinger operator with Mathieu potential $V(t)=20\cos(2\pi t)$\label{Fig:touchingbands}. Different bands are differently colored.}
\end{figure}
\begin{figure}
\centerline{\includegraphics[height=9cm]{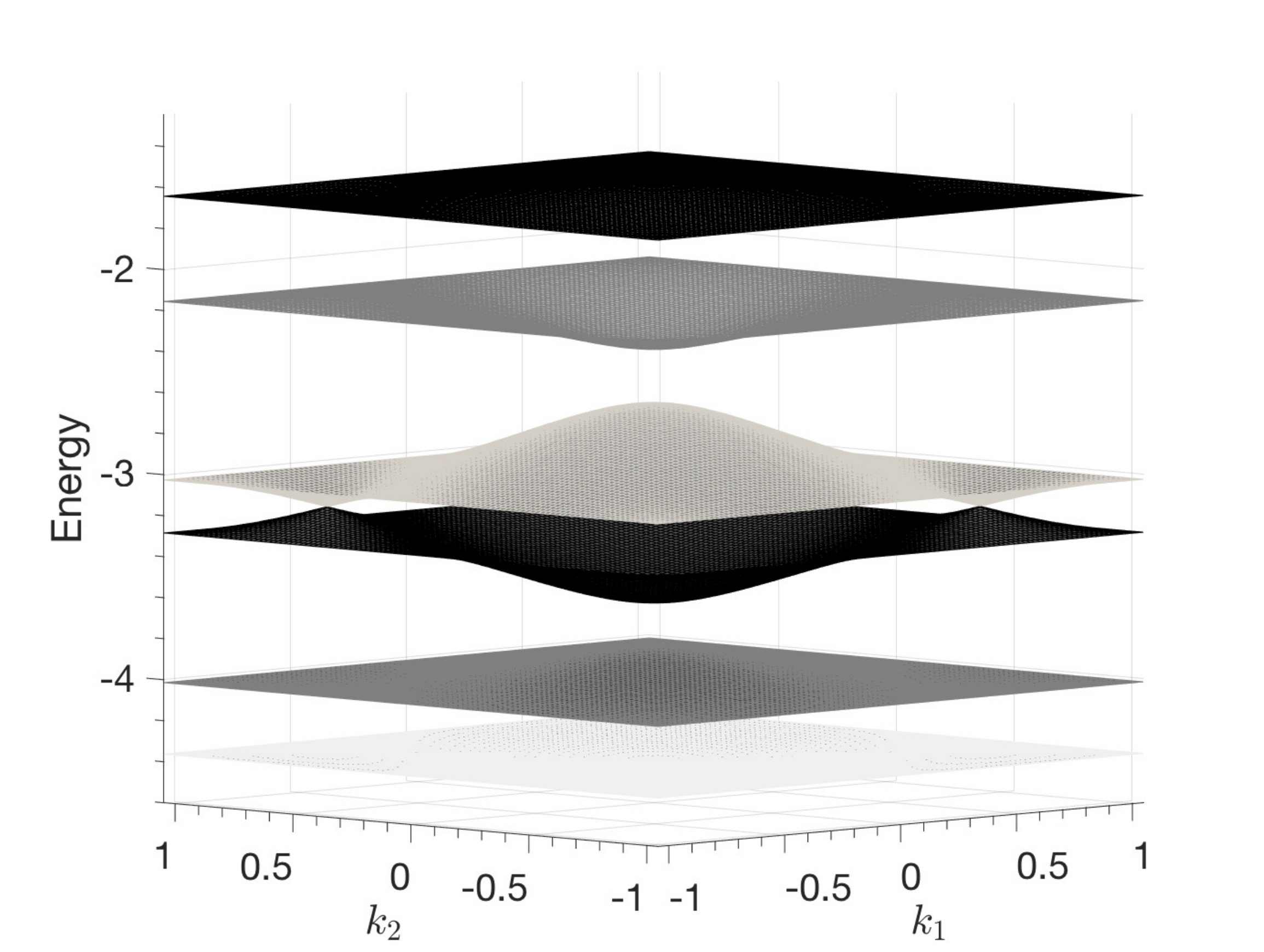}} 
\caption{Only the third and fourth band touch at the Dirac points for $\frac{\Phi}{2\pi}=\frac{1}{3}$ on the first Hill band of a Schr\"odinger operator with Mathieu potential $V(t)=20\cos(2\pi t)$.\label{Fig:notouchingbands} Different bands are differently colored.}
\end{figure}

\begin{figure}
\center\includegraphics[width=13cm]{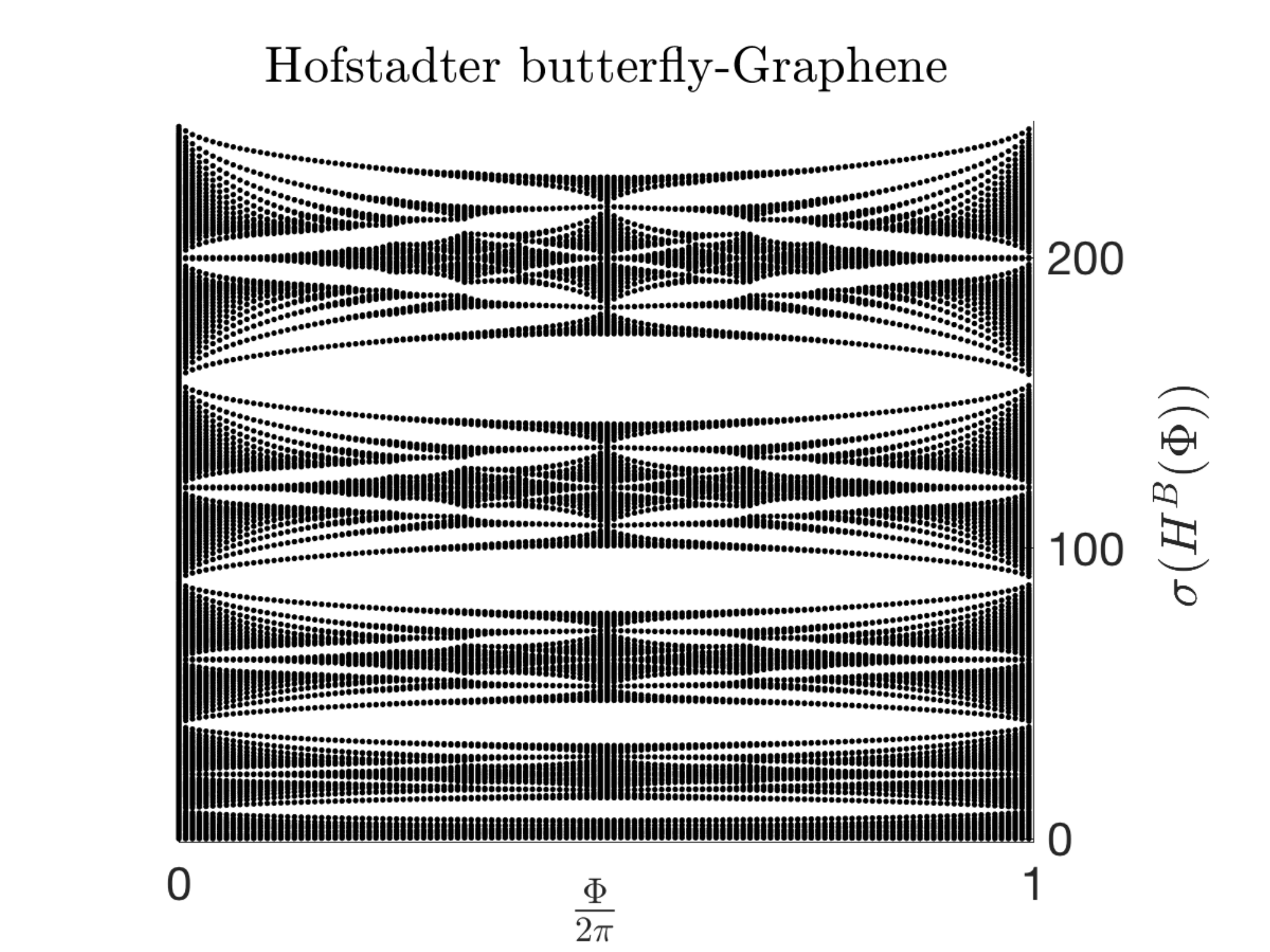} 
\caption{The Hofstadter butterfly for $H^B$ with $V=0$ on the first five Hill bands $B_k=[\pi^2(k-1)^2,\pi^2 k^2]$ for $k \in\left\{1,..,5\right\}$ and magnetic flux quanta $\frac{\Phi}{2\pi}=\frac{p}{q} \in [0,1]$ with $q\le 50$. \label{Fig:butterfly}}
\end{figure}
 
\subsection{Singular continuous Cantor spectrum for irrational flux quanta}\label{scsec}

\begin{proof}
By Lemma \ref{measure0cor}, the spectrum of $Q_{\Lambda}(\Phi)$ for irrational $\frac{\Phi}{2\pi}$ is a Cantor set of measure zero. Thus, the pullback of $\sigma(Q_{\Lambda})$ by $\Delta \vert_{\operatorname{int}(B_n)}$ is still a Cantor set of zero measure that coincides with $\sigma(H^B)\backslash \sigma(H^D).$ Therefore, the absolutely continuous spectrum of $H^B$ has to be empty.
(2) of Theorem \ref{T1} then follows from (4) of Lemma \ref{simpleproperties}.
\end{proof}


\appendix

\section{Proof of Proposition \ref{LE=0}}
The proof of this result is very similar to that for the almost Mathieu operator and the extended Harper's model. 
We will present it briefly here for completeness. Readers could refer to Theorem 3.2 (together with its proof in Appendix 2) of \cite{ajm} for a more detailed discussion.

Let $D^{\lambda}$ be defined as in \eqref{defD}, in which $v(\theta)=2\cos{2\pi \theta}$ and $c(\theta)=1+e^{-2\pi i \theta}$, hence
\begin{align}
D^{\lambda}(\theta)=\left(
\begin{matrix}
\lambda-e^{2\pi i \theta}-e^{-2\pi i \theta}\ \ &-1-e^{2\pi i (\theta-\frac{\Phi}{2\pi})}\\
1+e^{-2\pi i \theta}\ \ &0
\end{matrix}
\right).
\end{align}
Let us complexify $\theta$ and define $D^{\lambda}_{\epsilon}$ for $\epsilon\in \R$ as follows
\begin{align}
D^{\lambda}_{\epsilon}(\theta):=D^{\lambda}(\theta+i\epsilon).
\end{align}
Let 
\begin{align}
L(D_{\epsilon}^{\lambda}, \Phi):=\lim_{n\to\infty}\frac{1}{n}\int_{\T_1}\log{\|\prod_{j=n-1}^0 D^{\lambda}_{\epsilon}(\theta+j\frac{\Phi}{2\pi})\|}\ \mathrm{d}\theta,
\end{align}
be the complexified Lyapunov exponent. 
By Hardy's convexity theorem, see e.g. Theorem 1.6 in \cite{Duren}, $L(D_{\epsilon}^{\lambda}, \Phi)$ is convex in $\epsilon$.

Let 
\begin{align}
\omega(\lambda, \Phi; \epsilon):=\frac{1}{2\pi} \lim_{h\rightarrow 0_+}\frac{L(D_{\epsilon+h}^{\lambda}, \Phi)-L(D_{\epsilon}^{\lambda}, \Phi)}{h}
\end{align}
be the right-derivative of the complexified Lyapunov exponent, which has been dubbed {\it acceleration} in \cite{Global}.

By Theorem 1 of \cite{JMarxErr}, since $\det(D^{\lambda}(\theta+i\epsilon))\neq 0$ for $\epsilon\neq 0$, we have
\begin{align}\label{omegainZ}
\omega(\lambda, \Phi; \epsilon)\in \Z,\ \ \text{for}\ \epsilon\neq 0.
\end{align}
This is usually referred to as {\it quantization of acceleration}.

One can also easily compute the following asymptotic behaviour
\begin{equation}
\begin{aligned}
D^{\lambda}_{\epsilon}(\theta)=&\left(\begin{matrix}-e^{2\pi \epsilon}\ \ 0\\e^{2\pi \epsilon}\ \ \ \ 0\end{matrix}\right)+O(1),\ \ \ \epsilon\to\infty\\
D^{\lambda}_{\epsilon}(\theta)=&\left(\begin{matrix}-e^{-2\pi \epsilon}\ \ -e^{-i\Phi}e^{-2\pi\epsilon}\\0\ \ \ \ \ \ \ \ \ \ \ \ 0\end{matrix}\right)+O(1),\ \ \ \epsilon\to-\infty,
\end{aligned}
\end{equation}
hence by \eqref{omegainZ},
\begin{equation}
\begin{aligned}
\begin{cases}
L(D_{\epsilon}^{\lambda}, \Phi)=\epsilon,\ \ \ \ \ \epsilon>\epsilon_0>0,\\
L(D_{\epsilon}^{\lambda}, \Phi)=-\epsilon,\ \ \ \epsilon<-\epsilon_0.
\end{cases}
\end{aligned}
\end{equation}

Hence convexity of $L(D_{\epsilon}^{\lambda}, \Phi)$ and quantization of acceleration force either 
\begin{itemize}
\item $L(D_{0}^{\lambda}, \Phi)=0$\ \ or
\item $L(D_{0}^{\lambda}, \Phi)>0$ with $\omega(0, \Phi; \epsilon)=0$.
\end{itemize}
By Theorem 1.2 of \cite{AJSadel}, the second case is equivalent to $(\frac{\Phi}{2\pi}, D_{0}^{\lambda})$ inducing a {\it dominated splitting}.
This is equivalent to $\lambda \notin \Sigma_{\Phi}$, by \cite{MarxDom}.

Finally note that we always have
\begin{align}
L(\lambda, \Phi)= L(D_{0}^{\lambda}, \Phi)-\int_{\T_1}\log{|1+e^{-2\pi i \theta}|}\ \mathrm{d}\theta=L(D_{0}^{\lambda}, \Phi).
\end{align}
Hence $L(\lambda, \Phi)=0$ if and only if $\lambda\in \Sigma_{\Phi}$.
\qed

\section{Proof of Lemma \ref{det=tr}}
Assume $\theta=\frac{1}{2}-k_0\frac{p}{q}$.
Let $(H_{2\pi p/q, \theta})|_{[0,k-1]}$ be the restriction of $H_{2\pi p/q, \theta}$ onto interval $[0, k-1]$ with Dirichlet boundary condition.
Let $P_k(\theta)=\det{(\lambda -(H_{2\pi p/q,  \theta})|_{[0,k-1]})}$ be the determinant of this $k\times k$ matrix.
One can prove by induction (in $k$) that the following holds
\begin{align}\label{DPk}
D^{\lambda}_k(\theta)=
\left(
\begin{matrix}
P_k(\theta)\ \ &\overline{c(\theta-\frac{p}{q})}P_{k-1}(\theta+\frac{p}{q})\\
c(\theta+(k-1)\frac{p}{q})P_{k-1}(\theta)\ \ &-\overline{c(\theta-\frac{p}{q})} c(\theta+(k-1)\frac{p}{q}) P_{k-2}(\theta+\frac{p}{q})
\end{matrix}
\right).
\end{align}
Thus 
\begin{align}
\tr(D^{\lambda}_q(\theta))=P_q(\theta)+|c(\theta-\frac{p}{q})|^2 P_{q-2}(\theta+\frac{p}{q}).
\end{align}
It then suffices to note that 
\begin{align}
\begin{cases}
\tr(D^{\lambda}_q(\theta-(k_0-1)\frac{p}{q}))=\tr(D^{\lambda}_q(\theta)),\\ 
c(\theta-k_0\frac{p}{q})=0,\\
(H_{\frac{2\pi p}{q}, \theta-(k_0-1)\frac{p}{q})})|_{[0,q-1]}=M_q.
\end{cases}
\end{align}

\section{$1/2$-H\"older continuity of spectra of Jacobi matrices}
\subsubsection*{Proof of Lemma \ref{continuity}}
We will prove the following general result for quasi-periodic Jacobi matrices. Let $H_{\alpha, \theta}\in \mathcal{L}(l^2(\Z))$ be defined as
\begin{align}\label{defHgeneral}
(H_{\alpha, \theta}u)_n=c(\theta+n\alpha)u_{n+1}+\overline{c(\theta+(n-1)\alpha)}u_{n-1}+v(\theta+n\alpha)u_n.
\end{align}
Let $\sigma_{\alpha}:=\cup_{\theta\in \T_1}\sigma(H_{\alpha, \theta})$.
\begin{lemm}\label{generalcontinuity}
Let $c(\cdot), v(\cdot)\in C^1(\T_1, \C)$. There exist constants $\tilde{C}(c,v), C(c,v)>0$ such that if $\lambda \in \sigma_{\alpha}$ and $\alpha^\prime\in \T_1$ is such that $|\alpha-\alpha^\prime|<\tilde{C}(c,v)$, then there is a $\lambda^\prime\in \sigma_{\alpha^\prime}$ such that
\begin{align*}
|\lambda-\lambda^\prime| \le C(c,v)|\alpha-\ap|^{\frac{1}{2}}.
\end{align*}
\end{lemm}
Lemma \ref{continuity} follows from Lemma \ref{generalcontinuity} by taking $\Phi=2\pi \alpha$ and $\Phi^\prime=2\pi \alpha^\prime$. Lemma  \ref{generalcontinuity} is in turn the argument of \cite{AvMS} adapted to the Jacobi setting.
\subsection*{Proof of Lemma \ref{generalcontinuity}}
Let $L\geq 1$ be given. There exists $\phi_{L} \in l^2(\Z)$ and $\theta$ such that
\begin{align}
\|(H_{\alpha, \theta}-\lambda )\phi_{L}\|\leq \frac{1}{L} \|\phi_{L}\|.
\end{align}
Let $\eta_{j, L}$ be the test function centered at $j$,
\begin{align*}
\eta_{j, L}(n)=
\begin{cases}
(1-|n-j|/L),\ \ \ &|n-j|\leq L,\\
0, &|n-j|\geq L.
\end{cases}
\end{align*}
Then for large $L$,
\begin{align}\label{sumeta}
\sum_j (\eta_{j,L}(n))^2=1+\frac{(L-1)(2L-1)}{3L}\equiv a_L.
\end{align}
is independent of $n$. Clearly,
\begin{align}\label{C1}
\sum_j \|\eta_{j,L}(H_{\alpha, \theta}-\lambda )\phi_{L}\|^2=a_L \|(H_{\alpha, \theta}-\lambda)\phi_{L}\|^2\leq \frac{a_L}{L^2} \|\phi_{L}\|^2=\frac{1}{L^2} \sum_j \|\eta_{j,L}\phi_{L}\|^2.
\end{align}
Since $\|u+v\|^2\leq 2\|v\|^2+2 \|u\|^2$, by (\ref{C1}), we get
\begin{align}\label{C2}
\sum_j \|(H_{\alpha, \theta}-\lambda )\eta_{j,L}\phi_{L}\|^2\leq  &2 \sum_j \|\eta_{j,L}(H_{\alpha, \theta}-\lambda)\phi_{L}\|^2+2\sum_j \|[\eta_{j,L}, H_{\alpha, \theta}]\phi_{L}\|^2 \nonumber \\
\leq &\frac{2}{L^2} \sum_j \|\eta_{j,L}\phi_{L}\|^2+2\sum_j \|[\eta_{j,L}, H_{\alpha, \theta}]\phi_{L}\|^2,
\end{align}
where $[\eta_{j,L}, H_{\alpha, \theta}]=\eta_{j,L}H_{\alpha, \theta}-H_{\alpha, \theta}\eta_{j,L}$ is the commutator.
Note that
\begin{align*}
([\eta_{j,l}, H_{\alpha, \theta}]\phi)_n=c(\theta+n\alpha)(\eta_{j,L}(n)-&\eta_{j,L}(n+1))\phi_{n+1}
\\ +&\overline{c(\theta+(n-1)\alpha)}(\eta_{j,L}(n)-\eta_{j,L}(n-1))\phi_{n-1},
\end{align*}
which implies
\begin{align*}
\sum_{j}\|[\eta_{j,L}, H_{\alpha, \theta}]\phi_{L}\|^2\leq \frac{8\|c\|_{\infty}^2}{L}  \|\phi_{L}\|^2\leq \frac{8\|c\|_{\infty}^2}{La_L} \sum_j \|\eta_{j,L}\phi_{L}\|^2.
\end{align*}
Combining this with (\ref{C2}) and taking into account that $a_L\sim\frac{2}{3}L$, we get
\begin{align*}
\sum_j \|(H_{\alpha, \theta}-\lambda )\eta_{j,L}\phi_{L}\|^2\leq \frac{2+25\|c\|_{\infty}^2}{L^2}  \sum_j \|\eta_{j,L}\phi_{L}\|^2,
\end{align*}
for $L>L_0$. 
Hence for certain $j$, $\eta_{j,L}\phi_{L}\neq 0$ and 
\begin{align}\label{C3}
\|(H_{\alpha, \theta}-\lambda )\eta_{j,L}\phi_{L}\| \leq \frac{(2+25\|c\|_{\infty}^2)^{\frac{1}{2}}}{L} \|\eta_{j,L}\phi_{L}\|^2.
\end{align}

Given $\ap$ near $\alpha$, choose $\theta^\prime$ such that
\begin{align*}
\theta+j\alpha=\theta^\prime +j\ap.
\end{align*}
Then on $\supp (\eta_{j,L }\phi_{\epsilon})$, 
\begin{align}\label{C4}
|f(\theta+n \alpha)-f(\theta^\prime +n\ap)|\leq L\|f^\prime\|_{\infty} |\alpha-\ap|,
\end{align}
holds for $f=c, v$. 
Thus, by (\ref{C3}) and (\ref{C4}),
\begin{align*}
\|(H_{\ap, \theta^\prime}-\lambda )\eta_{j,L}\phi_{L} \|\leq C_1(c,v) \|\eta_{j,L} \phi_{L}\|,
\end{align*}
where 
\begin{align*}
C_1(c,v)=\frac{(2+25\|c\|_{\infty}^2)^{\frac{1}{2}}}{L}+(6\|c^\prime\|_{\infty}^2+3\|v^\prime\|_{\infty}^2)^{\frac{1}{2}}L|\alpha-\ap|.
\end{align*}
Finally, taking
\begin{align*}
L=C_2(c,v)|\alpha-\ap|^{-\frac{1}{2}}>L_0,
\end{align*}
we get
\begin{align*}
\|(H_{\ap, \theta^\prime}-\lambda )\eta_{j,L}\phi_{L} \| \leq C(c,v) |\alpha-\ap|^{\frac{1}{2}} \|\eta_{j,L}\phi_{L} \|.
\end{align*}
\qed

\smallsection{Acknowledgements} 
This research was partially supported by the NSF DMS1401204. Support by the EPSRC grant EP/L016516/1 for the University of Cambridge CDT, the CCA is gratefully acknowledged (S.B.).

%


\end{document}